\pdfoutput=1  
\documentclass[11pt,reqno]{amsart}
\usepackage[utf8]{inputenc}
\usepackage{amsmath,amssymb,amsthm}
\usepackage{booktabs}
\usepackage{array}   
\usepackage{longtable}  
\usepackage[table]{xcolor}
\definecolor{hproved}{HTML}{2A78D6}   
\definecolor{hverif}{HTML}{EB6834}    
\definecolor{hband}{HTML}{F2F1EC}     
\newcommand{\stproved}{\textcolor{hproved}{\small proved}}
\newcommand{\stverif}{\textcolor{hverif}{\small verified}}
\definecolor{hext}{HTML}{6B6560}     
\newcommand{\stext}{\textcolor{hext}{\small external}}
\definecolor{hrule}{HTML}{B9B7AE}     
\newcommand{\keybox}[1]{%
  \par\medskip\noindent
  \fcolorbox{hrule}{hband}{%
    \parbox{\dimexpr\linewidth-2\fboxsep-2\fboxrule\relax}{\smallskip #1\smallskip}}%
  \par\medskip}
\usepackage{graphicx}
\usepackage{tikz}
\usetikzlibrary{calc,arrows.meta}
\usepackage[margin=3.1cm]{geometry}
\usepackage[colorlinks=true,linkcolor=blue!55!black,citecolor=blue!55!black,urlcolor=blue!55!black]{hyperref}
\hypersetup{
  pdftitle={Schur polynomials twisted by roots of unity and reciprocal pairs: exactly three factors, and where they vanish},
  pdfauthor={Carles Mar\'{\i}n},
  pdfsubject={Schur polynomials evaluated at a full orbit of roots of unity together with free reciprocal pairs},
  pdfkeywords={Schur polynomials; roots of unity; reciprocal variables; factorization; cores and quotients; signed enumeration}}
\usepackage{caption}
\makeatletter
\renewcommand\section{\@startsection{section}{1}%
  \z@{\linespacing\@plus\linespacing}{.5\linespacing}%
  {\normalfont\large\bfseries\centering}}
\makeatother

\theoremstyle{plain}
\newtheorem{theorem}{Theorem}[section]
\newtheorem{lemma}[theorem]{Lemma}
\newtheorem{corollary}[theorem]{Corollary}
\newtheorem{proposition}[theorem]{Proposition}
\newtheorem{conjecture}[theorem]{Conjecture}
\theoremstyle{definition}
\newtheorem{remark}[theorem]{Remark}
\newtheorem{definition}[theorem]{Definition}
\newtheorem{example}[theorem]{Example}
\newtheorem{problem}[theorem]{Problem}

\newcommand{\ZZ}{\mathbb{Z}}
\newcommand{\CC}{\mathbb{C}}
\newcommand{\PP}{\mathcal{P}}
\newcommand{\zt}{\mu_t}
\newcommand{\inv}{\operatorname{inv}}
\newcommand{\sgn}{\operatorname{sgn}}
\newcommand{\core}{\operatorname{core}}
\newcommand{\quot}{\operatorname{quot}}
\newcommand{\zb}{\bar z}
\newcommand{\Fix}{\operatorname{Fix}}
\newcommand{\wt}{\operatorname{wt}}

\ifdefined\LONGABS
  \makeatletter
  \def\@setabstracta{%
    \ifvoid\abstractbox
    \else
      \skip@20\p@ \advance\skip@-\lastskip
      \advance\skip@-\baselineskip \vskip\skip@
      \unvbox\abstractbox
    \fi}
  \makeatother
\fi

\begin{document}

\title[Schur polynomials twisted by roots of unity and reciprocal pairs]
{Schur polynomials twisted by roots of unity and reciprocal pairs: exactly three factors, and where they vanish}

\author{Carles Mar\'in}
\address{Independent researcher}
\email{karlesmarin@gmail.com}

\date{August 21, 2026}


\begin{abstract}
\ifdefined\LONGABS
Let $\zt=\{1,\zeta,\dots,\zeta^{t-1}\}$ be the full set of $t$-th roots of unity and let
$(z,z^{-1})$ be a free reciprocal pair. We evaluate $s_\lambda(\zt,z,z^{-1})$ for every $t\ge2$ and
every partition $\lambda$ with at most $t+2$ parts, and no hypothesis on its shape: it is a signed product of
exactly three factors over a fixed denominator, or zero, and all three arguments are read off the
residue profile and the
$t$-quotient of $\lambda$. Equivalently, and with no exponential in sight, it is a ratio of
$\mathfrak{sl}_2$ characters: a triple product over the square of one. What the identity exposes is a
structure, and the structure is the part we expect to outlast it: everything this evaluation can see
of $\lambda$ is a residue profile, an interval triple and a sign, and we isolate that datum as an
\emph{evaluation invariant}. The value factors through it; it is complete, two partitions of
any sizes that share the invariant sharing the value; and, read as a multiset together with the
sign, it is also minimal, two partitions sharing a nonzero value only if they share it. So an
arbitrarily large partition is compressed to three integers and a sign, exactly. The proof is a Laplace
expansion along the $t$ frozen rows of the
bialternant together with one cancellation lemma in the symmetric group, and it delivers the sign,
which is the sorting sign already present in Littlewood's evaluation of $s_\lambda(\zt)$ and which we
also give in closed form in the notation of Ayyer and Kumari. Three
consequences are recorded. First, a vanishing criterion: the character vanishes exactly when some
residue class modulo $t$ is empty, or when two distinguished classes are concentric as intervals ---
and the second case occurs only for $t$ even, so for odd $t$ an empty class is the only way to
vanish. Second, an extension of a recent independence criterion of Ayyer and Kumari: their theorem
determines when a character is independent of the variables it is twisted by, and we show that for
two-row shapes on the reciprocal locus the criterion acquires exactly one extra family, which we classify by its core
and quotient. Their theorem is stated for free variables and a reciprocal pair is not free, so the
locus lies outside its hypotheses rather than against them, and what we determine is what the
criterion becomes there. Third, an
enumerative reading: the
identity is a product formula for a root-of-unity weighted count of semistandard tableaux in which a
parameter is still free, and at $t=2$ this is a $(-1)$-enumeration of plane partitions in a box
refined by that parameter; a sign-reversing proof of it has to carry that parameter, and we supply
the half of the cancellation that crosses the splitting of the alphabet. Finally we show the
factorization is isolated: it fails under each of the
four deformations of the alphabet --- more reciprocal pairs, more orbit variables, the orbit replaced
by a coset, the reciprocal pair replaced by a free pair --- and a single reading of the proof
accounts for all four failures at once.

What survives the first of those deformations is the \emph{zero locus}, and the last section is about
it. Write $\Psi_r=s_\lambda(1,-1,z_1,z_1^{-1},\dots,z_r,z_r^{-1})$. Two conditions make $\Psi_r$
vanish: the beta set of $\lambda$ having constant parity, and $\lambda$ being self-complementary of
odd width. \emph{That} implication we prove for every $r$ and every $\lambda$, and we claim nothing
for it: it is a short corollary of the complementation identity over an index family Ayyer and
Behrend already single out, and at the same alphabet \emph{without} the letter $-1$ their theorems
make those very shapes factorize instead --- the determinant of the alphabet is what decides.
The new content is the \emph{converse}, that nothing else vanishes, and we prove it for every $r$
and every $\lambda$ by an extremal argument in the degree filtration whose only external input is a
rigidity theorem for products of Schur polynomials. Read on the group rather than on the alphabet,
that criterion classifies exactly the irreducible polynomial $GL(N)$-modules whose restriction to
$O(N)$ is stable under twisting by the determinant. The same argument settles a second line of the
$(t,r)$ plane and settles it outright, with no external input at all: for every \emph{odd} $t$ and
every $r$, the character vanishes if and only if a residue class is absent from the beta set, so the
concentric branch is a phenomenon of even $t$ alone. That leaves $t\ge4$ even and $r\ge2$ as the one
region where the zero locus is open, and there we record the necessary condition the argument does
give. One pair also turns out not to be typical:
there $\Psi_1$ is a genuine character up to sign, so a single order-two specialization decides
whether it vanishes, and from two pairs on that stops being true.
\else
Let $\zt$ be the full set of $t$-th roots of unity. Adjoining $r$ free reciprocal pairs gives a
two-parameter family of alphabets; we settle three parts of it.

At $r=1$, for every $t\ge2$ and every $\lambda$ with at most $t+2$ parts,
$s_\lambda(\zt,z,z^{-1})$ is a signed product of exactly three factors over a fixed denominator, or
zero, the arguments read off core and quotient. What it sees of $\lambda$ is a multiset of
three integers and a sign, and exactly that: two partitions of any sizes share a
nonzero value if and only if they agree on that datum. The proof is a Laplace expansion along the
$t$ frozen rows with one cancellation lemma, and delivers the sign, of which Littlewood's is one factor.

At $t=2$ and every $r$, $s_\lambda(1,-1,z_1^{\pm1},\dots,z_r^{\pm1})$ vanishes exactly when the beta
set has constant parity or $\lambda$ is self-complementary of odd width; that direction is a
corollary of complementation over an index family of Ayyer and Behrend, the converse an
extremal argument in the degree filtration, modulo one rigidity theorem for Schur products.
Equivalently: exactly those $V_\lambda$ restrict to $O(N,\mathbb{C})$ $\det$-stably. At odd
$t$ and every $r$ it vanishes exactly when a residue class is absent, at no external cost. And for
every $t$ and $r$, a reflection of the beta set's \emph{excess part} with one increment hitting its
centre forces vanishing.

Three consequences of the first. A vanishing criterion: an empty residue class, or two
distinguished classes concentric as intervals, the second only for even $t$. An extension of
Ayyer-Kumari's independence criterion: on the reciprocal locus it acquires one further family,
classified by core and quotient. And at $t=2$ a $(-1)$-enumeration of plane partitions in a box
refined by a parameter that stays free. The factorization is isolated: it fails under each of four
deformations of the alphabet, for one reason.
\fi
\end{abstract}

\maketitle

\section{Introduction}

A Schur polynomial is a character. Written $s_\lambda(x_1,\dots,x_n)$, it is the trace of a matrix
with eigenvalues $x_1,\dots,x_n$ acting on the irreducible representation of $GL_n$ labelled by
$\lambda$, and reading it that way turns every choice of the variables into a question about a group
element rather than about a polynomial. Two such questions are as old as the subject. At the
identity, all $x_i=1$, the trace is a dimension, and it grows with $\lambda$. At an element of finite
order it stops growing: if the variables run over a full orbit of $t$-th roots of unity the value is
$0$ or $\pm1$, and which of the three it is depends on a divisibility property of $\lambda$ and not
on its size. Something that could have been arbitrarily large has been compressed to almost nothing,
and the reason is arithmetic. Everything below is about how far that compression reaches: how much
of an alphabet one may add before the collapse stops, and what exactly survives of a partition when
it does not.

Two structural reasons for such a collapse are known, and each has produced its own literature. The
first is that the alphabet is an \emph{orbit}: a full set of $t$-th roots of unity is the eigenvalue
list of an element of order $t$, and orbits are what the operators of that line are built to see. The
second is that the alphabet is \emph{closed under $x\mapsto x^{-1}$}: such a list is the spectrum of
an orthogonal matrix, so the $GL$ character is being restricted to a smaller group, and it is the
restriction that factorizes. The two reasons are independent of one another, and so are the two
literatures; the paragraphs that follow trace each, and \S\ref{sec:sharp} shows that neither survives
being pushed into the other's territory. The alphabet of this paper is the smallest one that carries
one reason of each kind.

Figure \ref{fig:thread} is the paper read as a single chain: each answer is what raises the
next question, and each pearl carries the section that settles it. Figure \ref{fig:alphabet}
draws the alphabet itself, which is the one object everything below is about.

\begin{figure}[!ht]
\centering
\includegraphics[width=\textwidth]{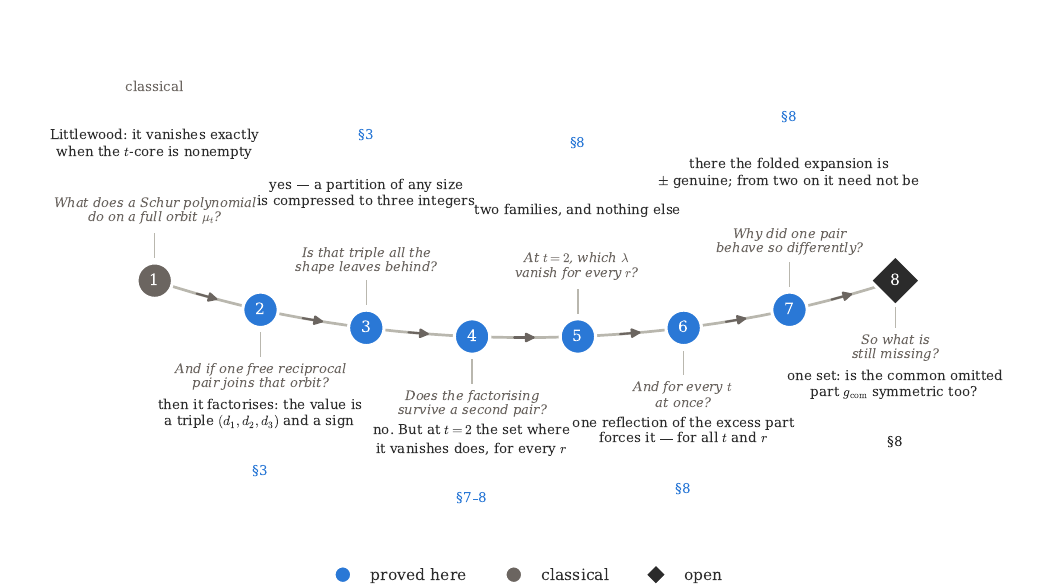}
\caption{The paper as one thread. Each question is answered by the pearl below it, and each
answer is what makes the next question the natural one to ask: the orbit invites the free pair,
the factorisation invites the triple, the triple invites the second pair, and the second pair
turns a formula into a locus. Colour marks what is classical, what is proved here, and what is
left open, and the tag under each answer is the section that settles it. Drawn by
\texttt{anc/fig\_intro.py}.}
\label{fig:thread}
\end{figure}

\begin{figure}[!ht]
\centering
\includegraphics[width=\textwidth]{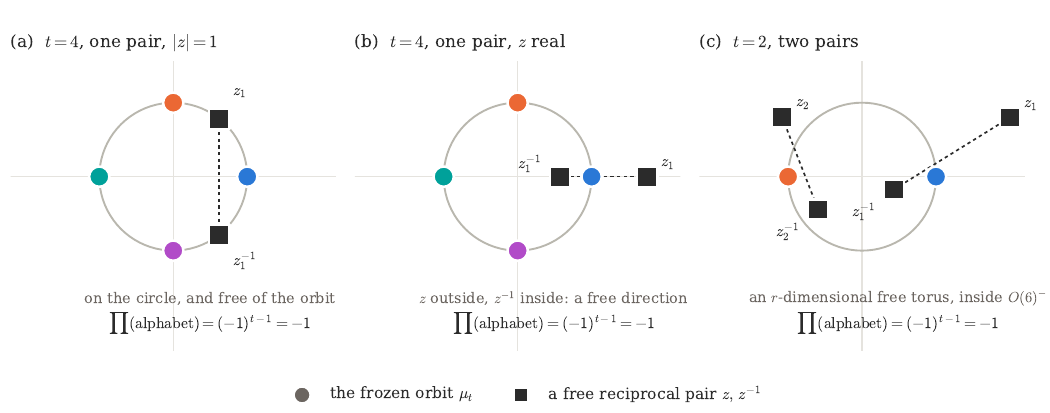}
\caption{The alphabet, in the complex plane. The orbit $\mu_t$ is frozen --- a regular
$t$-gon on the unit circle, coloured by residue --- and each reciprocal pair is one free complex
direction, drawn joined by a dashed segment. \emph{(a)} the pair may sit on the circle; it is a free
variable and not part of the orbit, and for all but finitely many $z$ it misses $\mu_t$, which is why
an operator built for $\zeta$-orbits has nothing to say about it. \emph{(b)} off the circle, $z$
outside and $z^{-1}$ inside. \emph{(c)} $r$ pairs give an $r$-dimensional free torus. In every case
the alphabet is closed under inversion, so it is the eigenvalue list of an element of $O(N,\CC)$
--- panel (b) shows why the group has to be the complex one --- and the product of its letters is
$\zeta^{t(t-1)/2}=(-1)^{t-1}$: for even $t$ that matrix lies in the non-identity component.
Computed by \texttt{anc/fig\_intro.py}.}
\label{fig:alphabet}
\end{figure}

A partition can be arbitrarily large, and an alphabet of $N=t+2$ letters can only see so much
of it. This paper settles, for one alphabet, exactly how much. The full set of $t$-th roots of
unity together with one free reciprocal pair sees a multiset of three integers and a sign, and
nothing else whatever: two partitions of any sizes that agree on that datum take the same value
there. The closed form of Theorem \ref{thm:main} is how we prove it, and the compression is the
part we expect to outlast the formula.

The evaluation of Schur polynomials at roots of unity is classical and remarkably clean. Littlewood
and Richardson \cite{LR34} showed that $s_\lambda$ specializes to $0$ or $\pm1$ when the variables
are taken to be all the $t$-th roots of unity, the sign being a sorting sign and the vanishing being
governed by the $t$-core of $\lambda$; and in the same work they generalized this to an alphabet
carrying a free variable in every letter. The latter result was rediscovered by Prasad \cite{Pra16},
given a new proof by Karmakar \cite{Kar22}, extended to all classical groups by Ayyer and Kumari
\cite{AK22} --- who also identified the earlier, almost unnoticed work of Lecouvey \cite{Lec09} ---
and carried to the universal characters by Albion \cite{Alb23}. Kumari \cite{Kum24} extended the
family further, Karmakar \cite{Kar24} treated elements of order two, and Ayyer and Kumari
\cite{AK25} added several more specializations, and Albion \cite{Alb25} related the factorizations to
$z$-asymmetric partitions and plethysm. Sathish Kumar \cite{SK23} carried the factorization to
flagged skew shapes, where it produces a family of Demazure characters; the flags there are
required to be multiples of $t$, which is the same demand for root-of-unity structure in another
guise. Read together, that line does more than accumulate cases. The statement migrates from a group
to an alphabet: once the characters are universal, evaluating one is a question about which letters
are written down, and the classical types stop being separate theorems --- \cite[Theorem 5.3]{AK25}
is one statement for $s$, $sp$ and $o$ at once. That migration is what makes the question of this
paper askable, because the alphabet below belongs to no group in the list; and the tools we borrow in
\S\ref{sec:zeros} are the ones that migration produced. We refer to Albion \cite{Alb23} for an
account of the whole line.

Read for its \emph{value}, Littlewood's theorem is a factorization. Read for its \emph{zeros} it is
something else --- a classification of the $\lambda$ that die at an element of order $t$ --- and
that second reading has a descendant of its own, which is worth naming here because
\S\ref{sec:zeros} turns out to be its cousin. Prasad \cite[Theorem 2]{Pra16} asks the vanishing
question not at a point but along a whole family: on the non-identity component of
$GL_m^n\rtimes\ZZ/n$ that paper determines when the character is identically zero \emph{as a function of the
free parameter}, and the answer is again read off the residue classes of $\lambda+\delta$ --- each
class must be represented exactly $m$ times. At $m=1$ that is Littlewood's condition. So one 1934
theorem has two descendants, and they part company over what is allowed to stay free: there,
complete root-of-unity orbits; here, as we shall see, a single frozen orbit together with free
\emph{reciprocal} directions. Neither statement contains the other.

A second line of factorization results concerns alphabets closed under $x\mapsto x^{-1}$. Okada
\cite{Oka98} and Ciucu--Krattenthaler \cite{CK09} treated rectangular shapes; Behrend, Fischer and
Konvalinka, and Ayyer, Behrend and Fischer, treated double staircases, for which we follow the
attributions given in \cite{AB19}; and Ayyer--Behrend
\cite{AB19} treated self-dual shapes, with bijective proofs of the latter given by Ayyer and Fischer
\cite{AF20}. There the payoff is enumerative: the factorizations produce
product formulas for symmetry classes of plane partitions and for rhombus tilings. The
self-complementary shapes of \cite{AB19} reappear in \S\ref{sec:zeros} from the other side: adding
the letter $-1$ turns the determinant of the alphabet from $+1$ to $-1$, and the shapes that
factorize there are exactly the shapes that vanish here.

In the factorisation literature surveyed here the two lines do not appear to have shared an
alphabet, and the obstruction is structural in both directions.
The engine of the first is an operator --- the Verschiebung $\varphi_t$, in Albion's formulation
\cite{Alb23}, acting on universal characters by $\varphi_t h_r=h_{r/t}$ if $t\mid r$ and $0$
otherwise --- which requires the alphabet to be a union of $\zeta$-orbits; every extension in that
line therefore adds \emph{more} root-of-unity structure, never less. The engine of the second
requires the alphabet to be entirely reciprocal, or the shape to be self-complementary. A free
reciprocal pair is not a $\zeta$-orbit, and a root-of-unity orbit is not a union of free pairs.

The word \emph{smallest} below can be made precise, and inversion-closure is what makes it so. An
inversion-closed extension of $\zt$ can adjoin only reciprocal pairs $\{x,x^{-1}\}$ and the fixed
points of inversion, which are $1$ and $-1$; and $1$ lies in $\zt$ for every $t$, while $-1$ lies in
$\zt$ exactly when $t$ is even. The only extension smaller than a pair is therefore $\zt\cup\{-1\}$
with $t$ odd, and it is frozen: over $t=3,5,7$ and $|\lambda|\le14$, all $1092$ shapes,
$s_\lambda(\zt,-1)$ takes only the values $0$ and $\pm1$, with nothing left to depend on. A free
reciprocal pair is the smallest extension that leaves anything to evaluate.

In this paper we take the smallest alphabet containing one letter of each kind,
\begin{equation}\label{eq:object}
\Phi_t(\lambda;z)\;:=\;s_\lambda\bigl(\,\underbrace{1,\zeta,\dots,\zeta^{t-1}}_{\zt},\;z,\;z^{-1}\bigr),
\qquad \zeta=e^{2\pi i/t},
\end{equation}
a Schur polynomial in $N=t+2$ variables.

Described that way \eqref{eq:object} is a compromise between two lines, which is how we found it. It
is better described as one object, and that description is what returns in every section below:
\eqref{eq:object} is \emph{a torus element of $O(N)$ with a single free direction}. It is closed
under inversion, so it already lives where the second line lives; its determinant is $(-1)^{t+1}$,
and that determinant is what separates factorizing from vanishing in \S\ref{sec:zeros}; the
$\zeta$-orbit is what places it in the non-identity component when $t$ is even, where a character is
a twining character, which is Remark \ref{rem:twining}; and the one free pair is the one variable
that survives the signed enumeration of \S\ref{sec:enum}. The three results of this paper are three
readings of that sentence.

Our main result, Theorem \ref{thm:main}, evaluates it for
every $t$ and every $\lambda\in\PP_N$ --- every partition with at most $N$ parts, with no
restriction whatever on its shape or its size. Write $\beta=\beta(\lambda,N)$ for the shifted partition and let
$n_i(\lambda)$ be the number of parts of $\beta$ congruent to $i$ modulo $t$, as in \cite{AK22}. Since
there are $t$ residue classes and $N=t+2$ parts, $\sum_i(n_i-1)=2$; so either some class is empty,
or, every class being occupied, the excess above one each is exactly two. Only three profiles
occur: some $n_i$ vanishes, or two classes have size two, or one class has size three. In the first
case $\Phi_t=0$; in the other two, if $A=\{a_1>a_2\}$ and $B=\{b_1>b_2\}$ denote the two
distinguished classes and
\[
d_1=a_1-a_2,\qquad d_2=b_1-b_2,\qquad
\tilde d_3=a_1+a_2-b_1-b_2,\qquad d_3=|\tilde d_3| ,
\]
then $\Phi_t=0$ if the \emph{oriented coupling} $\tilde d_3$ vanishes, and otherwise, with
$z=e^{\theta}$,
\[
\boxed{\;\;\Phi_t(\lambda;z)\;=\;\varepsilon_\lambda\,
\frac{\sinh(d_1\theta/2)\,\sinh(d_2\theta/2)\,\sinh(d_3\theta/2)}
{\sinh^2(t\theta/2)\,\sinh\theta}\;\;}
\]
with $\varepsilon_\lambda=\pm1$ given explicitly. \emph{Three factors, never more, and no hypothesis
on the shape of $\lambda$.} With the reciprocal pair removed the statement degenerates to
Littlewood's, so the content is what that pair adds: exactly one coupling term, namely $d_3$. Read
with $\mathfrak{sl}_2$ characters the same statement carries no exponential at all: it is the triple
product $\chi_{d_1/2-1}\chi_{d_2/2-1}\chi_{d_3/2-1}$ over $\chi_{t/2-1}^2$, uniformly in $t$
(Corollary \ref{cor:chars}), with half-integer indices exactly when $t$ is odd.

The alphabet is one object, and so is the answer. The residue profile, the triple $(d_1,d_2,d_3)$ and
the sign are the whole of what the theorem needs, and in \S\ref{sec:main} we collect them into a
single \emph{evaluation invariant} $I_t(\lambda)$, \eqref{eq:invariant}, through which the value factors.
It is complete but not minimal: \eqref{eq:main} is symmetric in the three integers, so what is seen
is the multiset and the sign, which is the compression announced above. What is \emph{not} seen is
then a well-posed question. Section \ref{sec:fibrelattice} answers it for the three integers ---
the partitions sharing them form a lattice, and we write its generating function down --- and what
the sign adds to that is Problem \ref{prob:fibres}.

The three integers have a reading that makes the formula transparent, and we state it here because it
is how we think about the result. Regard $A$ and $B$ as intervals $[a_2,a_1]$ and $[b_2,b_1]$ on the
beta line. Then
\begin{equation}\label{eq:intervals}
d_1=|A|,\qquad d_2=|B|,\qquad d_3=2\,\bigl|\mathrm{centre}(A)-\mathrm{centre}(B)\bigr| :
\end{equation}
\emph{the two lengths, and twice the distance between the centres} (Figure \ref{fig:intervals}). The
vanishing criterion becomes geometric --- a residue class is empty, or the two intervals are
concentric --- and it explains at once why the size-three profile, where the intervals share an
endpoint, never vanishes. Concentric intervals turn out to require $t$ even (Proposition
\ref{rem:lattice}), so for odd $t$ an empty residue class is the only way to vanish at all --- and
that stays true for every number of reciprocal pairs, which is Corollary \ref{cor:oddgen}.
Equivalently, in the language of the first line, $d_1$ and $d_2$ are $t$
times the $\mathfrak{sl}_2$ contents of the two distinguished quotient components, and the oriented
coupling $\tilde d_3$ is an \emph{affine} functional on the vector of quotient \emph{sizes}, the two
residues entering as a shift, with $d_3=|\tilde d_3|$
(Proposition \ref{prop:quotient}). It is not a functional on the residue vector, and hence not one
on the root lattice of type $A_{t-1}$ under the Garvan--Kim--Stanton correspondence \cite{GKS90}:
already at $t=2$ the single two-class profile carries ten different values of $d_3$ over
$|\lambda|\le16$. Which half of the core--quotient decomposition each clause of the vanishing
criterion lives in is settled in Proposition \ref{rem:lattice}.

Two consequences follow. \cite[Theorem 5.3]{AK25} determines when a universal character is
independent of the variables by which it is twisted, uniformly in the three classical types, and read
on \eqref{eq:object} it would say that
$s_\lambda(z,z^{-1},\zt)=s_\lambda(z,z^{-1})$ exactly on the $t$-cores, for $\ell(\lambda)\le2$. That
criterion is stated for free variables, and a reciprocal pair is not free; on that locus one further
family appears, and
Theorem \ref{thm:extra} classifies it by its core and quotient. In \S\ref{sec:enum} we read the
identity as an enumeration: a product formula for a root-of-unity weighted count of semistandard
tableaux in which the parameter $z$ is still free, and at $t=2$ a $(-1)$-enumeration of plane
partitions in a box refined by $z$.

The last two sections ask a different question, not about what the identity gives but about how far
it reaches. Section \ref{sec:sharp} tests four natural deformations of the alphabet, each of which
destroys the identity, and one reading of the proof accounts for all four. Then
\S\ref{sec:zeros} isolates what does survive. Adding further reciprocal pairs
costs the product formula of Theorem \ref{thm:main}, but not the \emph{zero locus}: for every number
$r$ of pairs, the character
$s_\lambda(1,-1,z_1,\bar z_1,\dots,z_r,\bar z_r)$ vanishes exactly when the beta set has constant
parity or $\lambda$ is self-complementary of odd width (Theorem \ref{conj:crit}). Read on the group
rather than on the alphabet the same statement is a classification: those are exactly the $\lambda$
whose $GL(N)$-module restricts to $O(N)$ stably under the twist by $\det$ (Corollary
\ref{cor:dettwist}). The shapes are the
self-complementary ones of \cite{AB19}, where the same alphabet without the letter $-1$ makes the
character factorize instead; the determinant of the alphabet decides which. The sufficient direction
is a short corollary of complementation and we say so; the content is the converse, and we prove it
for every $r$ and every $\lambda$, modulo one rigidity theorem for products of Schur polynomials.

It is worth drawing the map, and Figure \ref{fig:plane} is it. Adjoining $r$ reciprocal pairs to the
orbit $\zt$ gives a family in the two parameters $(t,r)$, and the zero locus is settled on three
parts of it --- one row, one column, and every odd column, which is half the plane.

\begin{figure}[!ht]
\centering
\includegraphics[width=\textwidth]{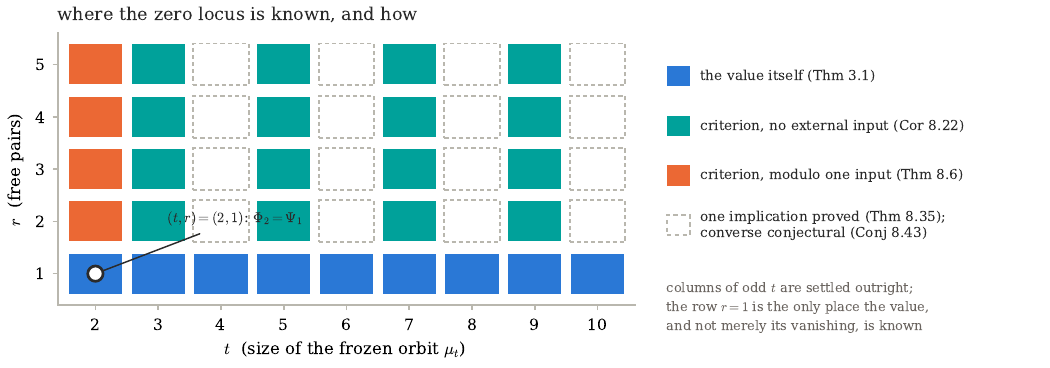}
\caption{The two-parameter family, and what is known on it. Each cell is an alphabet
$\zt\cup\{z_1^{\pm1},\dots,z_r^{\pm1}\}$; the colour is not assigned by hand but computed from the
statements themselves, one rule per line of the paper. The row $r=1$ is the only place where the
\emph{value} is known and not merely its vanishing, which is Theorem \ref{thm:main}; the columns of
odd $t$ are settled by Corollary \ref{cor:oddgen} with no external input; the column $t=2$ by
Theorem \ref{conj:crit}, which uses one. On the remaining cells Theorem \ref{thm:suff} proves one
implication and Conjecture \ref{conj:general} is the other. The marked corner is where the two
halves of the paper meet: there $\Phi_2=\Psi_1$, and it is the only alphabet both of them see.
Computed by \texttt{anc/fig\_plane.py}.}
\label{fig:plane}
\end{figure} Theorem \ref{thm:main}
is $r=1$: one pair, every $t$, and there we have the value and not merely its vanishing. Theorem
\ref{conj:crit} is $t=2$: every number of pairs, the smallest orbit. And Corollary
\ref{cor:oddgen} is \emph{every odd $t$ and every $r$} at once: there the character vanishes if and
only if a residue class is absent, so the concentric branch is a phenomenon of even $t$ alone. The
third of these is not an edge but half the plane, and it is the cheapest of the three --- it drops
out of the extremal argument of \S\ref{sec:extremal} with no external input at all, where the
second needs Theorem \ref{thm:PvW}. Off those three regions the locus is not settled, but it is no
longer untouched. Theorem \ref{thm:suff} gives a condition on the beta set --- a reflection of its
excess part, and one increment hitting the centre --- that forces the character to vanish for
\emph{every} $t$ and $r$, by an involution pairing off every term of the expansion; Corollary
\ref{cor:necgen} gives a necessary condition from the other side. That the two meet is Conjecture
\ref{conj:general}, and what separates them is a single implication, which \S\ref{sec:rigidity}
shows cannot come from the stratum that argument works in. The product formula of Theorem
\ref{thm:main} already fails at $r=2$, by (D1) --- that no such formula exists for any $r\ge2$ is
Conjecture \ref{conj:rank2} and not something we prove --- so past the line $r=1$ the locus is what
we can offer.

One thing the alphabet gives away deserves saying here rather than in passing. It has a
representation-theoretic reading --- of the kind \cite{CK09} record as missing for their
factorizations, as do \cite{AB19} and \cite{AK25} --- and the mechanism is not ours: the alphabet is closed under
inversion, hence a torus element of $O(N)$ of determinant $(-1)^{t+1}$, so for even $t$ it sits in
the non-identity component, where a character is a \emph{twining} character in the sense of Jantzen
\cite{Jantzen} and Kumar--Lusztig--Prasad \cite{KLP} and the folding is by the orbit Lie algebra.
Remark \ref{rem:twining} makes that precise and Lemma \ref{lem:AtoSp} makes it concrete --- there the
folded algebra is $C_r$ and the twining character is a symplectic character on the nose. The proof of
Theorem \ref{thm:main} uses none of it, which is why we state it as a reading and not as a method.

The article is organized as follows. We begin by fixing notation and recalling the two evaluations we
build on in Section \ref{sec:background}. In Section \ref{sec:main} we state the main theorem,
identify its three integers in terms of the core and the $t$-quotient, describe the partitions they
fail to separate, and record the vanishing criterion and the image of the resulting map. Section \ref{sec:proof} contains the proof. In Section
\ref{sec:independence} we compare the theorem with the independence criterion of \cite{AK25} and
determine the extra family it acquires on the reciprocal locus. In Section \ref{sec:enum} we read the
theorem as an enumeration and assemble, at $t=2$, the sign-reversing involution it asks for,
and in Section \ref{sec:sharp} we try four deformations of the alphabet, each of which destroys the
identity.
Section \ref{sec:zeros} determines the zero locus at $t=2$ for an arbitrary number of reciprocal
pairs --- and, at no external cost, at every odd $t$ as well --- reads that criterion on the group as
a determinant-twist rigidity, and gives a sufficient criterion, with a conjectural converse, for
arbitrary $t$ and $r$.
Section \ref{sec:verif} collects the numerical checks, Section \ref{sec:attr} separates what we use
from what we claim, statement by statement, Section \ref{sec:open} states the open problems, and \S\ref{sec:closing} closes where the
introduction began.

Two of those problems have moved since the first version, and both moves are recorded where the
problems stand rather than announced here as results. Problem \ref{prob:versch} asked whether the
two formalisms --- the Verschiebung $\varphi_t$ on the twisted half, the universal characters on the
reciprocal half --- compose. They do, and Remark \ref{rem:composes} writes the composition down; it
is Albion's Theorem 3.1 evaluated at $1$ rather than a theorem of ours, and it also corrects a
sentence of that problem which asserted more ignorance than was warranted, since the one-free-letter
operator it says we knew of nothing about is in \cite[\S6.1]{Alb23}, the very reference the problem
names. Three of them are also touched by the companion paper \cite{PaperII}, which continues
this one past $r=1$; what it moves is said where each problem stands, with the limits that paper
puts on it, and is not repeated here. And the \emph{classical} symplectic analogue of Problem \ref{prob:types} is
\cite[Theorem 2.8]{Kum24}, whose vanishing criterion we have checked against ours; the single-orbit
universal object is a different function, and that gap is isolated where the problem stands.

\section{Background and notation}\label{sec:background}

We follow the notation of \cite{AK22,AK25}. All algebraic groups here are over $\CC$; in particular
$O(N)$ means $O(N,\CC)$, which is what an alphabet closed under $x\mapsto x^{-1}$ gives when its
letters are allowed off the unit circle.

\subsection{Partitions, beta sets, cores and quotients}
A partition $\lambda=(\lambda_1,\dots,\lambda_n)$ is a weakly decreasing sequence of nonnegative
integers; $\PP_n$ denotes the set of partitions of length at most $n$, $|\lambda|$ the size and
$\ell(\lambda)$ the length. Throughout, $t\ge2$ is a fixed integer, $\zeta$ is a primitive $t$-th root
of unity, $\zt=\{1,\zeta,\dots,\zeta^{t-1}\}$, and
\[
N:=t+2 .
\]
For $\lambda\in\PP_N$ the \emph{beta set} is $\beta(\lambda)=(\beta_1,\dots,\beta_N)$ with
$\beta_j=\lambda_j+N-j$, so that $\beta_1>\beta_2>\dots>\beta_N\ge0$; we index by \emph{columns}
$j=1,\dots,N$, and note that a larger $\beta$ means a smaller column index. For $0\le i\le t-1$ we
write $n_i(\lambda)$ for the number of parts of $\beta(\lambda)$ congruent to $i$ modulo $t$, so
\begin{equation}\label{eq:excess}
\sum_{i=0}^{t-1}n_i(\lambda)=N=t+2 .
\end{equation}
We call the vector $\bigl(n_0(\lambda),\dots,n_{t-1}(\lambda)\bigr)$ the \emph{residue profile} of
$\lambda$. By \eqref{eq:excess} we have $\sum_i\bigl(n_i(\lambda)-1\bigr)=2$; if moreover every class
is occupied, that is an excess of exactly two above one each, and only two partitions of two are
available. So a profile is of exactly one of three kinds: \emph{degenerate}, if some $n_i=0$;
\emph{two-class}, if two entries equal $2$; or \emph{size-three}, if one entry equals $3$.
Everything in this paper is controlled by which of the three occurs.
Let $\sigma\in S_N$ be the permutation rearranging the parts of $\beta(\lambda)$ so that they are
grouped by residue class, the classes in increasing order and the values within each class in
decreasing order; this is the permutation of \cite[(2.1)]{AK25}. We write $\core_t(\lambda)$ and
$\quot_t(\lambda)=(\lambda^{(0)},\dots,\lambda^{(t-1)})$ for the $t$-core and $t$-quotient, defined
from $\beta(\lambda)$ as in \cite[\S2.1]{AK25}. By \cite{GKS90}, $t$-cores are in bijection with
vectors of the root lattice of type $A_{t-1}$, and the residue counts $n_i$ are the natural
coordinates.

We set $\zb:=z^{-1}$, put $z=e^{\theta}$, and abbreviate
\[
f(u):=z^{u}-z^{-u}=2\sinh(u\theta),\qquad
V:=\prod_{0\le k<k'\le t-1}\bigl(\zeta^{k'}-\zeta^{k}\bigr).
\]
For a set $S$ of columns, $b_S$ is the word of residues $\beta_j\bmod t$ for $j\in S$ read in
increasing column order, and $\inv(b_S)$ its number of inversions. We use the bialternant
\[
s_\lambda(x_1,\dots,x_N)=\frac{\det\bigl(x_i^{\beta_j}\bigr)_{i,j=1}^N}
{\det\bigl(x_i^{N-j}\bigr)_{i,j=1}^N}.
\]

\subsection{The two evaluations we build on}
We recall the two statements the present result sits between. The first is Littlewood's evaluation,
in the form of \cite[Theorem 2.2]{AK25}.

\begin{theorem}[{\cite[Theorem IX]{LR34}}]\label{thm:LR}
Let $\lambda\in\PP_t$. Then
\[
s_\lambda(1,\zeta,\dots,\zeta^{t-1})=
\begin{cases}
(-1)^{\binom{t}{2}}\sgn(\sigma) & \text{if }\core_t(\lambda)=\varnothing,\\
0 & \text{otherwise.}
\end{cases}
\]
\end{theorem}

The second is the independence criterion we extend, \cite[Theorem 5.3]{AK25}; we state the case we
need. Let $f_\lambda$ be a universal character of type $s$, $sp$ or $o$, let $Y$ be a set of $m$
variables and let $x_1,\dots,x_{n-tm}$ be free. Then
\begin{equation}\label{eq:AK53}
f_\lambda(x_1,\dots,x_{n-tm},Y,\zeta Y,\dots,\zeta^{t-1}Y)=f_\lambda(x_1,\dots,x_{n-tm})
\quad\Longleftrightarrow\quad \lambda=\core_t(\lambda),
\end{equation}
for $\lambda\in\PP_{n-tm}$. The criterion is uniform in the type --- one statement for $s$, $sp$ and
$o$ at once --- which is what makes it quotable here: our alphabet is a $GL$ alphabet, and the
question we put to it is about a specialization and not about a group. Taking $m=1$, $Y=(1)$ and
$(x_1,x_2)=(z,\zb)$ gives $n=t+2$ and our alphabet, but that substitution is precisely where the
hypothesis that $x_1,x_2$ be free is given up. So \eqref{eq:AK53} is not being applied on the
reciprocal locus; it is being asked about there, which is \S\ref{sec:independence}.

\begin{remark}\label{rem:notacase}
The nearest statement in the literature adjoins two extra letters, as we do, so it is worth recording
why it does not contain \eqref{eq:object}. Kumari \cite[Theorem 2.2]{Kum24} evaluates
$s_\lambda(X,\zeta X,\dots,\zeta^{t-1}X,\,y,\zeta y,\dots,\zeta^{m-1}y)$; with $X=(1)$ and $m=2$ the
adjoined block is $\{y,\zeta y\}$, while ours is $\{z,z^{-1}\}$. The two coincide only when
$z^{2}=\zeta^{\pm1}$, a finite set of points rather than a free variable. The same applies to the
single extra free variable of \cite[Theorem 2.7]{AK22} and to Littlewood--Richardson's Theorem XI.
\end{remark}

Figure \ref{fig:beta} carries out the whole dictionary on one shape.

\begin{figure}[!ht]
\centering
\includegraphics[width=\textwidth]{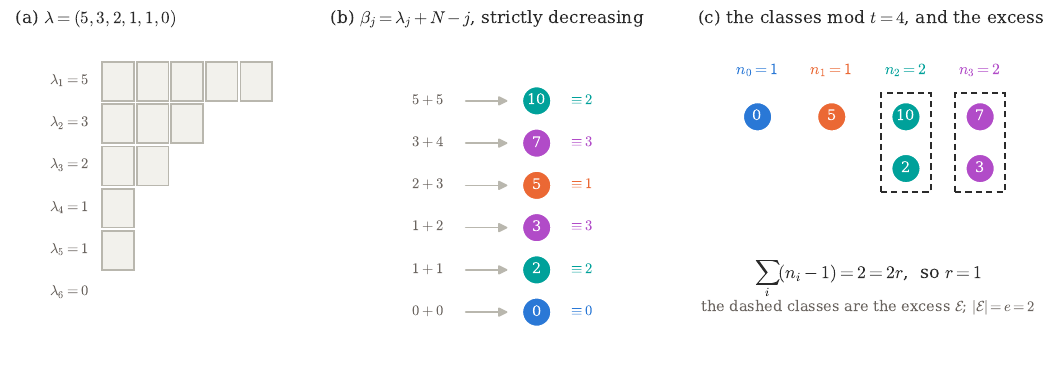}
\caption{The dictionary the rest of the paper runs on, on $\lambda=(5,3,2,1,1,0)$ at $t=4$.
\emph{(a)} the shape, padded to $N$ parts. \emph{(b)} $\beta_j=\lambda_j+N-j$, which turns a
weakly decreasing sequence into a strictly decreasing one and so into a set; each value is then
coloured by its residue modulo $t$. \emph{(c)} the residue profile $(n_0,\dots,n_{t-1})$ and the
count \eqref{eq:excess}: $\sum_i(n_i-1)=2r$ always, so when every class is occupied the classes with
$n_i\ge2$ --- the excess $\mathcal{E}$, dashed --- carry exactly $2r$ elements more than one each.
If some class is empty the signed sum is still $2r$, but it is no longer an excess; that is the
degenerate branch, and it is settled first. Everything the paper proves is a statement about that
picture. Computed by \texttt{anc/fig\_intro.py}.}
\label{fig:beta}
\end{figure}

\section{The evaluation}\label{sec:main}

\keybox{\emph{The theorem in one sentence.} The value depends on $\lambda$ through its residue
profile, an interval triple $(d_1,d_2,d_3)$ and a sign, and through nothing else: an arbitrarily
large partition is compressed to three integers and a sign.}

\begin{theorem}\label{thm:main}
Let $t\ge2$ and $\lambda\in\PP_N$, $N=t+2$.
\begin{enumerate}
\item[\rm(i)] If $n_i(\lambda)=0$ for some $i$, then $\Phi_t(\lambda;z)=0$.
\item[\rm(ii)] Otherwise, by \eqref{eq:excess}, either exactly two residue classes have $n_i=2$, or
exactly one has $n_i=3$. Let $r_A\le r_B$ be the residues involved and let
\[
A=\{a_1>a_2\},\qquad B=\{b_1>b_2\}
\]
be, in the first case, the parts of $\beta(\lambda)$ in the classes $r_A$ and $r_B$, and in the
second case the two overlapping pairs $A=\{p,q\}$, $B=\{q,r\}$ of the class $\{p>q>r\}$, so that
$r_A=r_B$. Set
\[
d_1=a_1-a_2,\qquad d_2=b_1-b_2,\qquad
\tilde d_3=a_1+a_2-b_1-b_2,\qquad d_3=|\tilde d_3| ,
\]
and call $\tilde d_3$ the \emph{oriented coupling}. If $\tilde d_3=0$ then $\Phi_t(\lambda;z)=0$.
Otherwise, with $z=e^\theta$,
\begin{equation}\label{eq:main}
\boxed{\;\;\Phi_t(\lambda;z)\;=\;\varepsilon_\lambda\cdot
\frac{\sinh(d_1\theta/2)\,\sinh(d_2\theta/2)\,\sinh(d_3\theta/2)}
{\sinh^2(t\theta/2)\,\sinh\theta}\;\;}
\end{equation}
where, writing $j_{A_1}$ and $j_{B_1}$ for the columns carrying $a_1$ and $b_1$, $S$ for the
complement of $\{j_{A_1},j_{B_1}\}$ in $\{1,\dots,N\}$, and $b_S$ for the word of residues of
$\beta(\lambda)$ on $S$ read in increasing column order,
\begin{equation}\label{eq:sign}
\varepsilon_\lambda\;=\;(-1)^{\,t+\binom{N+1}{2}}\;
(-1)^{\,j_{A_1}+j_{B_1}+\inv(b_S)}\;\sgn(a_1-b_1)\;\sgn(\tilde d_3).
\end{equation}
\end{enumerate}
Each of the four factors is $\pm1$ --- the last because $\tilde d_3\ne0$ is exactly the hypothesis
of {\rm(ii)} --- so $\varepsilon_\lambda=\pm1$. In the size-three profile $r_A=r_B$ and
$\tilde d_3=p-r>0$, so that factor is $+1$.

The left-hand side is a Laurent polynomial in $z$. We read \eqref{eq:main} as an identity in
$\CC(z)$; after cancellation the right-hand side is a Laurent polynomial too, and at the zeros of
the denominator --- $z=\pm1$ and $z^t=1$ --- we use that continuation. We do so in
\S\ref{sec:enum}, where the value is taken at $z=1$.
\end{theorem}

\noindent Splitting $d_3$ into its size and its orientation is not bookkeeping. The sign
\eqref{eq:sign} needs the orientation and the formula needs only the size, and at $\tilde d_3=0$ the
orientation does not exist: a statement carrying $\sgn(a_1+a_2-b_1-b_2)$ and asserting that it is
$\pm1$ is false on the concentric locus, where that argument is zero. The value there is $0$ by the
factor $\sinh(d_3\theta/2)$, so nothing downstream changes; what changes is that the case is now
excluded before the sign is written rather than after.

We write
\[
d(\lambda):=(d_1,d_2,d_3)
\]
and call it the \emph{interval triple} of $\lambda$, after the reading \eqref{eq:intervals}; and we
call $d_3$ the \emph{coupling term}, since it is the only one of the three that mixes the two
distinguished classes.

It is worth giving the whole datum a name, because the theorem is better read as a statement about it
than as a formula. Set
\begin{equation}\label{eq:invariant}
\boxed{\;\;I_t(\lambda)\;:=\;\begin{cases}
0 & \text{if the residue profile is degenerate, or }\tilde d_3=0,\\[2pt]
\bigl(d_1,d_2,d_3,\varepsilon_\lambda\bigr) & \text{otherwise},
\end{cases}\;\;}
\end{equation}
and call it the \emph{evaluation invariant} of $\lambda$. Its first branch is exactly the vanishing
locus of Corollary \ref{cor:zero}, and it is one branch and not two because on the whole of it there
is no sign to record: \eqref{eq:sign} needs $\tilde d_3\ne0$. Theorem \ref{thm:main} then says two
things at once. First, $\Phi_t$ \emph{factors through} $I_t$: the map
$\lambda\mapsto\Phi_t(\lambda;z)$ is
the composite of $\lambda\mapsto I_t(\lambda)$ with a function of the invariant alone. Second, that
invariant is \emph{complete} for this evaluation: if two partitions have the same $I_t$ they have the
same $\Phi_t$, whatever their sizes. Completeness is what makes the compression announced above worth
a name rather than an observation --- it says the invariant is not a convenient summary of the
dependence but the whole of it --- and it turns the natural next question, what the invariant
forgets, into a well-posed one. Proposition \ref{prop:minimal} closes the question from the other
side: once the triple is read as a multiset, nothing in $I_t$ can be dropped either. So
\S\ref{sec:fibrelattice} settles what is forgotten for the triple and Problem
\ref{prob:fibres} is what the sign leaves. The $(d_1,d_2,d_3)$ of Figure
\ref{fig:image} are the values this invariant takes, drawn.

That factorisation is one pipeline, and every arrow in it loses information that never comes back:
\begin{multline*}
\lambda\;\longrightarrow\;\beta(\lambda)\;\longrightarrow\;
\bigl(\text{residue profile},\ \text{the two distinguished intervals}\bigr)\\
\longrightarrow\;\bigl(d_1,d_2,d_3,\varepsilon_\lambda\bigr)\;\longrightarrow\;\Phi_t(\lambda;z).
\end{multline*}
The intervals have to be carried alongside the profile: the profile says which classes are
distinguished and nothing about where their beta numbers sit, so it does not determine the triple.
Coordinate by coordinate, the invariant divides the work as follows.

\begin{center}\small
\begin{tabular}{ll}
\toprule
datum of $\lambda$ & what it controls\\
\midrule
residue profile & the degenerate vanishing branch\\
the two interval lengths $d_1,d_2$ & the first two factors\\
the distance between the centres, $d_3$ & the coupling factor, and the concentric vanishing branch\\
the orientation $\sgn(\tilde d_3)$ & one factor of the sign\\
the total sign $\varepsilon_\lambda$ & the overall sign\\
\bottomrule
\end{tabular}
\end{center}

\noindent Neither vanishing branch is the other's: the profile cannot see the concentric one, and
the triple cannot see the degenerate one, which is why the criterion of Corollary \ref{cor:zero} has
two clauses and \S\ref{sec:main} spends a proposition on where each of them lives.


\enlargethispage{2\baselineskip}
\keybox{One count runs through the paper. With $t$ residue classes and $N=t+2$ parts, on the branch
where every class is occupied the excess is exactly two, so by \eqref{eq:excess} a profile can only
be degenerate, two-class or size-three:
that trichotomy is the shape of Theorem \ref{thm:main}, the case split of its proof in
\S\ref{sec:proof}, and the first branch of the zero locus in \S\ref{sec:zeros}. Two distinguished
classes are what the reciprocal pair has to couple, hence exactly one coupling term and three
factors rather than two.}

The count of three has a second explanation, which owes nothing to the proof. The excess two gives
two distinguished classes, which are two intervals on the beta line, and an interval pair is four
numbers $(a_1,a_2,b_1,b_2)$. The three arguments of \eqref{eq:main} see how that pair is shaped and
not where it sits: a common shift $\beta\mapsto\beta+m$ leaves $d_1$, $d_2$ and $\tilde d_3$ alone.
So what can enter them is a linear functional killing the common translation, and those form the
hyperplane $\sum\alpha_i=0$ of the four-dimensional dual: a space of dimension exactly \emph{three}.
The triple $d_1=(1,-1,0,0)$, $d_2=(0,0,1,-1)$ and $\tilde d_3=(1,1,-1,-1)$ is a basis of it. Modulo
the common translation an interval pair has three independent linear coordinates, and that is where
the count of three comes from. What Theorem \ref{thm:main} adds, and what the dimension count does
not force, is that in this specialization those three coordinates \emph{separate multiplicatively}:
a product of one factor each, rather than some other function of three variables.

The \emph{value} is not translation-invariant, and saying so is worth a line, because what it is
instead is a constraint on the sign. A shift $\beta\mapsto\beta+m$ is $\lambda\mapsto\lambda+(m^N)$,
and $s_{\lambda+(m^N)}(A)=\det(A)^{m}s_\lambda(A)$ for any alphabet $A$, so by \eqref{eq:object}
\begin{equation}\label{eq:shift}
\Phi_t\bigl(\lambda+(m^N);z\bigr)\;=\;(-1)^{(t+1)m}\,\Phi_t(\lambda;z).
\end{equation}
For odd $t$ the value is invariant; for even $t$ an odd shift reverses it. The triple does not move,
so the whole of that reversal is carried by $\varepsilon_\lambda$ --- which makes \eqref{eq:shift} a
check on \eqref{eq:sign} that costs nothing and that the sign passes, over the shapes of
\S\ref{sec:verif}. It is also the reason the slogan has to be stated about the triple and not about
the value: for $t$ even the value is not a function of the triple's translation class.

The sign \eqref{eq:sign} refines the sorting sign of Theorem
\ref{thm:LR}; Proposition \ref{rem:sign} below records a shorter expression for it, in the notation
of \cite{AK25}.

The right-hand side of \eqref{eq:main} is a ratio, and the content of the theorem is that the quotient is a Laurent polynomial
in $z$. For odd $t$ the individual factor $\sinh(t\theta/2)$ is not a Laurent monomial, so the three
factors do not separately lie in $\ZZ[z,z^{-1}]$; only their combination does. The statement is
uniform in $t$ nonetheless.

\begin{corollary}\label{cor:zero}
$\Phi_t(\lambda;z)\equiv0$ if and only if some residue class modulo $t$ is empty, or
$a_1+a_2=b_1+b_2$. In particular the size-three profile never vanishes.
\end{corollary}

\noindent The vanishing is identical in $z$ and not at a point: when it happens, $\lambda$
contributes nothing at any element of the one-parameter family the reciprocal pair sweeps out. That
is the form in which Prasad's question was posed in \S\ref{sec:background}, and the sense in which
\eqref{eq:object} is read outside combinatorics --- the frozen half of the alphabet being a
holonomy of finite order, and the free pair what survives of the torus.

\begin{corollary}\label{cor:chars}
Write $\chi_k(z)=(z^{k+1}-z^{-k-1})/(z-z^{-1})$. For an integer $k\ge0$ this is the character of the
$(k+1)$-dimensional irreducible $\mathfrak{sl}_2$-module. For half-integral $k$ we keep the same
formula, read in $u=z^{1/2}$; it is then notation and not a character, since the individual factors
are not Laurent in $z$. Then
\eqref{eq:main} can be written with no exponential and no convention for $\theta$ at all:
\begin{equation}\label{eq:chiratio}
\boxed{\;\;\Phi_t(\lambda;z)\;=\;\varepsilon_\lambda\,
\frac{\chi_{\frac{d_1}{2}-1}(z)\;\chi_{\frac{d_2}{2}-1}(z)\;\chi_{\frac{d_3}{2}-1}(z)}
{\chi_{\frac{t}{2}-1}(z)^{2}}\;\;}
\end{equation}
For $t=2$ the denominator is $\chi_0=1$ and every $d_i$ is even, so $\Phi_2$ is a signed product of
three $\mathfrak{sl}_2$ characters outright. For odd $t$ the indices are half-integers, which is the
statement above in another form: the three factors are not separately Laurent, and only the ratio is.
Each factor is read in $u=z^{1/2}$ and reverses sign under $u\mapsto-u$ exactly when its index is
half-integral; since $d_1+d_2+d_3$ is even --- it has the parity of
$d_1+d_2+\tilde d_3=2(a_1-b_2)$ --- the ratio does not, and it is the Laurent polynomial in $z$ of
\eqref{eq:main}.
\end{corollary}

\begin{proof}
Write $f(u)=z^{u}-z^{-u}=2\sinh(u\theta)$, so that $\chi_k=f(k+1)/f(1)$. The numerator of
\eqref{eq:main} is $f(d_1/2)f(d_2/2)f(d_3/2)/8$ and its denominator is $f(t/2)^2f(1)/8$, so
\eqref{eq:main} equals $\varepsilon_\lambda f(d_1/2)f(d_2/2)f(d_3/2)/\bigl(f(t/2)^2f(1)\bigr)$;
dividing three factors of $f(1)$ out of the numerator and two out of the denominator gives
\eqref{eq:chiratio}.
\end{proof}

\subsection{The three integers come from the quotient}
The following identifies $(d_1,d_2,d_3)$ in the language of the root-of-unity line. For a partition
$\nu$ put $k(\nu)=\nu_1-\nu_2$, its $\mathfrak{sl}_2$ content.

\begin{proposition}\label{prop:quotient}
In the two-class profile, with $\quot_t(\lambda)=(\lambda^{(0)},\dots,\lambda^{(t-1)})$,
\[
d_1=t\bigl(k(\lambda^{(r_A)})+1\bigr),\qquad
d_2=t\bigl(k(\lambda^{(r_B)})+1\bigr),\qquad
\tilde d_3=t\bigl(|\lambda^{(r_A)}|-|\lambda^{(r_B)}|\bigr)+2(r_A-r_B) .
\]
In the size-three profile the single distinguished component $\lambda^{(r_A)}$ has three parts and
\[
d_1=t\bigl(\lambda^{(r_A)}_1-\lambda^{(r_A)}_2+1\bigr),\quad
d_2=t\bigl(\lambda^{(r_A)}_2-\lambda^{(r_A)}_3+1\bigr),\quad \tilde d_3=d_1+d_2 .
\]
\end{proposition}

\begin{proof}
Write $a_i=t\tilde a_i+r_A$. The class $r_A$ carries two beta numbers, so
$\lambda^{(r_A)}=(\tilde a_1-1,\tilde a_2)$ by the definition of the quotient, whence
$k(\lambda^{(r_A)})=\tilde a_1-\tilde a_2-1$ and $d_1=t(\tilde a_1-\tilde a_2)$, which is the first
identity; the second is the same computation. For the third,
$\tilde a_1+\tilde a_2=|\lambda^{(r_A)}|+1$ and likewise for $B$, so
$a_1+a_2-b_1-b_2=t(|\lambda^{(r_A)}|-|\lambda^{(r_B)}|)+2(r_A-r_B)$. The size-three case is
identical with $\lambda^{(r_A)}=(\tilde p-2,\tilde q-1,\tilde r)$.
\end{proof}

For $t=2$ the third identity reads $\tilde d_3=2\bigl(|\lambda^{(0)}|-|\lambda^{(1)}|-1\bigr)$, which
already shows that $\tilde d_3$ is a functional on the quotient sizes and not on the residue vector, so it
is not one on the root lattice of type $A_{t-1}$ under the Garvan--Kim--Stanton correspondence
\cite{GKS90} between that lattice and the $t$-cores. The condition $r_B-r_A=t/2$, which
governs the family of Theorem \ref{thm:extra} below, is an order-two element of the relevant
quotient, and this is the source of the parity phenomena throughout the paper.

\begin{proposition}[the concentric locus]\label{rem:lattice}
In the two-class profile, ordered so that $r_A<r_B$,
\[
\boxed{\;\;d_3=0\quad\Longleftrightarrow\quad t\ \text{is even},\quad r_B-r_A=\tfrac t2,
\quad\text{and}\quad \bigl|\lambda^{(r_A)}\bigr|=\bigl|\lambda^{(r_B)}\bigr|+1\;\;}
\]
In particular, for $t$ odd $\Phi_t(\lambda;z)=0$ if and only if some residue class modulo $t$ is
empty: the concentric clause of Corollary \ref{cor:zero} is vacuous there.
\end{proposition}

\begin{proof}
Since $a_1,a_2\equiv r_A$ and $b_1,b_2\equiv r_B$ modulo $t$, the equality $a_1+a_2=b_1+b_2$ forces
$2r_A\equiv2r_B$, that is $t\mid2(r_B-r_A)$. For $t$ odd this gives $t\mid r_B-r_A$ and hence
$r_A=r_B$, which is the size-three profile, where $d_3=d_1+d_2>0$ by Proposition
\ref{prop:quotient}. For $t$ even and $0\le r_A<r_B\le t-1$ the only solution of $t\mid2(r_B-r_A)$
is $r_B-r_A=t/2$. Writing $a_1+a_2=t\bigl(|\lambda^{(r_A)}|+1\bigr)+2r_A$ as in the proof of
Proposition \ref{prop:quotient}, and likewise for $B$, the equality becomes
$t\bigl(|\lambda^{(r_A)}|-|\lambda^{(r_B)}|\bigr)=2(r_B-r_A)=t$.
\end{proof}

The two clauses of Corollary \ref{cor:zero} therefore live in different halves of the
core--quotient decomposition, and that is the answer to the question the criterion raises. The first
is a condition on the residue profile, hence on the $t$-core: the coordinate hyperplanes $n_i=0$ cut
out of the slice $\sum_i n_i=t+2$, which is not the Coxeter arrangement of type $A_{t-1}$. The
second is not a condition on the core at all; it is a single hyperplane in the quotient sizes, and
the admissible pairs of residues are indexed by the order-two element $r_B-r_A=t/2$ rather than by
any reflection. No restatement of the vanishing criterion in terms of $\core_t(\lambda)$ alone can
exist, which is for $\Phi_t$ what Remark \ref{rem:core} records for $\Psi_r$.

\subsection{What the triple forgets}\label{sec:fibrelattice}
Proposition \ref{prop:quotient} reads the triple off the quotient. Run backwards, the same reading
says which partitions the triple cannot tell apart, and says it exactly. Nothing here is deep --- it
is the core--quotient bijection at a fixed profile --- but it is what turns the compression of
Theorem \ref{thm:main} from a picture into a count.

\begin{proposition}[the two-class stratum is a lattice]\label{prop:fibrelattice}
Let $t\ge2$, $N=t+2$, and let $\mathcal{T}_t\subset\PP_N$ be the set of partitions of two-class
profile. Sending $\lambda$ to its two distinguished residues together with its $t$-quotient is a
bijection
\begin{equation}\label{eq:fibrelattice}
\mathcal{T}_t\;\xrightarrow{\ \sim\ }\!\!\!\coprod_{0\le r_A<r_B\le t-1}\!\!\!
\PP_2\times\PP_2\times\ZZ_{\ge0}^{\,t-2},
\qquad
\lambda\longmapsto\bigl(r_A,r_B;\ \lambda^{(r_A)},\lambda^{(r_B)};\ (m_i)_{i\ne r_A,r_B}\bigr),
\end{equation}
where $m_i$ is the single part of $\lambda^{(i)}$. Under it $\core_t(\lambda)$ depends only on the
pair $(r_A,r_B)$; write $\kappa(r_A,r_B)$ for its size. Then
\begin{equation}\label{eq:sizelattice}
|\lambda|\;=\;\bigl|\core_t(\lambda)\bigr|
+t\Bigl(\bigl|\lambda^{(r_A)}\bigr|+\bigl|\lambda^{(r_B)}\bigr|+\textstyle\sum_i m_i\Bigr).
\end{equation}
\end{proposition}

\begin{proof}
A class carrying $n_i$ beta numbers $a_1>\dots>a_{n_i}$, written $a_j=t\tilde a_j+i$, contributes
$\lambda^{(i)}=(\tilde a_1-n_i+1,\dots,\tilde a_{n_i})$, so a class of size $n_i$ gives a partition
with at most $n_i$ parts, and every such partition arises from exactly one strictly decreasing
$(\tilde a_j)$. In the two-class profile the sizes are $2,2$ and $1$ repeated $t-2$ times, which is
\eqref{eq:fibrelattice}; the residues are recorded rather than forgotten because the profile is part
of the datum, and they determine $\core_t(\lambda)$ by the Garvan--Kim--Stanton coordinate of
Remark \ref{rem:core}. Then \eqref{eq:sizelattice} is the core--quotient bijection.
\end{proof}

\begin{corollary}[the fibre of the triple, and its generating function]\label{cor:fibregf}
Let $(d_1,d_2,d_3)$ be an interval triple attained in the two-class profile, and set
$k_A=d_1/t-1$ and $k_B=d_2/t-1$. Its fibre in $\mathcal{T}_t$ is infinite. It is the disjoint union
of the branches indexed by the pairs $r_A<r_B$ and the signs $\eta=\pm1$ for which
\begin{equation}\label{eq:branch}
c:=\tfrac12\Bigl(\tfrac{\eta\,d_3-2(r_A-r_B)}{t}-k_A+k_B\Bigr)\ \in\ \ZZ ,
\end{equation}
the two signs giving one and the same branch when $d_3=0$. On a branch the free parameters are an
integer $v\ge\max(0,c)$ and $(m_i)\in\ZZ_{\ge0}^{t-2}$, with $\lambda^{(r_A)}=(v+k_A,v)$ and
$\lambda^{(r_B)}=(v-c+k_B,v-c)$. Consequently
\begin{equation}\label{eq:fibregf}
\sum_{\lambda\ \mathrm{in\ the\ fibre}}\!\!q^{|\lambda|}
\;=\;\frac{\sum_{\mathrm{branches}}q^{\,e}}{\bigl(1-q^{4t}\bigr)\bigl(1-q^{t}\bigr)^{t-2}},
\qquad e=\kappa(r_A,r_B)+t\bigl(k_A+k_B-2c\bigr)+4t\max(0,c),
\end{equation}
and the number of $\lambda$ in the fibre with $|\lambda|\le n$ is eventually a quasi-polynomial in
$n$ of degree $t-1$ --- eventually, because the shifts $q^{\,e}$ in the numerator put the count at
$0$ below the smallest of them.
\end{corollary}

\begin{proof}
By Proposition \ref{prop:quotient} the first two entries fix $k(\lambda^{(r_A)})$ and
$k(\lambda^{(r_B)})$, so each distinguished component is determined by its size alone; writing them
$(v+k_A,v)$ and $(u+k_B,u)$, the third entry reads
$t\bigl((2v+k_A)-(2u+k_B)\bigr)+2(r_A-r_B)=\eta\,d_3$, which is \eqref{eq:branch} with $c=v-u$; the
absolute value in $d_3$ is what puts the two signs there, and it is vacuous at $d_3=0$.
Substituting in \eqref{eq:sizelattice} gives $|\lambda|=e+4t\bigl(v-\max(0,c)\bigr)+t\sum_i m_i$, and
summing the geometric series in $v$ and in each $m_i$ gives \eqref{eq:fibregf}. Its denominator has
$t-1$ factors, whence the degree.
\end{proof}

\begin{remark}[the sign is not along for the ride]\label{rem:signsees}
Two readings, and the second is a warning. First, a fibre is infinite for every $t$ and every
attained triple, which is what Figure \ref{fig:fibres} shows and \eqref{eq:fibregf} explains: the
columns there do not thin out because a fibre is a lattice section and not a finite accident.
Second, \eqref{eq:fibregf} is the fibre of the \emph{triple} and not of $I_t$. The $t-2$ parts $m_i$
are invisible to $(d_1,d_2,d_3)$ by Proposition \ref{prop:quotient}, and it would be natural to read
them as directions along which $\Phi_t$ is constant. They are not, from $t=3$ on: the sign sees them,
as \eqref{eq:sign} makes possible, since $\varepsilon_\lambda$ carries $\inv(b_S)$, a statistic of
the whole beta-word, while \eqref{eq:fibrelattice} splits that word into two distinguished classes
and $t-2$ free letters. Over the ranges of \S\ref{sec:verif} the sign is constant in $(m_i)$ at fixed
$v$ in all $107$ of the slots at $t=2$, but in only $43$ of $62$ at $t=3$ and $26$ of $36$ at $t=4$.
What the sign does on a branch is Problem \ref{prob:fibres}.
\end{remark}

\subsection{The interval reading}\label{sec:intervals}
Figure \ref{fig:intervals} draws \eqref{eq:intervals} in an example. Corollary \ref{cor:zero} reads:
the character vanishes when a residue class is empty, or when the two intervals are
\emph{concentric}; and two intervals sharing an endpoint are concentric only if they coincide, which
is why the size-three profile never vanishes.

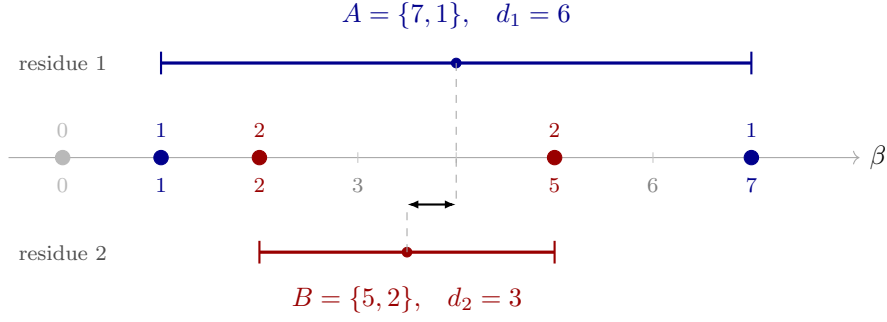
\begin{figure}[t]
\centering
\begin{tikzpicture}[font=\small,x=1.3cm,y=1.0cm]
\draw[blue!55!black,very thick] (1,1.25) -- (7,1.25);
\draw[blue!55!black,thick] (1,1.10) -- (1,1.40);
\draw[blue!55!black,thick] (7,1.10) -- (7,1.40);
\filldraw[blue!55!black] (4,1.25) circle (1.8pt);
\node[blue!55!black,above=9pt] at (4,1.25) {$A=\{7,1\}$,\quad $d_1=6$};
\draw[red!60!black,very thick] (2,-1.25) -- (5,-1.25);
\draw[red!60!black,thick] (2,-1.40) -- (2,-1.10);
\draw[red!60!black,thick] (5,-1.40) -- (5,-1.10);
\filldraw[red!60!black] (3.5,-1.25) circle (1.8pt);
\node[red!60!black,below=9pt] at (3.5,-1.25) {$B=\{5,2\}$,\quad $d_2=3$};
\draw[->,gray!70] (-0.55,0) -- (8.1,0) node[right,black]{$\beta$};
\foreach \b in {0,...,7} \draw[gray!45] (\b,-0.07) -- (\b,0.07);
\foreach \b in {3,6} \node[gray,font=\scriptsize,below=4pt] at (\b,0) {$\b$};
\foreach \b/\c/\r in {7/blue!55!black/1, 5/red!60!black/2, 2/red!60!black/2,
                      1/blue!55!black/1, 0/gray!55/0}
  {\filldraw[\c] (\b,0) circle (2.7pt);
   \node[\c,font=\scriptsize,below=4pt] at (\b,0) {$\b$};
   \node[\c,font=\scriptsize,above=4pt] at (\b,0) {$\r$};}
\draw[gray!70,dashed] (4,1.25) -- (4,-0.62);
\draw[gray!70,dashed] (3.5,-1.25) -- (3.5,-0.62);
\draw[{Latex[length=1.5mm]}-{Latex[length=1.5mm]},thick] (3.5,-0.62) -- (4,-0.62);
\node[font=\scriptsize,gray!60!black,anchor=west] at (-0.55,1.25) {residue $1$};
\node[font=\scriptsize,gray!60!black,anchor=west] at (-0.55,-1.25) {residue $2$};
\end{tikzpicture}
\caption{The mechanism for $t=3$, $\lambda=(3,2)$: here $N=5$, $\beta=(7,5,2,1,0)$ with residues
$(1,2,2,1,0)$ modulo $3$, so no class is empty and two classes have size two. The three arguments of
\eqref{eq:main} are the two interval lengths and twice the distance between the interval centres;
the short double arrow marks that distance, here $\tfrac12$, so that $d_3=1$.}
\label{fig:intervals}
\end{figure}

\begin{example}
For $t=3$, $\lambda=(3,2)$ as in Figure \ref{fig:intervals}, $(d_1,d_2,d_3)=(6,3,1)$ and
\[
\Phi_3\bigl((3,2);z\bigr)=\varepsilon\,
\frac{\sinh(3\theta)\sinh(\theta/2)}{\sinh(3\theta/2)\sinh\theta}.
\]
For $t=3$, $\lambda=(3,1)$ the profile is the size-three one, $(d_1,d_2,d_3)=(3,3,6)$ and
$\Phi_3=\varepsilon\,\chi_2(z)$. For $t=4$, $\lambda=(1,1,1)$ we have $\beta=(6,5,4,2,1,0)$, the
class $3$ is empty and $\Phi_4=0$.
\end{example}

\subsection{The image of the map}
Theorem \ref{thm:main} sends every partition with at most $N$ rows to three integers, so the whole
family collapses onto a sparse subset of $\ZZ^3$. By Proposition \ref{prop:quotient} the first two
coordinates are multiples of $t$; the third is not a multiple of $t$ except in the size-three
profile, where $d_3=d_1+d_2$. Figure \ref{fig:image} shows the image for $t=3$ and $t=4$, coloured
by the number of partitions collapsing onto each point.

\begin{figure}[t]
\centering
\includegraphics[width=\textwidth]{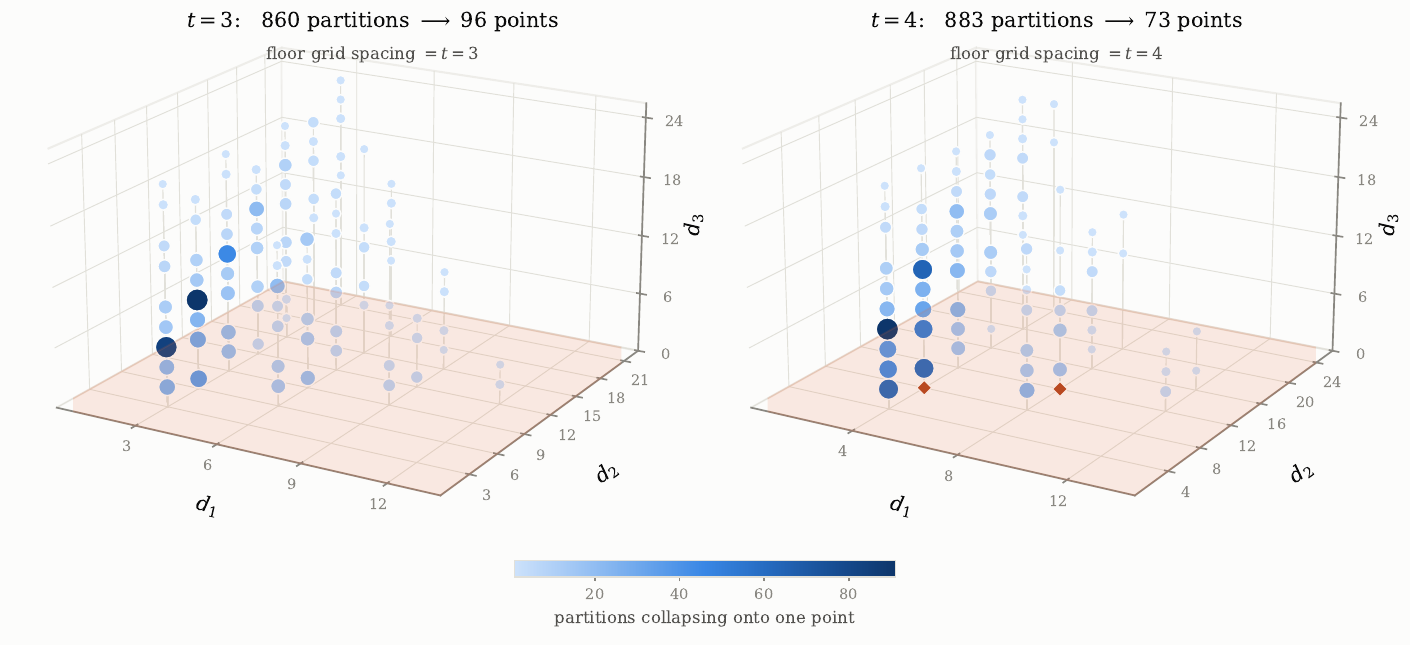}
\caption{The image of $\lambda\mapsto(d_1,d_2,d_3)$ over all $\lambda$ with $|\lambda|\le20$ and at
most $t+2$ rows, for $t=3$ (left) and $t=4$ (right), taken \emph{up to the exchange of $d_1$ and
$d_2$}, under which \eqref{eq:main} is symmetric; without that identification, and with $A$ and $B$
in the order $r_A\le r_B$ of Theorem \ref{thm:main}, the two panels carry $126$ and $92$ points
rather than $96$ and $73$. The compression is the point of the picture: $860$
partitions land on $96$ triples on the left and $883$ on $73$ on the right, and colour and area give
the number of partitions mapping to each. The floor grid spacing is $t$, since $d_1,d_2\in t\ZZ$ by Proposition
\ref{prop:quotient}; the shaded plane $d_3=0$ is the concentric locus of Corollary \ref{cor:zero},
where $s_\lambda$ vanishes, so those shapes are not in the image and are drawn on the plane as
diamonds instead. The two panels are the two parities, which is the content of Proposition
\ref{rem:lattice}: the odd one carries no diamond, and can carry none; the even one carries two,
between them the images of $32$ partitions.}
\label{fig:image}
\end{figure}

Figure \ref{fig:image} is one half of the picture: it shows which triples occur, not what a triple
hides. Figure \ref{fig:fibres} is the other half. Over each cell of the $(d_1,d_2)$ floor it stands
the column of sizes $|\lambda|$ that land there, so a column is a set of partitions of different
sizes on which $\Phi_t$ takes one and the same value --- and the columns do not thin out as
$|\lambda|$ grows, since the values are confined to a lattice and the partitions are not. The most
visible one is the fibre of the empty partition: at $t=3$ the invariant $I_3=(3,3,2,+1)$ is shared by
$26$ partitions, of $19$ different sizes between $0$ and $20$, so $\Phi_3$ does not distinguish
$\varnothing$ from a shape of twenty boxes.

\begin{figure}[!ht]
\centering
\includegraphics[width=\textwidth]{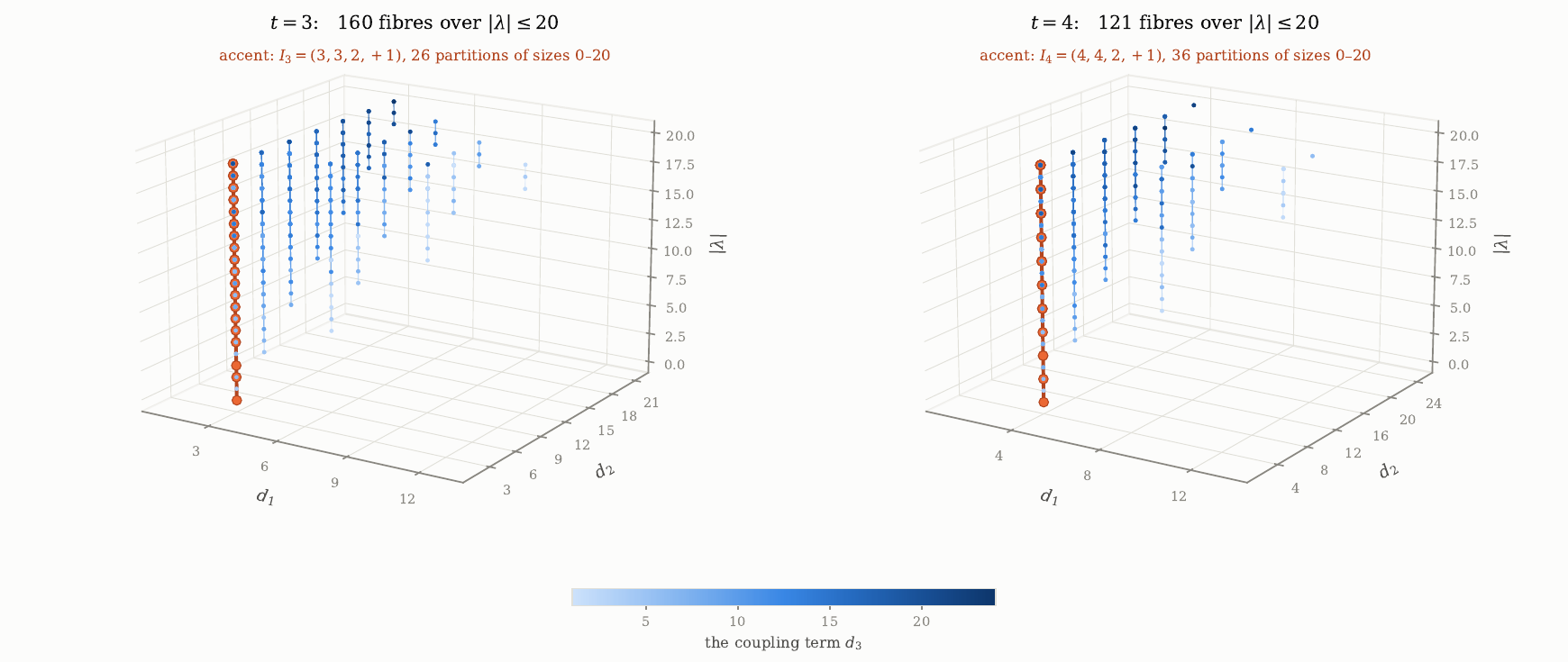}
\caption{The fibres of the evaluation invariant \eqref{eq:invariant}, the complement of Figure
\ref{fig:image}: the same range $|\lambda|\le20$, with the coupling term moved to colour and the
size $|\lambda|$ on the vertical axis. A column is one fibre, so its height is the range of sizes
this alphabet cannot tell apart; the floor ticks are the lattice $t\ZZ$ of Proposition
\ref{prop:quotient}, and only the half $d_1\le d_2$ is occupied because the triple is taken up to the
exchange, as in Figure \ref{fig:image}. That every marker of a column carries the same value is
computed and not assumed: over the $860$ partitions on the left and the $883$ on the right the
bialternant is constant on each fibre to $10^{-37}$. The accent column is the fibre of
$\lambda=\varnothing$.}
\label{fig:fibres}
\end{figure}

Figure \ref{fig:fibres} makes one thing about \eqref{eq:invariant} precise: $I_t$ is complete but
not minimal. The numerator of \eqref{eq:main} is symmetric in $d_1,d_2,d_3$, so the value sees only the multiset
$\{d_1,d_2,d_3\}$ together with the sign, and that is a real reduction --- over $t=4$ and
$|\lambda|\le20$ the $121$ invariants take only $110$ values, while at $t=3$ in the same range the
two counts agree. That reduction is the whole of it: nothing below the multiset and the sign is
lost, and this is where the compression announced in the introduction becomes an exact statement
rather than an upper bound.

\begin{proposition}[the multiset and the sign are exactly what is seen]\label{prop:minimal}
Fix $t\ge2$ and let $\lambda,\mu\in\PP_N$ with $\Phi_t(\lambda;z)\not\equiv0$. Then
\[
\Phi_t(\lambda;z)=\Phi_t(\mu;z)
\quad\Longleftrightarrow\quad
\{d_1,d_2,d_3\}(\lambda)=\{d_1,d_2,d_3\}(\mu)\ \text{as multisets and}\
\varepsilon_\lambda=\varepsilon_\mu .
\]
\end{proposition}

\begin{proof}
$(\Leftarrow)$ is \eqref{eq:main}, whose right-hand side is symmetric in the three arguments. For
$(\Rightarrow)$, note first that $\Phi_t(\mu;z)=\Phi_t(\lambda;z)\not\equiv0$, so $\mu$ too has a
non-degenerate profile and $\tilde d_3(\mu)\ne0$ and both invariants are defined. Write
$d=d(\lambda)$ and $e=d(\mu)$, all six entries being $\ge1$, and put
$u=z^{1/2}$, so that $2\sinh(v\theta/2)=u^{v}-u^{-v}$ with integer exponents. The
denominator of \eqref{eq:main} depends on $t$ alone, so the hypothesis is
\[
\varepsilon_\lambda\prod_{i=1}^{3}\bigl(u^{d_i}-u^{-d_i}\bigr)
=\varepsilon_\mu\prod_{i=1}^{3}\bigl(u^{e_i}-u^{-e_i}\bigr).
\]
Using $u^{v}-u^{-v}=u^{-v}(u^{2v}-1)$ on each factor and setting $D=\sum_id_i$, $E=\sum_ie_i$,
\[
\varepsilon_\lambda\,u^{-D}\prod_{i}\bigl(u^{2d_i}-1\bigr)
=\varepsilon_\mu\,u^{-E}\prod_{i}\bigl(u^{2e_i}-1\bigr).
\]
Each product is a polynomial with constant term $(-1)^3=-1$, so the lowest exponent occurring is
$-D$ on the left and $-E$ on the right, with coefficients $-\varepsilon_\lambda$ and
$-\varepsilon_\mu$: hence $D=E$ and $\varepsilon_\lambda=\varepsilon_\mu$, and the two products
agree. Now $u^{2v}-1=\prod_{m\mid2v}Q_m(u)$, where $Q_m$ is the $m$-th cyclotomic polynomial ---
written $Q$ and not $\Phi$, which in this paper is the object. The $Q_m$ are
irreducible and pairwise distinct, so by unique factorization in $\ZZ[u]$ the two sides carry the
same multiset of cyclotomic factors,
\[
\biguplus_i\{m:m\mid 2d_i\}=\biguplus_i\{m:m\mid 2e_i\} .
\]
Its largest element is $\max_i2d_i$ on the left and $\max_i2e_i$ on the right, so those agree;
deleting one copy of the divisor set of that maximum from each side and repeating twice more gives
$\{2d_1,2d_2,2d_3\}=\{2e_1,2e_2,2e_3\}$.
\end{proof}

\noindent So the pair (multiset, sign) is a \emph{complete separating} invariant for this evaluation
off the vanishing locus --- complete because the value factors through it, separating because
distinct data give distinct values --- and that is a theorem and not a range. We stop short of
calling it \emph{minimal} outright, since minimality is relative to a class of invariants one would
have to fix first. The hypothesis $\Phi_t(\lambda;z)\not\equiv0$ is not
removable and not a defect: on the vanishing locus every invariant collapses to one value, which is
why \eqref{eq:invariant} sends the whole of it to $0$. We keep the ordered form there
because $d_3$ is the coordinate the reciprocal pair contributes and the other two are not, and
Corollary \ref{cor:fibregf} describes the fibres of either, since the two differ by a symmetry of
the numerator and not by anything the lattice sees. Lemma \ref{lem:single} below is the same
argument run against a target that is not itself a value of $\Phi_t$.

\begin{proposition}[a shorter form of the sign]\label{rem:sign}
Let the residue profile be non-degenerate and $\tilde d_3\ne0$, and order the two distinguished
classes so that $r_A\le r_B$. Then
\begin{equation}\label{eq:shortsign}
\boxed{\;\;\varepsilon_\lambda\;=\;(-1)^{\lfloor t/2\rfloor}\,\sgn(\sigma)\,(-1)^{\,r_A+r_B}\,
\sgn(\tilde d_3)\;\;}
\end{equation}
where $\sigma$ is the permutation of \S\ref{sec:background}.
\end{proposition}

\begin{proof}
Let $w$ be the residue word of $\beta(\lambda)$ read in increasing column order. Since $\beta$
decreases with the column index, grouping the parts by class with the values decreasing inside each
class is the stable sort of $w$, so $\sgn(\sigma)=(-1)^{\inv(w)}$. Both sides of
\eqref{eq:shortsign} carry the factor $\sgn(\tilde d_3)$, and
$\sgn(a_1-b_1)=(-1)^{[a_1<b_1]}$, so by \eqref{eq:sign} the assertion is equivalent to
\begin{equation}\label{eq:signparity}
j_{A_1}+j_{B_1}+\inv(w)+\inv(b_S)+[a_1<b_1]\;\equiv\;
\Bigl\lfloor\tfrac t2\Bigr\rfloor+t+\binom{N+1}{2}+r_A+r_B \pmod 2 .
\end{equation}

We use one parity count. Let $u$ be a word, $p$ a position, $c=u_p$, and let $I(p)$ be the number of
inversions of $u$ in which $p$ takes part; write $\nu_{<c}$ for the number of letters of $u$ strictly
below $c$ and $e_p$ for the number of occurrences of $c$ strictly before $p$. Putting
$x=\#\{j<p:u_j>c\}$, the positions before $p$ split as $p-1=x+e_p+\#\{j<p:u_j<c\}$ and the letters
below $c$ split as $\nu_{<c}=\#\{j<p:u_j<c\}+\#\{j>p:u_j<c\}$, whence
\begin{equation}\label{eq:parity}
I(p)\;=\;x+\#\{j>p:u_j<c\}\;=\;2x+e_p+\nu_{<c}-(p-1)\;\equiv\;(p-1)+\nu_{<c}+e_p \pmod 2 .
\end{equation}
Since $b_S$ is $w$ with the letters at the positions $j_{A_1}$ and $j_{B_1}$ deleted,
\[
\inv(w)-\inv(b_S)=I(j_{A_1})+I(j_{B_1})
-\bigl[\,\{j_{A_1},j_{B_1}\}\ \text{is an inversion of }w\,\bigr],
\]
and $\inv(w)+\inv(b_S)\equiv\inv(w)-\inv(b_S)$, so \eqref{eq:parity} computes the left-hand side of
\eqref{eq:signparity}.

In the two-class profile $r_A<r_B$, and $w$ carries $r_A$ twice, $r_B$ twice and every other residue
once. Because $\beta$ decreases with the column index, $a_1>a_2$ puts $j_{A_1}$ at the \emph{first}
occurrence of $r_A$, so $e=0$ there, and likewise at $j_{B_1}$. The letters below $r_A$ are one copy
of each of $0,\dots,r_A-1$, so $\nu_{<r_A}=r_A$; those below $r_B$ are the same together with the
second copy of $r_A$, so $\nu_{<r_B}=r_B+1$. The two positions form an inversion exactly when the
larger letter comes first, that is when $j_{B_1}<j_{A_1}$, that is when $a_1<b_1$. Hence
\begin{multline*}
\inv(w)-\inv(b_S)\;\equiv\;(j_{A_1}-1+r_A)+(j_{B_1}-1+r_B+1)+[a_1<b_1]\\
\equiv\;j_{A_1}+j_{B_1}+r_A+r_B+1+[a_1<b_1],
\end{multline*}
and substituting this in \eqref{eq:signparity} cancels the columns and $[a_1<b_1]$ in pairs, leaving
$\lfloor t/2\rfloor+t+\binom{N+1}{2}\equiv1$.

In the size-three profile $r_A=r_B=:r$ and the class $\{p>q>p'\}$ sits at columns $j_1<j_2<j_3$ with
$A=\{p,q\}$ and $B=\{q,p'\}$, so $j_{A_1}=j_1$ and $j_{B_1}=j_2$ carry the same letter $r$, with
$e=0$ and $e=1$ respectively and $\nu_{<r}=r$ in both; two equal letters are not an inversion, and
$a_1>b_1$. So $\inv(w)-\inv(b_S)\equiv j_1+j_2+1$, which is the displayed expression again because
$r_A+r_B=2r$ and $[a_1<b_1]=0$, and \eqref{eq:signparity} reduces to the same congruence.

Finally $N=t+2$, so that congruence reads $\lfloor t/2\rfloor+t+\binom{t+3}{2}\equiv1$: for $t=2m$
its three terms are $m$, $0$ and $m+1$, and for $t=2m+1$ they are $m$, $1$ and $m$. Both sums are
odd, and the identity holds for every $t\ge2$.
\end{proof}

The ordering by residue is not a convenience, and it is worth saying why. Exchanging the two blocks
negates the sign $\kappa_{11}$ of Lemma \ref{lem:L4} --- through its $\sgn(a_1-b_1)$ factor --- and negates
$\sgn(\tilde d_3)$ as well, so it leaves $\varepsilon_\lambda$, which is their product,
invariant; as it must, since $\varepsilon_\lambda$ is attached to $\lambda$. The right-hand side
above is \emph{not} invariant: $(-1)^{r_A+r_B}$ is symmetric, so only its last factor changes sign.
Without a rule fixing which block is $A$, the right-hand side is therefore not a function of
$\lambda$ at all. The hypothesis enters the proof at one point and one only: it is what makes
$\nu_{<r_A}=r_A$ and $\nu_{<r_B}=r_B+1$ rather than the other way round, and dropping it exchanges
those two counts and breaks \eqref{eq:signparity}.

Read the other way round, \eqref{eq:shortsign} says that $\Phi_t$ is a function of a small
combinatorial datum and of nothing else, and it says so for every $t$ rather than over a range: by
Theorem \ref{thm:main} the triple $(d_1,d_2,d_3)$ determines $\Phi_t$ up to sign, and adjoining
$\sgn(\sigma)$, the ordered pair $r_A\le r_B$ and the \emph{orientation}
$\sgn(\tilde d_3)$ determines it outright. The shape of $\lambda$ enters nowhere else.
The orientation is not optional, and that part remains a check rather than a proof: over $t\le6$ and
$|\lambda|\le14$, putting $\sgn(a_1-b_1)$ in its place leaves collisions at $t=2$.

\section{Proof of Theorem \ref{thm:main}}\label{sec:proof}

The proof uses the Vandermonde determinant, the Laplace expansion, and one cancellation lemma. We
give the steps separately so that a reader checking the ancillary scripts can localise any
disagreement.

\begin{figure}[!ht]
\centering
\includegraphics[width=\textwidth]{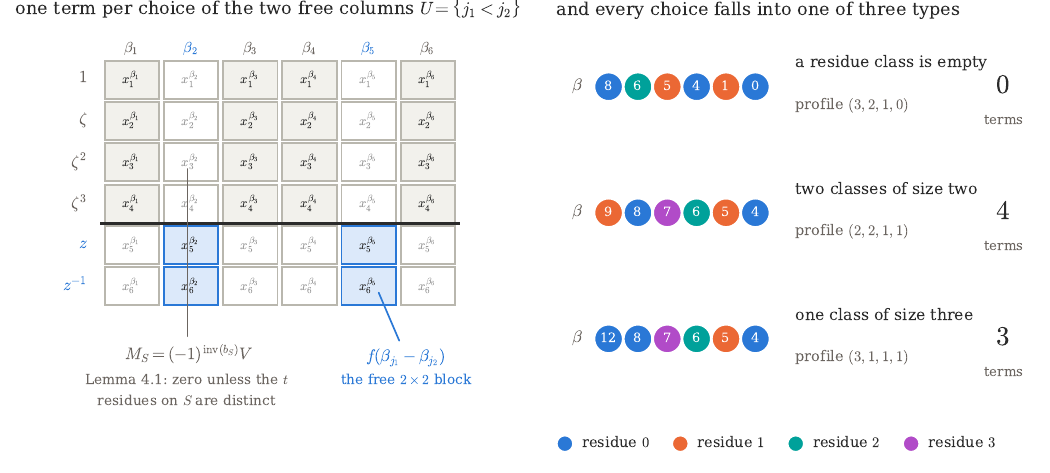}
\caption{The architecture of the proof, drawn at $t=4$, $N=6$. \emph{Left:} the determinant
$\det(x_i^{\beta_j})$ expanded along the $t$ frozen rows. Each choice of the two columns
$U=\{j_1<j_2\}$ left to the free rows $z,z^{-1}$ leaves a $t\times t$ minor $M_S$ on the
complementary columns, which by Lemma \ref{lem:L1} is $\pm V$ or zero, times the $2\times2$
block $f(\beta_{j_1}-\beta_{j_2})$. \emph{Right:} by the excess-two count \eqref{eq:excess}
only three residue profiles occur, one per row, each drawn on a beta set that realises it; the
number at the right is how many choices of $U$ survive, counted rather than asserted --- $0$, $4$
and $3$, which is the trichotomy the section is organised around. Colour is the residue class.
Computed by \texttt{anc/fig\_proof.py}.}
\label{fig:laplace}
\end{figure}

The expansion itself is the standard move of this literature and we claim nothing for it: Albion
\cite{Alb25} applies the Laplace expansion to a determinant of this shape \emph{according to the
given block structure}, holding a first set of rows fixed, following Krattenthaler and
Ciucu--Krattenthaler. What is particular here is not the expansion but what it runs into.

What the argument really does is classify rather than compute. Expanding along the $t$ frozen rows
leaves one term for each way of choosing which two columns are \emph{not} frozen, and the excess-two
count says every such choice falls into exactly one of three combinatorial types: none survives
(a residue class is empty), four survive (two classes of size two), or three do (one class of size
three). Everything after that is bookkeeping inside a type, and Lemma \ref{lem:L5} shows the last two
types are one statement seen twice. Figure \ref{fig:laplace} is that sentence drawn.

\begin{lemma}\label{lem:L1}
Let $S$ be a set of $t$ columns and $M_S=\det(\zeta^{k\beta_j})_{0\le k<t,\,j\in S}$. Then $M_S=0$
unless the residues $\{\beta_j\bmod t:j\in S\}$ are pairwise distinct, in which case
$M_S=(-1)^{\inv(b_S)}V$.
\end{lemma}

\begin{proof}
Since $\zeta^{k\beta_j}=(\zeta^{\beta_j})^k$, the matrix is a Vandermonde matrix in the values
$y_j=\zeta^{\beta_j}$, whose determinant is $\prod_{j<j'}(y_{j'}-y_j)$. The value $y_j$ depends on
$\beta_j$ only modulo $t$, so two columns of $S$ sharing a residue make the determinant vanish. If
the residues are pairwise distinct then $|S|=t$ forces $\{y_j\}_{j\in S}=\zt$ as a set, and the
Vandermonde determinant is alternating in its arguments, so it equals $V$ times the sign of the
permutation sorting $b_S$ increasingly, which is $(-1)^{\inv(b_S)}$.
\end{proof}

\begin{lemma}\label{lem:L2}
$\det\bigl(x_i^{N-j}\bigr)_{i,j=1}^{N}
=(-1)^{\binom t2}\,V\,(z^t-1)(z^{-t}-1)(z-z^{-1})$ for the alphabet
$x=(1,\zeta,\dots,\zeta^{t-1},z,z^{-1})$ in that order.
\end{lemma}

\begin{proof}
The determinant is $\prod_{i<i'}(x_i-x_{i'})$. Pairs inside $\zt$ contribute
$\prod_{k<k'}(\zeta^k-\zeta^{k'})=(-1)^{\binom t2}V$; pairs $(\zeta^k,z)$ contribute
$\prod_k(\zeta^k-z)=(-1)^t(z^t-1)$ and pairs $(\zeta^k,z^{-1})$ contribute $(-1)^t(z^{-t}-1)$; the
last pair contributes $z-z^{-1}$. The two factors $(-1)^t$ cancel.
\end{proof}

\noindent We compute it because it is three lines, not because it is unfamiliar: a determinant over
an alphabet closed under $x\mapsto x^{-1}$, evaluated in closed product form, is exactly the subject
of Krattenthaler's complement to the determinant calculus \cite[\S5.11]{KratADCC}, whose Lemma 20
carries the Weyl denominator formulas for $B_n$, $C_n$ and $D_n$ among its special cases. What that
literature does not evaluate --- and what the next lemma needs --- is the same determinant with a
root-of-unity orbit sitting inside the alphabet.

\begin{lemma}\label{lem:L3}
Let $\mathcal U$ be the set of pairs $U=\{j_1<j_2\}$ of columns whose complement $S_U$ carries all
$t$ residues. Then
\[
\Phi_t(\lambda;z)=\frac{(-1)^{\,t+\binom{N+1}{2}}}{(z^t-1)(z^{-t}-1)(z-z^{-1})}
\sum_{U\in\mathcal U}(-1)^{\,j_1+j_2+\inv(b_{S_U})}\,f(\beta_{j_1}-\beta_{j_2}).
\]
\end{lemma}

\begin{proof}
Expanding the numerator by Laplace along the first $t$ rows,
\[
\det\bigl(x_i^{\beta_j}\bigr)=\sum_{|S|=t}(-1)^{\binom{t+1}{2}+\sum_{j\in S}j}\,M_S\,C_{S^c},
\]
where for $S^c=\{j_1<j_2\}$ the complementary $2\times2$ minor is
$z^{\beta_{j_1}-\beta_{j_2}}-z^{-(\beta_{j_1}-\beta_{j_2})}=f(\beta_{j_1}-\beta_{j_2})$. By Lemma
\ref{lem:L1} only $S$ with pairwise distinct residues contribute. Writing
$\sum_{j\in S}j=\binom{N+1}{2}-j_1-j_2$ and dividing by Lemma \ref{lem:L2}, the factors $V$ cancel
and $(-1)^{\binom{t+1}{2}-\binom{t}{2}}=(-1)^t$.
\end{proof}

If some $n_i=0$ then no $S$ of size $t$ carries all residues, $\mathcal U=\varnothing$, and
$\Phi_t=0$: this is Theorem \ref{thm:main}(i). Otherwise all $n_i\ge1$ and by \eqref{eq:excess}
either two classes have size two, in which case $S_U$ must drop one element of each and
$|\mathcal U|=4$, or one class has size three, in which case $S_U$ drops two of the three and
$|\mathcal U|=3$.

\begin{lemma}[the column move]\label{lem:L4}
Let $A$ and $B$ be distinct classes of size two, at columns $j_{A_1}<j_{A_2}$ and
$j_{B_1}<j_{B_2}$. For $i,j\in\{1,2\}$ let $U_{ij}$ consist of the column of $a_i$ and the column of
$b_j$, and set $\kappa_{ij}=(-1)^{\,j_{A_i}+j_{B_j}+\inv(b_{S_{U_{ij}}})}\sgn(a_i-b_j)$, so that
the $U_{ij}$ term of Lemma \ref{lem:L3} equals $\kappa_{ij}f(a_i-b_j)$. Then
$\kappa_{ij}=\kappa_{11}(-1)^{i-1}(-1)^{j-1}$.
\end{lemma}

\begin{proof}
Fix $j$ and compare $i=1$ with $i=2$. The sets $S_{U_{1j}}$ and $S_{U_{2j}}$ differ only in that the
first contains the column $j_{A_2}$ and the second the column $j_{A_1}$, and the letter carried is
the residue $r_A$ in both. Passing from one word to the other moves that letter across exactly the
kept columns strictly between $j_{A_1}$ and $j_{A_2}$, and none of those letters equals $r_A$
because the class $A$ has only two elements; each crossing changes the parity of $\inv$. Hence
\[
(-1)^{\inv(b_{S_{U_{1j}}})-\inv(b_{S_{U_{2j}}})}
=(-1)^{(j_{A_2}-j_{A_1}-1)-[\,a_2<b_j<a_1\,]},
\]
using that $\beta$ is strictly decreasing in the column index, so that $j_{B_j}$ lies strictly
between $j_{A_1}$ and $j_{A_2}$ exactly when $a_2<b_j<a_1$. Multiplying by the Laplace factor
$(-1)^{j_{A_1}+j_{A_2}}$,
\[
\frac{\kappa_{1j}}{\kappa_{2j}}
=(-1)^{2j_{A_2}-1}(-1)^{[\,a_2<b_j<a_1\,]}\sgn(a_1-b_j)\sgn(a_2-b_j)=-1,
\]
because $\sgn(a_1-b_j)\sgn(a_2-b_j)=-1$ exactly when $b_j$ lies between $a_2$ and $a_1$. Exchanging
the roles of $A$ and $B$ gives $\kappa_{i1}/\kappa_{i2}=-1$.
\end{proof}

\begin{lemma}\label{lem:L5}
For all $c,p,q$ and all $u,v$,
\begin{align*}
\text{\rm(a)}&\quad \sum_{\epsilon,\eta\in\{\pm1\}}\epsilon\eta\,f(c+\epsilon p+\eta q)=f(c)f(p)f(q),\\
\text{\rm(b)}&\quad f(2u)-f(2u+2v)+f(2v)=-f(u)f(v)f(u+v).
\end{align*}
\end{lemma}

\begin{proof}
For (a), $\sum_{\epsilon,\eta}\epsilon\eta\,z^{c+\epsilon p+\eta q}
=z^{c}\bigl(\sum_\epsilon\epsilon z^{\epsilon p}\bigr)\bigl(\sum_\eta\eta z^{\eta q}\bigr)
=z^{c}f(p)f(q)$, and the same computation for the negative exponents gives
$z^{-c}(-f(p))(-f(q))$; subtract. For (b), put $s=z^u$, $w=z^v$ and expand
$(s-s^{-1})(w-w^{-1})(sw-s^{-1}w^{-1})$: the eight terms recombine as
$(s^2w^2-s^{-2}w^{-2})-(s^2-s^{-2})-(w^2-w^{-2})=f(2u+2v)-f(2u)-f(2v)$.
\end{proof}

\begin{proof}[Proof of Theorem \ref{thm:main}]
In the two-class profile, Lemmas \ref{lem:L3} and \ref{lem:L4} give
\[
\sum_{U\in\mathcal U}(\cdots)=\kappa_{11}\sum_{i,j}(-1)^{i-1}(-1)^{j-1}f(a_i-b_j).
\]
Writing $a_i=\alpha+\epsilon_ip$ and $b_j=\gamma+\eta_jq$ with $p=d_1/2$, $q=d_2/2$ and
$\epsilon=\eta=(+1,-1)$, we have $a_i-b_j=c+\epsilon_ip-\eta_jq$ with $c=\alpha-\gamma$. Replacing
$\eta$ by $-\eta$ turns the sum into $-\sum_{\epsilon,\eta}\epsilon\eta f(c+\epsilon p+\eta q)$,
which is $-f(c)f(p)f(q)$ by Lemma \ref{lem:L5}(a). Since
$(z^t-1)(z^{-t}-1)=2-z^t-z^{-t}=-f(t/2)^2$ and $z-z^{-1}=f(1)$, the two minus signs cancel, and
substituting $f(u)=2\sinh(u\theta)$ the numerical factors cancel. Here $2c=\tilde d_3$. If
$\tilde d_3=0$ then $f(c)=0$ and the whole sum vanishes, which is the first assertion of {\rm(ii)};
otherwise $\sinh(c\theta)=\sgn(\tilde d_3)\sinh(d_3\theta/2)$, giving \eqref{eq:main} with
$\varepsilon_\lambda=(-1)^{\,t+\binom{N+1}{2}}\kappa_{11}\sgn(\tilde d_3)$, which is
\eqref{eq:sign} by the definition of $\kappa_{11}$ in Lemma \ref{lem:L4}.

For the size-three profile $\{p>q>r\}$ at columns $j_1<j_2<j_3$ the same column-move argument
applies, again because no letter strictly between two of these columns carries the class residue,
and gives alternating signs $(\varsigma,-\varsigma,\varsigma)$ for the three terms $f(p-q)$,
$f(p-r)$, $f(q-r)$. By Lemma \ref{lem:L5}(b) with $u=(p-q)/2$, $v=(q-r)/2$ the sum is
$-\varsigma f(u)f(v)f(u+v)$, which is the two-class formula for $A=\{p,q\}$, $B=\{q,r\}$ with
$\varsigma=\kappa_{11}$. Equivalently, the missing fourth selection of the two-class profile would
drop $q$ twice and its term carries $f(0)=0$, so the two profiles are one statement.
\end{proof}

All the statements of this paper were also checked numerically; the counts are collected once, in
\S\ref{sec:verif}, rather than repeated as they arise.

\section{Independence on the reciprocal locus}\label{sec:independence}

Theorem 5.3 of \cite{AK25} is stated for free variables, and $(z,\zb)$ is not a free pair. Nothing
in this section contradicts it: the reciprocal locus lies outside its hypotheses, not against them,
and what we do here is determine what its criterion becomes there. Read at face value,
\eqref{eq:AK53} would say that for $\ell(\lambda)\le2$ one has
$s_\lambda(z,\zb,\zt)=s_\lambda(z,\zb)$ exactly on the $t$-cores. Theorem
\ref{thm:main} evaluates the left-hand side for every $\lambda$, and comparing the two shows that the
reciprocal locus carries further solutions --- exactly one family. Measuring how far a criterion
survives a degeneration is a way of saying how sharp it was, and the answer here is that it is sharp
up to a single family.

We first isolate when $\Phi_t$ degenerates to a single character. Recall
$\chi_k(z)=(z^{k+1}-z^{-k-1})/(z-z^{-1})$.

\begin{lemma}\label{lem:single}
$\Phi_t(\lambda;z)=\pm\chi_k(z)$ for some $k\ge0$ if and only if the multiset $\{d_1,d_2,d_3\}$
contains $t$ at least twice, and in that case the remaining entry equals $2(k+1)$.
\end{lemma}

\begin{proof}
Put $u=z^{1/2}$, so that $2\sinh(d_i\theta/2)=u^{d_i}-u^{-d_i}$, $2\sinh(t\theta/2)=u^{t}-u^{-t}$ and
$2\sinh\theta=u^{2}-u^{-2}$; all exponents are integers. By \eqref{eq:main}, the asserted equality is
\[
\prod_{i=1}^{3}\bigl(u^{d_i}-u^{-d_i}\bigr)
=\pm\bigl(u^{t}-u^{-t}\bigr)^{2}\bigl(u^{2k+2}-u^{-2k-2}\bigr).
\]
Using $u^{n}-u^{-n}=u^{-n}(u^{2n}-1)$ on each factor and comparing the monomial prefactors gives
$d_1+d_2+d_3=2t+2(k+1)$ together with
\[
\prod_{i=1}^{3}\bigl(u^{2d_i}-1\bigr)=\pm\bigl(u^{2t}-1\bigr)^{2}\bigl(u^{4(k+1)}-1\bigr),
\]
and evaluating at $u=0$ fixes the sign to $+$. Now $u^{2n}-1=\prod_{m\mid 2n}Q_m(u)$ with
$Q_m$ the cyclotomic polynomials as in Proposition \ref{prop:minimal}, which are irreducible and
pairwise distinct, so by unique
factorization in $\ZZ[u]$ the two sides have the same multiset of cyclotomic factors. On the left
that multiset is $\biguplus_i\{m:m\mid 2d_i\}$ and on the right it is
$2\times\{m:m\mid 2t\}\uplus\{m:m\mid 4(k+1)\}$. Its largest element is $\max_i 2d_i$ on the left and
$\max(2t,4(k+1))$ on the right; deleting the divisor set of that maximum from both sides and
repeating twice more yields
\[
\{2d_1,2d_2,2d_3\}=\{2t,\,2t,\,4(k+1)\}
\]
as multisets, which is the assertion. The converse is the same computation read backwards.
\end{proof}

\begin{theorem}\label{thm:extra}
Let $\ell(\lambda)\le2$. Then
\[
\Phi_t(\lambda;z)=s_\lambda(z,\zb)
\]
--- an equality, not an equality up to sign --- if and only if either
$\lambda=\core_t(\lambda)$, or
\begin{equation}\label{eq:extra}
t\ \text{is even}\quad\text{and}\quad \lambda=\Bigl(\lambda_2+\tfrac{3t}{2}-1,\ \lambda_2\Bigr)
\quad\text{with}\quad \tfrac t2\le\lambda_2\le t-1 .
\end{equation}
In the second case $\{d_1,d_2,d_3\}=\{3t,t,t\}$ and $\varepsilon_\lambda=+1$, while
$\core_t(\lambda)=\bigl(\tfrac t2-1+j,j\bigr)$
and $\quot_t(\lambda)$ has the single nonempty component $(2)$, in slot $j=\lambda_2-\tfrac t2$.
\end{theorem}

Figure \ref{fig:locus} is the statement computed rather than asserted: both families are plotted
over a range of $t$, and the odd $t$ where \eqref{eq:extra} cannot occur is drawn as the control.

\begin{figure}[!ht]
\centering
\includegraphics[width=\textwidth]{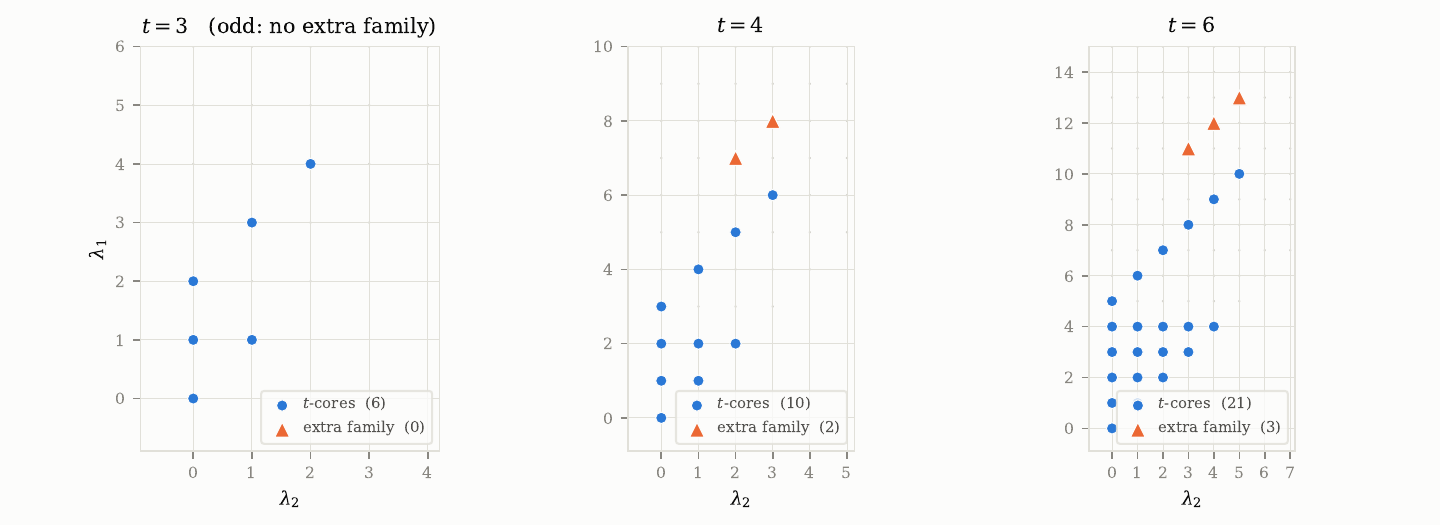}
\caption{Theorem \ref{thm:extra}, computed rather than asserted. For every two-row $\lambda$ in
range, both sides of $\Phi_t(\lambda;z)=s_\lambda(z,\zb)$ were evaluated, sign included; the
partitions where
they agree are plotted and coloured by family. Circles are the $t$-cores, the solutions already
described by \eqref{eq:AK53}; triangles are the extra family \eqref{eq:extra}, which the reciprocal
specialization creates. The odd case $t=3$ is the control: by Theorem \ref{thm:extra} no extra
solution can exist there, and none is found.}
\label{fig:locus}
\end{figure}

\begin{proof}
Since $\ell(\lambda)\le 2$ and $N=t+2$, the beta set is
$\beta(\lambda)=\bigl(A'+t,\;B'+t,\;t-1,t-2,\dots,1,0\bigr)$ where
\[
A':=\lambda_1+1,\qquad B':=\lambda_2,\qquad A'>B'\ge0 .
\]
The block $\{0,1,\dots,t-1\}$ meets every residue class exactly once, so no class is empty and the
two large parts decide the profile. Write $A'=t\alpha+\rho_A$ and $B'=t\gamma+\rho_B$ with
$0\le\rho_A,\rho_B\le t-1$. We first determine the solutions of
$\Phi_t(\lambda;z)=\pm s_\lambda(z,\zb)$ and settle the sign at the end. As
$s_\lambda(z,\zb)=\chi_{\lambda_1-\lambda_2}(z)$, Lemma
\ref{lem:single} says that this holds precisely when $\{d_1,d_2,d_3\}$ contains $t$ twice
and the third entry is $2(\lambda_1-\lambda_2+1)=2(A'-B')$.

\emph{Case $\rho_A=\rho_B=:\rho$.} One class has size three, namely $\{A'+t,B'+t,\rho\}$, and by
Theorem \ref{thm:main} the three arguments are
$d_1=A'-B'$, $d_2=B'+t-\rho=t(\gamma+1)$, $d_3=A'+t-\rho=t(\alpha+1)$. Here $t\mid A'-B'$, say
$A'-B'=ts$ with $s\ge1$. The pair $d_2=d_3=t$ forces $\alpha=\gamma=0$, hence $A'=B'$, excluded.
The pair $d_1=d_3=t$ forces $s=1$ and $\alpha=0$, hence $A'<t$ and $A'=B'+t\ge t$, excluded. The pair
$d_1=d_2=t$ forces $s=1$ and $\gamma=0$, i.e. $A'=B'+t$ with $B'<t$; then $\alpha=1$ and the third
entry is $d_3=2t=2(A'-B')$, as required. So this case contributes exactly the $\lambda$ with
$A'=B'+t$, $B'<t$.

\emph{Case $\rho_A\ne\rho_B$.} Two classes have size two, namely $\{A'+t,\rho_A\}$ and
$\{B'+t,\rho_B\}$, and the three arguments are
\[
d_A=t(\alpha+1),\qquad d_B=t(\gamma+1),\qquad
d_3=\bigl|\,t(\alpha-\gamma)+2(\rho_A-\rho_B)\,\bigr| .
\]
If $d_A=d_B=t$ then $\alpha=\gamma=0$, so $A',B'<t$ and $d_3=2|\rho_A-\rho_B|=2(A'-B')$
automatically: every such $\lambda$ is a solution. Otherwise exactly one of $d_A,d_B$ equals $t$ and
$d_3=t$. If $\alpha=0$ and $\gamma\ge1$ then $|-t\gamma+2(\rho_A-\rho_B)|=t$ with
$|\rho_A-\rho_B|\le t-1$ forces $\gamma=2$ and $\rho_A-\rho_B=t/2$; but then
$\lambda_1=A'-1=\rho_A-1<t$ while $\lambda_2=B'\ge2t$, contradicting $\lambda_1\ge\lambda_2$. If
$\gamma=0$ and $\alpha\ge1$ then $|t\alpha+2(\rho_A-\rho_B)|=t$ forces, for $\alpha=1$,
$\rho_A=\rho_B$ (excluded), and for $\alpha=2$,
\[
\rho_B-\rho_A=\tfrac t2,
\]
which requires $t$ even; $\alpha\ge3$ gives $|2(\rho_A-\rho_B)|\ge2t$, impossible. So $A'=2t+\rho_A$,
$B'=\rho_B=\rho_A+\tfrac t2$, whence
\[
\lambda_1=2t+\rho_A-1,\qquad \lambda_2=\rho_A+\tfrac t2,\qquad
\lambda_1-\lambda_2=\tfrac{3t}{2}-1,
\]
and $\rho_B\le t-1$ gives $0\le\rho_A\le\tfrac t2-1$, i.e. $\tfrac t2\le\lambda_2\le t-1$. The
remaining entry is $d_A=t(\alpha+1)=3t=2(\lambda_1-\lambda_2+1)$, as required. This is
\eqref{eq:extra}.

It remains to identify the solutions collected in the first case of each analysis with the
$t$-cores. For $\ell(\lambda)\le2$ the beta set in two slots is $(A',B')$, and $\lambda$ is a
$t$-core if and only if neither beta number can be lowered by $t$ into an unoccupied nonnegative
position, i.e. if and only if $B'<t$ and either $A'<t$ or $A'=B'+t$. These are exactly the two
families found above. Finally, in case \eqref{eq:extra} the class $\rho_A$ carries
$\{2t+\rho_A+t,\rho_A\}$, so its quotient component is $(2)$, every other component is empty, and
the core is as stated.

\emph{The sign.} The two families get it for different reasons. On the $t$-cores the equality is
\eqref{eq:AK53} itself, which is an equality and not an equality up to sign, so there is nothing to
prove. On \eqref{eq:extra} write $t=2m$ and $\lambda_2=m+j$ with $0\le j\le m-1$, so that
$\rho_A=j$ and $\rho_B=m+j$ and
\[
\beta(\lambda)=\bigl(6m+j,\;3m+j,\;2m-1,2m-2,\dots,1,0\bigr).
\]
Reducing modulo $t=2m$, the residue word read in increasing column order is
\[
w=\bigl(j,\;m+j,\;2m-1,2m-2,\dots,1,0\bigr),
\]
since $6m+j\equiv j$ and $3m+j\equiv m+j$, the latter because $j\le m-1$. Its inversions are of
three kinds: the decreasing tail contributes $\binom{2m}{2}$; the first letter is above exactly the
$j$ tail letters $0,\dots,j-1$; the second is above exactly the $m+j$ tail letters $0,\dots,m+j-1$;
and the first two letters are in increasing order, so they contribute none. Hence
\begin{equation}\label{eq:extrainv}
\inv(w)=\binom{2m}{2}+j+(m+j)=2m^{2}+2j\;\equiv\;0\pmod2 ,
\end{equation}
so $\sgn(\sigma)=(-1)^{\inv(w)}=+1$ by the first line of the proof of Proposition \ref{rem:sign}.
The two distinguished classes are $r_A=j<r_B=m+j$, carrying $A=\{6m+j,\,j\}$ and
$B=\{3m+j,\,m+j\}$, so $\tilde d_3=(6m+2j)-(4m+2j)=2m>0$ and $\sgn(\tilde d_3)=+1$. The remaining
two factors of \eqref{eq:shortsign} are $(-1)^{\lfloor t/2\rfloor}=(-1)^{m}$ and
$(-1)^{r_A+r_B}=(-1)^{m+2j}=(-1)^{m}$, whose product is $+1$. Therefore
$\varepsilon_\lambda=+1$, and since $d_1=6m=3t$, $d_2=2m=t$ and $d_3=2m=t$, \eqref{eq:main} reads
\[
\Phi_t(\lambda;z)=\frac{\sinh(3t\theta/2)\sinh^{2}(t\theta/2)}{\sinh^{2}(t\theta/2)\sinh\theta}
=\chi_{\frac{3t}{2}-1}(z)=\chi_{\lambda_1-\lambda_2}(z)=s_\lambda(z,\zb),
\]
with no sign to carry.
\end{proof}

\noindent What \eqref{eq:extrainv} buys is not decoration. Stated up to sign, Theorem
\ref{thm:extra} says that the specialization creates a family of shapes on which the twist is
invisible \emph{to within a sign}, which is a weaker statement than the one \eqref{eq:AK53} makes on
the cores; with the sign settled the two families are solutions of the same equation, and
``the reciprocal locus acquires exactly one further family'' is literally true.

\begin{remark}\label{rem:why}
Lemma \ref{lem:single} controls the value only up to sign, which is why the sign of Theorem
\ref{thm:extra} needs the separate paragraph above rather than falling out of the classification.

The extra family belongs to the specialization rather than to the alphabet. The step is
\cite[Lemma 5.2]{AK25}, on which \cite[Theorem 5.3]{AK25} rests: it reduces the independence to the
vanishing of $s_{\lambda/\mu}$ at the twisted alphabet for all $\mu\subsetneq\lambda$, and the
reduction uses the linear independence of the universal characters $f_\mu$ in the free variables. On the reciprocal locus $s_\mu(z,\zb)=\chi_{\mu_1-\mu_2}(z)$
depends only on $\mu_1-\mu_2$, so $s_{(1)}$, $s_{(2,1)}$ and $s_{(3,2)}$ agree there and that
independence is not available; Theorem \ref{thm:extra} measures the difference. A control confirms
the reading: for $\lambda=(3,1)$ and $t=2$ one has $s_\lambda(\zt,z,\zb)=s_\lambda(z,\zb)$, while
$s_\lambda(\zt,x,y)\ne s_\lambda(x,y)$ for free $x,y$.
\end{remark}

So the criterion does not survive the specialization unchanged, and what it acquires is small and
nameable: one family, for two-row shapes, indexed by an order-two element of the residue lattice and
classified by its core and quotient. The extra solutions belong to the reciprocal locus, not to the
alphabet --- they exist because a specialization has removed a linear independence the original
argument used, and that missing step is exactly what Theorem \ref{thm:extra} measures.

\section{An enumeration in which a parameter survives}\label{sec:enum}

Expanding $s_\lambda$ over semistandard tableaux turns Theorem \ref{thm:main} into a product formula
for a weighted count. Let $m_k(T)$ be the number of entries equal to $k$ in $T$.

\begin{corollary}\label{cor:enum}
Expanding $s_\lambda$ over semistandard tableaux and applying Theorem \ref{thm:main}: for every
$t\ge2$ and every $\lambda\in\PP_{t+2}$,
\[
\sum_{T\in\mathrm{SSYT}(\lambda,\,t+2)}
\zeta^{\,\sum_{k=1}^{t}(k-1)m_k(T)}\;z^{\,m_{t+1}(T)-m_{t+2}(T)}
\;=\;\varepsilon_\lambda\,
\frac{\sinh(d_1\theta/2)\sinh(d_2\theta/2)\sinh(d_3\theta/2)}{\sinh^2(t\theta/2)\sinh\theta}.
\]
\end{corollary}

For $t=2$ the weight is $(-1)^{m_2(T)}$, and if $\lambda=(c^a)$ is rectangular then
$\mathrm{SSYT}(\lambda,4)$ is in bijection with plane partitions in an $a\times(4-a)\times c$ box,
equivalently with lozenge tilings of a hexagon. Corollary \ref{cor:enum} is then a
$(-1)$-enumeration of those objects \emph{refined by a free parameter} $z$, which records the
difference between the numbers of entries $3$ and $4$. Setting $z=1$ recovers the plain signed
count, which by \eqref{eq:main} equals $\pm d_1d_2d_3/(2t^2)$: for $a=2$, $c=2$ it is $4$; for
$a=2$, $c=4$ it is $9$; for $a=3$, $c=3$ it is $-6$.

\begin{example}\label{ex:boxes}
Taking $\lambda=(c^{\,a})$ with $1\le a\le4$ and $t=2$, the shape has at most four rows, the tableaux
are the plane partitions in an $a\times(4-a)\times c$ box, and the closed form gives the signed
counts of Figure \ref{fig:signed}. Reading it off:
\[
\begin{array}{ll}
1\times3\times c: & \bigl\lfloor (c+2)^2/4\bigr\rfloor,\\[2pt]
2\times2\times c: & (c/2+1)^2 \text{ for } c \text{ even},\qquad 0 \text{ for } c \text{ odd},\\[2pt]
3\times1\times c: & (-1)^c\bigl\lfloor (c+2)^2/4\bigr\rfloor,\\[2pt]
4\times0\times c: & (-1)^c .
\end{array}
\]
The third row is the transpose of the first, which is why the magnitudes agree; the fourth is the
degenerate box, where $s_{(c^4)}(x)=(x_1x_2x_3x_4)^c$ and the alphabet has determinant $-1$. We do
not claim these particular counts as new --- the boxes are small and the sign here is the
specialization $x_2=-1$, not the complementation sign of Kuperberg's $(-1)$-enumerations --- and we
record them only to show what the theorem produces. What does appear to be new is that each of them
is the $z\mapsto1$ value of a one-parameter refinement with the same product formula.
\end{example}

The perfect squares in Example \ref{ex:boxes} are not an accident of small cases: they are what the
interval reading of \S\ref{sec:intervals} says a rectangle must produce.

\begin{proposition}\label{prop:rect}
Let $t=2$ and $\lambda=(c^{\,a})$ with $1\le a\le4$. Then one of $d_1,d_2,d_3$ equals $2$, and:
\begin{enumerate}
\item[\rm(i)] if $a=4$, all three equal $2$ for \emph{every} $c$ and the count is $(-1)^c$; this is
the full rectangle, where $s_{(c^4)}(x)=(x_1x_2x_3x_4)^{c}$ collapses to a monomial and the alphabet
contributes only its determinant;
\item[\rm(ii)] if $a\le3$ and $c$ is even, the other two are both equal to $c+2$, so the two
intervals have the \emph{same length} and the signed count is the perfect square
$\bigl(\tfrac{c+2}{2}\bigr)^2$;
\item[\rm(iii)] if $c$ is odd and $a=2$, the two intervals are \emph{concentric} and the count is $0$;
\item[\rm(iv)] if $c$ is odd and $a\in\{1,3\}$, the other two are $c+1$ and $c+3$, and the count is
$\pm\tfrac{(c+1)(c+3)}{4}=\pm\bigl\lfloor\tfrac{(c+2)^2}{4}\bigr\rfloor$.
\end{enumerate}
Case {\rm(i)} is not a degeneration of the interval triple but a collapse of the shape itself, and it
is the only case in which the parity of $c$ plays no role.
\end{proposition}

\begin{proof}
For $\lambda=(c^{\,a})$ the beta set consists of two blocks of consecutive integers, namely
$c+3,c+2,\dots,c+4-a$ and $3-a,2-a,\dots,0$; residues alternate along each block, which is what
forces the coincidences. Explicitly, for $a=1,2,3,4$ the beta set is $(c+3,2,1,0)$, $(c+3,c+2,1,0)$,
$(c+3,c+2,c+1,0)$ and $(c+3,c+2,c+1,c)$. Splitting each by parity gives, for $c$ even,
\[
\{c{+}3,1\},\{2,0\};\qquad \{c{+}3,1\},\{c{+}2,0\};\qquad \{c{+}3,c{+}1\},\{c{+}2,0\};\qquad
\{c{+}3,c{+}1\},\{c{+}2,c\},
\]
whence $(d_1,d_2,d_3)$ equals $(c{+}2,2,c{+}2)$, $(c{+}2,c{+}2,2)$, $(2,c{+}2,c{+}2)$ and
$(2,2,2)$ respectively; in each case the product over $2t^2=8$ is $\bigl(\tfrac{c+2}{2}\bigr)^2$,
resp. $1$. For $c$ odd and $a=2$ the split is $\{c{+}3,0\},\{c{+}2,1\}$, so
$d_3=|c{+}3+0-c{-}2-1|=0$: the intervals are concentric and Corollary \ref{cor:zero} applies. For
$c$ odd and $a=1,3$ the parity puts three beta numbers in one class, $\{c{+}3,2,0\}$ and
$\{c{+}3,c{+}1,0\}$ respectively, and the size-three profile gives $(c{+}1,2,c{+}3)$ and
$(2,c{+}1,c{+}3)$, with product over $8$ equal to $\tfrac{(c+1)(c+3)}{4}$. For $c$ odd and $a=4$ the
four beta numbers are again consecutive, so the split is $\{c{+}3,c{+}1\},\{c{+}2,c\}$ and the triple
is $(2,2,2)$ exactly as for $c$ even: at full height the answer does not see the parity of $c$, which
is why $a=4$ is separated out.
\end{proof}

So the square is the square of half the common interval length, the zeros are the concentric
configurations, the quarter-squares are the case where the two lengths differ by $2$, and at full
height all three intervals coincide, the shape becomes a monomial and only a sign survives. Which
pair of the three coincides rotates with $a$, which is why the same value appears in different
columns of Figure \ref{fig:signed}.

\begin{figure}[t]
\centering
\includegraphics[width=\textwidth]{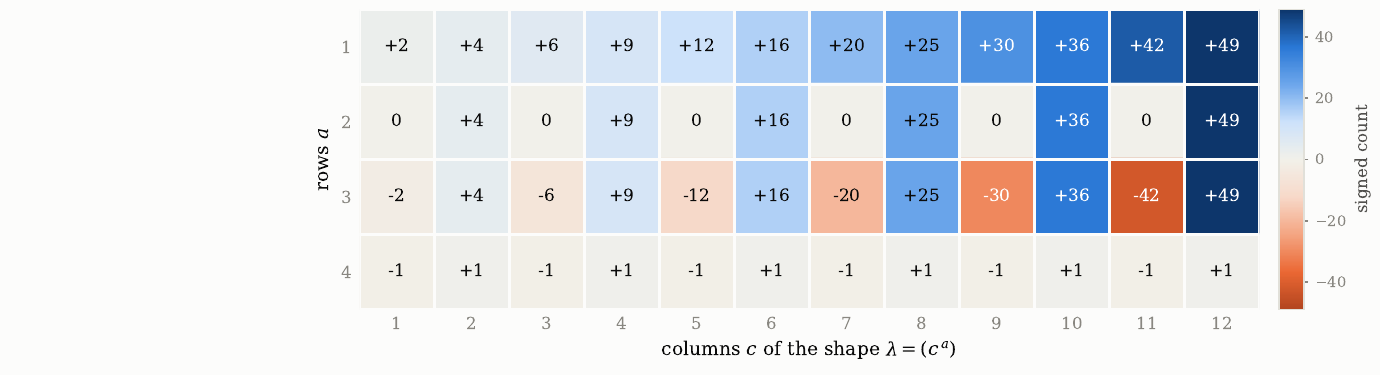}
\caption{The signed counts of Example \ref{ex:boxes}: the value of Corollary \ref{cor:enum} at
$z=1$ and $t=2$ for the rectangular shapes $\lambda=(c^{\,a})$, that is the $(-1)$-enumeration of
plane partitions in an $a\times(4-a)\times c$ box. Every entry is the closed form
$\varepsilon_\lambda\,d_1d_2d_3/(2t^2)$, cross-checked against a direct enumeration of the tableaux.
The row $a=2$ vanishes for odd $c$ and is a perfect square for even $c$.}
\label{fig:signed}
\end{figure}

The $(-1)$-enumeration literature computes such signed counts, and it is worth being exact about how,
because the difference here is narrow. Kuperberg \cite{Kup02} weights the objects by $\pm1$ directly.
Eisenk\"olbl \cite{Eis07} specializes $x_i=(-1)^i$ in a Schur function identity of the kind Stanley
\cite{Sta86} used --- an identity recorded there as also being \cite[Cor.~7.3]{IOTZ06} --- and gets the
$(-1)$-enumeration of self-complementary plane partitions with an odd side. Ciucu and
Krattenthaler evaluate the corresponding determinants. In each of these the alphabet is a number and
nothing is left free, which is the one point of difference: here the reciprocal pair keeps a variable
alive through the signed count. That is the alphabet's doing and not an argument of ours.

A refinement of a different kind is standard in the ribbon world, and it is worth saying which:
the LLT polynomials of Lascoux, Leclerc and Thibon \cite{LLT97} grade ribbon tableaux by a spin
statistic. That $q$ is an extra grading rather than a letter, and the object it grades is indexed by
a tuple of shapes; the parameter here is the alphabet's own and what carries it is a single
$s_\lambda$. The two refine different things, which is why we put the point narrowly: it is the
reciprocal pair, and not a statistic, that keeps a variable alive through the signed count.

A signed enumeration with a product formula asks for a sign-reversing involution, and the obstruction
is specific: the involution must preserve $m_{t+1}-m_{t+2}$, or the free parameter is not carried
across, and neither Bender--Knuth moves nor the single-cell toggle do. The construction in
\cite{KumariThesis} of a fixed-point-free sign-reversing involution on semistandard tableaux over a
$\pm$-paired alphabet, whose fixed set is the domino-tileable tableaux, does have the required
property in its $i=1$ instance: it touches only cells holding $1$ or $2$, hence preserves
$z^{m_3-m_4}$ and reverses $(-1)^{m_2}$. The triple product it would have to land on is the one of
Corollary \ref{cor:chars}, which at $t=2$ is a product of three $\mathfrak{sl}_2$ characters on the
nose, so the missing bijection is a Clebsch--Gordan statement. That is one half. The other is the
cancellation across the
splitting of the alphabet, and at $t=2$ it concerns a single one-parameter family. By
\eqref{eq:threestep} below,
\[
s_\lambda(1,-1,z,\zb)=\sum_{\ell(\nu)\le2}s_{\lambda/\nu}(1,-1)\,\chi_{\nu_1-\nu_2}(z),
\]
and an involution carrying the parameter must preserve $k=\nu_1-\nu_2$, since that is the index of
the character its term sits on. For fixed $k$ the admissible $\nu$ are exactly $(m+k,m)$ with
$m\ge0$, so that cancellation is a statement about one word.

The size of the terms is classical. Littlewood's theorem has a skew form --- $\varphi_t
s_{\lambda/\nu}$ vanishes unless $\lambda/\nu$ is tileable by $t$-ribbons, and is otherwise a ribbon
sign times a product over the quotient --- due to Macdonald \cite[p.~91]{Mac95}, and proved there by
the same Jacobi--Trudi route we take below; the case where $\nu$ is the $t$-core is Farahat's. At
$t=2$ on our alphabet it gives $\sigma_m\in\{0,\pm1\}$ at once. What is not classical is how the
sign moves with $m$, and that is what the involution needs.

\begin{figure}[t]
\centering
\includegraphics[width=\textwidth]{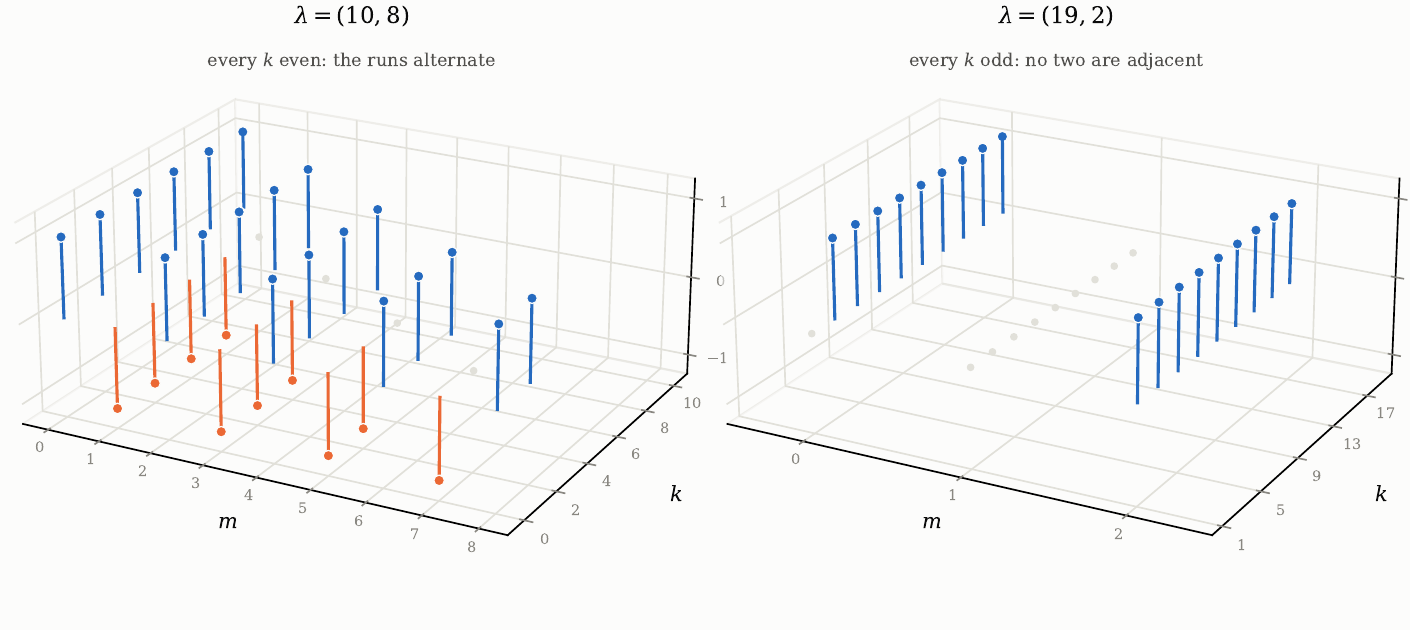}
\caption{Proposition \ref{rem:involution}, drawn. For each $k$ the admissible $\nu$ are the single
family $(m+k,m)$, and $\sigma_m=s_{\lambda/(m+k,m)}(1,-1)$ is plotted against $m$ and $k$; the grey
dots are the zeros. Every bar has height exactly $\pm1$, which is part (i): the terms are a set, so
the cancellation the enumeration asks for is a matching and not something carrying multiplicity. On
the left every $k$ is even and the colour alternates along each run, which is part (iii) and what
makes $m\leftrightarrow m+1$ sign-reversing; where a run ends, a zero ends it. On the right every
$k$ is odd and no two nonzero terms are adjacent, which is part (ii): the runs have length one and
there is nothing to cancel.}
\label{fig:runs}
\end{figure}

\begin{proposition}[the runs alternate]\label{rem:involution}
Let $t=2$, let $k\ge0$, and put $\sigma_m=s_{\lambda/(m+k,m)}(1,-1)$ for $m\ge0$. Then
\begin{enumerate}
\item[\rm(i)] $\sigma_m\in\{0,\pm1\}$ --- this part is Macdonald's, not ours;
\item[\rm(ii)] if $k$ is odd, no two consecutive $\sigma_m$ are nonzero;
\item[\rm(iii)] if $\sigma_m$ and $\sigma_{m+1}$ are both nonzero then $\sigma_{m+1}=-\sigma_m$.
\end{enumerate}
Consequently, pairing $m$ with $m+1$ inside each maximal run of consecutive nonzero $\sigma$ is a
sign-reversing involution, and $\sum_m\sigma_m$ is the signed count of the runs of odd length.
\end{proposition}

The three parts are visible at once in Figure \ref{fig:runs}: the bar heights are part {\rm(i)}, the
gaps on the odd panel are part {\rm(ii)}, and the alternation of colour along each run is part
{\rm(iii)}.

\begin{proof}
Write $\beta_i=\lambda_i-i$ and $c_j=\nu_j-j$, and take the Jacobi--Trudi matrix of
$s_{\lambda/\nu}$ in $n\ge\ell(\lambda)$ rows. Since $h_j(1,-1)=1$ for $j\ge0$ even and $0$
otherwise,
\begin{equation}\label{eq:betaJT}
M_{ij}=\bigl[\beta_i\ge c_j\bigr]\cdot\bigl[\beta_i\equiv c_j \bmod 2\bigr],\qquad
\sigma_m=\det M .
\end{equation}
The second factor says that column $j$ is supported on the rows of one parity class, the class of
$c_j$; and since $\beta_1>\dots>\beta_n$, inside that class the first factor cuts an \emph{initial
segment}. Order the rows by parity class: $M$ becomes block diagonal, so it is singular unless each
class carries as many columns as it has rows, and inside a class it is the staircase
$\bigl[\,\text{the rank of }i\text{ in its class}\le\ell_j\bigr]$ of the segment lengths $\ell_j$. That staircase is singular if two
lengths agree or one is $0$, and otherwise the lengths are a permutation of $1,\dots,e$ and its
determinant is the sign of that permutation. This gives (i), which is \cite[p.~91]{Mac95} in this
case; we re-derive it because the $\beta$-form \eqref{eq:betaJT} is what (ii) and (iii) rest on.

For $\nu=(m+k,m)$ we have $c_1=m+k-1$, $c_2=m-2$ and $c_j=-j$ for $j\ge3$. So $c_1-c_2=k+1$, and
passing from $m$ to $m+1$ raises $c_1$ and $c_2$ by one and leaves every other column alone: exactly
the first two columns change parity class.

If $k$ is odd then $c_1\equiv c_2$, so columns $1$ and $2$ lie in the same class and both leave it.
That class's column count changes by two while its row count does not, so at most one of $\sigma_m$,
$\sigma_{m+1}$ can be nonzero, which is (ii).

If $k$ is even then $c_1\not\equiv c_2$ and the two columns exchange classes. Suppose both are
nonzero. In the class of $c_1$ the columns are column $1$ together with a set of columns $j\ge3$
that does not move; by (i) the lengths are a permutation of $1,\dots,e$ before and after, and the
fixed columns contribute the same lengths to both, so the length column $2$ brings in equals the one
column $1$ took away --- and symmetrically in the other class. Each class therefore sees the same
lengths in the same order of columns, and both block determinants are unchanged. What does change is
the shuffle separating the two classes of columns: columns $1$ and $2$ are adjacent and exchange
membership, which alters the number of inversions of that shuffle by exactly one. Hence
$\sigma_{m+1}=-\sigma_m$, which is (iii).
\end{proof}

\begin{remark}[the sign here is the sign of Theorem \ref{thm:main}]\label{rem:onesign}
Write $\sigma_\nu$ for the permutation of \S\ref{sec:background} attached to a partition $\nu$: the
one sorting $\beta$ by residue class, which Proposition \ref{rem:sign} places inside
$\varepsilon_\lambda$. The terms of Proposition \ref{rem:involution} are values of that same sign.
Indeed $\sigma_m$ vanishes unless $\core_2(\lambda)=\core_2(\nu)$ and both components of the skew
$2$-quotient are horizontal strips, and equals $\sgn(\sigma_\lambda)\sgn(\sigma_\nu)$ when it does
not --- the identity $\sgn(\sigma_\lambda)\sgn(\sigma_\nu)=(-1)^{\sum_B\mathrm{ht}(B)}$ being the
ribbon sign of \cite[p.~91]{Mac95} written as a sorting sign, which is
\cite[Lemma 6.3]{KumariThesis}. Parts {\rm(ii)} and {\rm(iii)} are
then one statement about $\sigma_\nu$ alone: as $\nu=(m+k,m)$ moves to $(m+k+1,m+1)$,
$\sgn(\sigma_\nu)$ reverses for $k$ even and is constant for $k$ odd. So the permutation that
carries the sign of the main theorem is the permutation that makes the involution of this section
alternate; the two halves of the paper are signed by one object.
\end{remark}

The pairing is a bijection of tableaux and not only of terms, which is what carries the parameter. A
semistandard tableau of shape $(a,b)$ in the letters $3,4$ has second row $4^{\,b}$ and first row
$3^{\,p}4^{\,a-p}$ with $b\le p\le a$, of weight $z^{2p-a-b}$; adjoining a first column
$\left(\begin{smallmatrix}3\\4\end{smallmatrix}\right)$ sends it to the tableau of shape
$(a+1,b+1)$ with $p+1$ in place of $p$, and $2(p+1)-(a+1)-(b+1)=2p-a-b$, so $m_3-m_4$ is preserved
exactly. Composing that with Proposition \ref{rem:involution} pairs the terms of Corollary
\ref{cor:enum} in the way the enumeration asks for.

One half is not ours, and we say which. Proposition \ref{rem:involution} disposes of the cancellation
across $\nu$; the cancellation \emph{inside} each $s_{\lambda/\nu}(1,-1)$ is
\cite{KumariThesis}, and the two fit together. There, at $n=1$, the alphabet is $(x,-x)$ and the
weight of a two-letter filling is $(-1)^{c_2}$, which is $\sigma_m$; the involution
\cite[Lemma 2.18]{KumariThesis} is fixed-point-free on the fillings that are \emph{not coverable},
so the sum reduces to the coverable ones, which \cite[Lemma 2.16]{KumariThesis} matches with domino
tableaux. What makes the reduction usable here is \cite[Lemma 2.14]{KumariThesis}: the parity of the
number of vertical dominoes does not depend on the tableau, so every survivor carries the same sign
and $\sigma_m=\pm\#\{\text{coverable fillings}\}$. With part (i) that forces the count to be at most
one, so the fixed set is a single filling when $\sigma_m\ne0$ and empty otherwise --- checked
directly over the $1134$ skew shapes with at most four rows and $|\lambda|\le10$, where the two
counts agree without exception.

At $t=2$ the involution the enumeration asks for therefore exists: cancel inside each $\nu$ by
\cite[Lemma 2.18]{KumariThesis}, pair what is left across $\nu$ by Proposition \ref{rem:involution},
and carry the parameter by the column $\left(\begin{smallmatrix}3\\4\end{smallmatrix}\right)$. Its
fixed points are one filling for each run of odd length, and Corollary \ref{cor:enum} counts them.
The construction is Kumari's; the assembly is ours.

\subsection{Beyond \texorpdfstring{$t=2$}{t=2}: where such an involution can live}\label{sec:lift}

The assembly above is a $t=2$ statement, and Proposition \ref{rem:involution} is why: it is a
statement about one word, and that word is the $t=2$ one. What follows is not that construction for
general $t$. It is the smaller question of \emph{where} an involution of the required kind could
live at all --- which has an answer, and a negative one for the obvious candidate.

Fix $t$ and $\lambda$. Combining the branching identity \eqref{eq:composes} of Remark
\ref{rem:composes} with the two-letter expansion of $s_{\lambda/\mu}(z,\zb)$ gives
\begin{equation}\label{eq:Jlambda}
\Phi_t(\lambda;z)\;=\;\sum_{(\mu,\nu)\in J_\lambda}\sgn_t(\mu)\;z^{\,2|\nu|-|\mu|-|\lambda|},
\end{equation}
where $J_\lambda$ is the set of pairs $\mu\subseteq\nu\subseteq\lambda$ such that
$\core_t(\mu)=\varnothing$, every component of $\quot_t(\mu)$ has at most one row, and both
$\nu/\mu$ and $\lambda/\nu$ are horizontal strips. That is the term set the enumeration presents,
and the involution it asks for is one on $J_\lambda$: weight-preserving, sign-reversing off its
fixed set, and with the transversals of Lemma \ref{lem:L3} as the fixed set. The first two are
available; the third is not, and the obstruction is a count rather than a lack of ingenuity.

\begin{lemma}[the standard bound, weight space by weight space]\label{lem:fixbound}
Let $X$ be a finite signed set graded by weight, and write $n_\pm(e)$ for the number of its elements
of weight $e$ and sign $\pm$. If $\iota$ is an involution on $X$ preserving the weight and reversing
the sign off its fixed set, then
\[
\bigl|\Fix(\iota)\bigr|\;\ge\;\sum_e\bigl|n_+(e)-n_-(e)\bigr|,
\]
with equality if and only if, in every weight space, all the fixed points carry the same sign.
\end{lemma}

\noindent Nothing here is new. That a sign-reversing involution reduces a signed sum to its fixed
points is the standard device, and the bound is that identity applied one weight space at a time
together with the triangle inequality; we state it because the rest of the subsection is about the
number it produces.

\begin{proof}
Fix $e$ and let $F_\pm$ be the number of fixed points of weight $e$ and sign $\pm$. Each $2$-cycle of
$\iota$ lies inside a weight space and removes one element of each sign, so
$n_+(e)-n_-(e)=F_+-F_-$, whence $F_++F_-\ge|F_+-F_-|=|n_+(e)-n_-(e)|$, with equality exactly when
$F_+$ or $F_-$ vanishes. Summing over $e$ gives both statements.
\end{proof}

Applied to $J_\lambda$ through \eqref{eq:Jlambda}, the bound reads
$|\Fix(\iota)|\ge\|\Phi_t(\lambda;z)\|_1$, the sum of the absolute values of the coefficients. So
$\|\Phi_t(\lambda;z)\|_1$ is the \emph{minimum} number of fixed points of any such involution on
$J_\lambda$, and the minimum is attained: pairing off $+$ with $-$ arbitrarily inside each weight
space leaves exactly that many. It is not a quantity a construction can economise on --- the first
crossing, the last, or any other.

\begin{proposition}[$J_\lambda$ is the wrong set]\label{prop:wrongset}
By Lemma \ref{lem:L3} the numerator $D_t(z)\Phi_t(\lambda;z)$, with
$D_t(z)=(z^t-1)(z^{-t}-1)(z-z^{-1})$, is a signed sum of $|\mathcal U|\le4$ terms
$f(\beta_{j_1}-\beta_{j_2})$, hence of at most eight monomials. But $\|\Phi_t(\lambda;z)\|_1$ is not
bounded by eight. At $t=2$ and $\lambda=(7,4)$ the beta set is $(10,6,1,0)$, the profile is
size-three with triple $(d_1,d_2,d_3)=(4,6,10)$, and by Corollary \ref{cor:chars}
$\Phi_2=\pm\chi_1\chi_2\chi_4$, a product of characters with all coefficients of one sign, so
\[
\|\Phi_2\|_1=\chi_1(1)\chi_2(1)\chi_4(1)=2\cdot3\cdot5=30 .
\]
That single shape settles it, and the sweep says how common it is: over $t=2,3$ and
$\lambda_1\le7$ the norm reaches $30$ and $21$, exceeding eight in $128$ of the $329$ shapes and in
$216$ of the $791$ \stverif. For those $\lambda$ no involution on $J_\lambda$ of the kind above can
have the objects of Lemma \ref{lem:L3} as its fixed set.
\end{proposition}

The two counts are of objects on opposite sides of the denominator, and that is the whole diagnosis.
$J_\lambda$ indexes the quotient; the transversals of Lemma \ref{lem:L3} index the numerator; and an
involution on $J_\lambda$ would have to conceal exactly the factor $D_t$ that separates them. The
remedy is to multiply the index set by that factor instead of dividing the objects by it, and to
keep the factor \emph{factored}, so that each term retains its provenance.

\begin{definition}\label{def:lift}
Let $\Omega_t$ be the signed set of the eight choices of one term from each factor of $D_t$,
\[
(z^t-1)\mapsto(t,+1)\ \text{or}\ (0,-1),\qquad
(z^{-t}-1)\mapsto(-t,+1)\ \text{or}\ (0,-1),\qquad
(z-z^{-1})\mapsto(1,+1)\ \text{or}\ (-1,-1),
\]
each pair being a weight and a sign, and write $d(\omega)$ and $\sgn(\omega)$ for them. The
\emph{numerator lift} is $\widehat J_\lambda:=J_\lambda\times\Omega_t$, with
\[
\wt\bigl((\mu,\nu),\omega\bigr)=2|\nu|-|\mu|-|\lambda|+d(\omega),
\qquad
\sgn\bigl((\mu,\nu),\omega\bigr)=\sgn_t(\mu)\,\sgn(\omega).
\]
\end{definition}

\begin{lemma}\label{lem:liftid}
$\sum_{x\in\widehat J_\lambda}\sgn(x)z^{\wt(x)}=D_t(z)\,\Phi_t(\lambda;z)$, and
$|\widehat J_\lambda|=8|J_\lambda|$ with the two signs equinumerous.
\end{lemma}

\begin{proof}
A signed sum over a product of signed sets is the product of the sums; the sum over $\Omega_t$ is
$D_t(z)$, one factor at a time, and the sum over $J_\lambda$ is $\Phi_t(\lambda;z)$ by
\eqref{eq:Jlambda}. Each factor of $D_t$ contributes one sign of each kind, so the eight elements of
$\Omega_t$ split four and four.
\end{proof}

\begin{proposition}[on the lift the count fits, and exactly]\label{prop:liftcount}
Write $\mathcal L(\lambda)$ for the \emph{signed} multiset obtained by expanding the terms of
Lemma \ref{lem:L3}, each monomial carrying the product of the sign $\varepsilon_U$ of its Laplace
term with the sign of the chosen term of $f(\beta_{j_1}-\beta_{j_2})$, so that
$|\mathcal L(\lambda)|=2|\mathcal U(\lambda)|$. If $\Phi_t(\lambda;z)\ne0$ then
\[
\bigl\|D_t(z)\,\Phi_t(\lambda;z)\bigr\|_1\;=\;2\,|\mathcal U(\lambda)| .
\]
\end{proposition}

\begin{proof}
Since $(z^t-1)(z^{-t}-1)=-f(t/2)^2$ and $D_t=(z^t-1)(z^{-t}-1)(z-z^{-1})=-f(t/2)^2f(1)$, equation
\eqref{eq:main} gives
\[
D_t(z)\,\Phi_t(\lambda;z)\;=\;-\varepsilon_\lambda\,f(d_1/2)\,f(d_2/2)\,f(d_3/2).
\]
Expanding, the eight monomials carry exponents $\tfrac12(\pm d_1\pm d_2\pm d_3)$ and sign the
product of the three choices. Two of them cancel exactly when their exponents agree and their signs
differ, and comparing the two cases --- the choices differing in all three signs, or in one --- that
happens only if some $d_i$ vanishes or one of $d_1,d_2,d_3$ is the sum of the other two. Monomials of
equal exponent and equal sign reinforce, which does not change the $\ell^1$ norm.

In the two-class profile $|\mathcal U|=4$ by Lemma \ref{lem:L3}, and nothing cancels. Indeed
\[
2(a_1-b_1)=\tilde d_3+d_1-d_2,\quad
2(a_1-b_2)=\tilde d_3+d_1+d_2,\quad
2(a_2-b_1)=\tilde d_3-d_1-d_2,\quad
2(a_2-b_2)=\tilde d_3-d_1+d_2,
\]
so each of the four sums $\pm d_1\pm d_2+\tilde d_3$ vanishes only if some $a_i$ equals some $b_j$,
impossible because $A$ and $B$ lie in different residue classes; the same four with $-\tilde d_3$ are
their negatives. And $d_1,d_2\ne0$ because each class carries two distinct parts, while $d_3\ne0$ is
the hypothesis $\Phi_t\ne0$ of Theorem \ref{thm:main}(ii). Hence the norm is $8=2|\mathcal U|$.

In the size-three profile $|\mathcal U|=3$, and $A=\{p,q\}$, $B=\{q,r\}$ with $p>q>r$, so
$d_3=p-r=d_1+d_2$. Exactly one pair cancels, the two monomials of exponent $0$, and
\[
f(d_1/2)f(d_2/2)f(d_3/2)=z^{d_1+d_2}-z^{d_1}-z^{d_2}+z^{-d_2}+z^{-d_1}-z^{-d_1-d_2},
\]
of norm $6=2|\mathcal U|$ whether or not $d_1=d_2$. The six is not an accident of the expansion.
Writing $x=z^{d_1/2}$ and $y=z^{d_2/2}$, the product is $(x-x^{-1})(y-y^{-1})(xy-x^{-1}y^{-1})$, the
product over the three positive roots of $A_2$; the Weyl denominator formula turns it into the
alternating sum over $W(A_2)=S_3$, and $6=|W(A_2)|$; that identity is the standard Weyl
denominator formula for $A_2$.
\end{proof}

The bound $\|D_t\Phi\|_1\le2|\mathcal U|$ also follows from Lemma \ref{lem:L3} and the triangle
inequality alone, which is how we first saw it; the content above is that it is an equality. Hence,
by Lemma \ref{lem:fixbound} applied to $\widehat J_\lambda$, an involution there of the same kind
leaves at least $2|\mathcal U(\lambda)|$ fixed points --- exactly the number of monomials Lemma
\ref{lem:L3} supplies, with no slack on either side. Over $t=2,3$ and $\lambda_1\le7$ the equality
was measured before it was proved, in $309$ of the $329$ shapes and in all $791$, the $20$ exceptions
being the concentric ones, where $\Phi_t\equiv0$ \stverif.

That is the half of the question the lift settles, and Figure \ref{fig:lift} is it. On $J_\lambda$ the requirement is arithmetically
impossible for a third of the shapes in range; on $\widehat J_\lambda$ it is impossible for none, and
not because the bound has become generous: for every nonzero $\Phi_t$ it has become tight, the
vanishing cases being the ones where slack remains. What the lift does not do is
exhibit the involution, and we return to that in Problem \ref{prob:toggle}.

\begin{figure}[t]
\centering
\includegraphics[width=\textwidth]{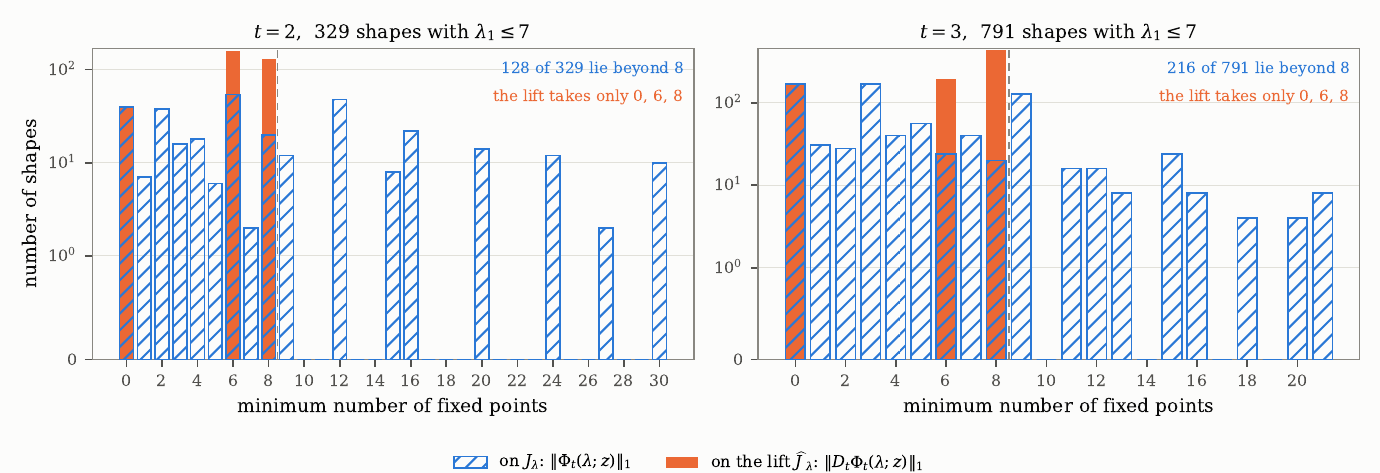}
\caption{What the lift does to the count. By Lemma \ref{lem:fixbound} an involution of the required
kind leaves at least $\|\Phi_t(\lambda;z)\|_1$ fixed points on $J_\lambda$ and at least
$\|D_t\Phi_t(\lambda;z)\|_1$ on $\widehat J_\lambda$; both are drawn as distributions over the same
shapes, on a logarithmic count. The dashed line sits at eight, the most monomials Lemma
\ref{lem:L3} can supply, so what to read is the horizontal extent and not the shape of either
histogram: the unlifted requirement runs past the line --- to $30$ at $t=2$ and $21$ at $t=3$, in
$128$ and $216$ of the shapes --- and the lifted one cannot, taking only the values $0$, $6$ and $8$.
For $\Phi_t\ne0$ the lifted minimum is $6$ or $8$, and Proposition \ref{prop:liftcount} says it is
exactly $2|\mathcal U(\lambda)|$; the value $0$ corresponds to the vanishing cases, where that
proposition does not apply. On the nonvanishing locus the lift does not move the requirement: it
collapses it exactly onto the Laplace count. The blue is drawn hollow over the orange so that
both are visible where they overlap, and the population excludes the empty partition, as the archived
run does. Computed by \texttt{anc/fig\_lift.py}, which rebuilds $J_\lambda$ from
\eqref{eq:Jlambda} and shares no code with the scripts behind \S\ref{sec:verif}.}
\label{fig:lift}
\end{figure}

It does, however, cut the problem into pieces small enough to attack one at a time.

\begin{corollary}[the requirement is local]\label{cor:local}
Suppose $\Phi_t(\lambda;z)\ne0$, so that Proposition \ref{prop:liftcount} applies. Then for every weight
$e$,
\[
\bigl|[z^e]\,D_t\Phi_t(\lambda;z)\bigr|
=\#\bigl\{\text{monomials of }\mathcal L(\lambda)\text{ of weight }e\bigr\},
\]
and those monomials all carry the same sign. So an involution on $\widehat J_\lambda$ may be sought
one weight space at a time: in the space of weight $e$ it must fix exactly the monomials of
$\mathcal L(\lambda)$ of that weight, and pair off everything else.
\end{corollary}

\begin{proof}
Termwise $|[z^e]D_t\Phi|\le\#\{\cdots\}$ by the triangle inequality, and by Proposition \ref{prop:liftcount}
the two sides have equal sums over $e$; so they agree for every $e$. Equality in a triangle inequality forces the
summands to share a sign.
\end{proof}

Measured over the same range, the local form holds in $1696$ of $1696$ unbalanced weight spaces at
$t=2$ and $4074$ of $4074$ at $t=3$ \stverif; the $80$ Laplace weights that carry no imbalance at
$t=2$ are the $20$ shapes with $\Phi_t\equiv0$ against their four weights, and at $t=3$ there are
none. The contrast is what makes the point: on $J_\lambda$ unlifted the coefficients of $\Phi$ reach
$7$ and $3$ in absolute value, so there the requirement is not local either \stverif.

\begin{remark}[what the skeleton already forces]\label{rem:skeleton}
Corollary \ref{cor:local} leaves a construction to make, but not everywhere: a weight space carrying
a single element admits no choice, and there the map $\mathcal L(\lambda)\to\widehat J_\lambda$ can
be read off rather than designed. Such spaces are $398$ of the $1696$ at $t=2$ and $1176$ of the
$4074$ at $t=3$, and on them two things hold without exception \stverif. First, the sign of the
forced element agrees with the sign of the Laplace monomial of its weight --- $398$ of $398$ and
$1176$ of $1176$ --- which is necessary for any involution meeting the requirement. Second, its
$\omega$ takes exactly one of the two torsion factors, never both and never neither: four of the
eight elements of $\Omega_t$ are excluded outright, again without exception. Three natural readings
fail on the same data, and we record them because each would have been a plausible ansatz: that
$\omega$ is one of the two elements of extreme weight $\pm(t+1)$ holds at $t=2$ but fails in $70$ of
$1176$ cases at $t=3$; that $\beta(\mu)$ is $\beta(\lambda)$ with the two columns of the transversal
removed fails in $384$ of $398$ and $1092$ of $1176$; and that $\nu$ is $\mu$ or $\lambda$ fails in
$260$ of $398$ and $860$ of $1176$, $\nu$ being strictly intermediate there.
\end{remark}

The construction of this subsection arose in the course of AI-assisted exploration, and the paper
would not contain it otherwise; the disclosure in \S\ref{sec:disclosure} says what that means here.
The choice between the two readings --- that the involution should live not on
$J_\lambda$ but on an index set for the \emph{numerator} --- came from there, and so did the reason
for it:
an involution on $J_\lambda$ would have to hide exactly the factor that separates the quotient from
the numerator. So is $\widehat J_\lambda$ itself, together with the instruction not to expand $D_t$
prematurely but to keep it factored, one signed choice per factor, so that terms of coincident weight
remain distinct as provenance-labelled objects. And so is the correction that makes Lemma \ref{lem:fixbound} true: we
had written the number of fixed points as an equality, and it is an inequality, with equality exactly
under the condition stated there. So, finally, is the observation that a transversal of Lemma
\ref{lem:L3} carries a $2\times2$ minor and therefore \emph{two} monomials, which is what turns a
count of transversals into Proposition \ref{prop:liftcount}. What is ours is the measurement, the
locality of Corollary \ref{cor:local}, and the forced skeleton. We are grateful for all of it.

We do not know of a name for the construction. Multiplying a signed set by a signed set representing
the terms of a denominator is the standard device for turning an identity between a quotient and a
sum into one between two signed sets, and composing the resulting involutions is the involution
principle of Garsia and Milne \cite{GM81}; what is particular here is that $D_t$ is kept factored, so
that provenance survives. The path formalism in which those choices become steps is Fulmek's
\cite{Fulmek}, where Cauchy--Binet and Laplace --- the two sides that would have to be matched ---
are obtained bijectively in the same picture. One control says which part of the construction is
doing the work. Collecting the eight provenance-labelled choices into the ordinary Laurent polynomial
$D_t$ merges those of equal weight and produces coefficients of modulus greater than one --- at
$t=2$, $D_2=-z^3+3z-3z^{-1}+z^{-3}$ --- so the collected object is not a signed set of six elements
with signs $\pm1$ at all. Collecting equal weights leaves the Laurent-polynomial identity of Lemma
\ref{lem:liftid} unchanged, and forgets the eight provenance-labelled choices that define
$\Omega_t$: the polynomial controls the identity, the factored lift controls the provenance. Two decoys confirm
that the signed set is the right one: dropping the factor $(z-z^{-1})$ breaks the identity in $289$
of $329$ shapes and in $623$ of $791$, and giving every element of $\Omega_t$ the sign $+1$ breaks it
in the same numbers.

\section{The factorization is isolated}\label{sec:sharp}

A factorization is only a result if it is not a specialization of a known one, and it is only a
\emph{sharp} result if the alphabet cannot be deformed. We do not prove that in general. What we do
is try four deformations of \eqref{eq:object}, and all four destroy the product.

The evidence in this section is computational, and we label it as such: none of (D1)--(D4) below is
a theorem, and in particular we prove no no-go statement. What we assert is that the identity
\eqref{eq:main} is false for each deformed alphabet, which a single counterexample settles, and we
give counts to indicate how far from true it is.

\begin{itemize}
\item[\rm(D1)] \emph{More reciprocal pairs.} For $s_\lambda(\zt,z_1,z_1^{-1},\dots,z_r,z_r^{-1})$
the product law of Theorem \ref{thm:main} already fails at $r=2$: in the computed range typical
values carry one large irreducible factor, and the only factors observed to split off are binomials
in the rotated variables $z_1z_2$ and $z_1z_2^{-1}$, that is,
the couplings between the pairs' $u(1)$ charges. What does survive is the \emph{zero locus}: at
$t=2$ it survives every $r$ and Theorem \ref{conj:crit} determines it completely, while for
arbitrary $t$ and $r$ Theorem \ref{thm:suff} gives a uniform sufficient criterion and Conjecture
\ref{conj:general} is the converse. On this evidence the product looks like a rank-one accident and
the vanishing like the part that generalizes; the first half of that sentence is Conjecture
\ref{conj:rank2} and only the second is a theorem.
\item[\rm(D2)] \emph{More orbit variables.} For $s_\lambda(X,\zeta X,\dots,\zeta^{t-1}X,z,z^{-1})$
the product law fails already at $|X|=n=2$. At $t=2$, $n=2$ the value for $\lambda=(2)$ is
$(x^2z^2+y^2z^2+z^4+z^2+1)/z^2$, irreducible over $\mathbb Q$; for $\lambda=(2,2)$ it is
$(x^2+y^2+z^2)(x^2z^2+y^2z^2+1)/z^2$, two factors rather than three; and for $\lambda=(3,1)$ it is
again irreducible. By contrast, with the pair deleted the same alphabet obeys \cite[Theorem
2.3]{AK25}: $s_{(2)}(x,y,-x,-y)=x^2+y^2$ and $s_{(2,2)}(x,y,-x,-y)=(x^2+y^2)^2$.
\item[\rm(D3)] \emph{The orbit replaced by a coset.} For $\zeta'\zt$ with $\zeta'=e^{\pi i/t}$, the
odd powers of a primitive $2t$-th root of unity --- the twist of \cite[\S3]{AK25}, and for $t$ even
the full set of primitive $2t$-th roots studied in \cite{HI24} --- the identity fails, vanishing
criterion included.
\item[\rm(D4)] \emph{The reciprocal pair replaced by a free pair.} $s_\lambda(\zt,x,y)$ was
irreducible in every case we computed. We make no claim beyond that range.
\end{itemize}

Counts for (D3) and (D4), over $t=2,\dots,6$ and $|\lambda|\le10$: the identity holds in all $600$
cases for \eqref{eq:object}, in $383$ of $600$ for (D3), and in $200$ of $600$ for (D4). The test is
therefore capable of failing, which is the point of quoting the first row.

One reading of the proof accounts for all four. Splitting the alphabet and applying Theorem
\ref{thm:LR} to the frozen half gives, for any free alphabet $B$,
\begin{equation}\label{eq:threestep}
s_\lambda(\zt\cup B)=\sum_{\nu}s_{\lambda/\nu}(\zt)\,s_\nu(B)=\sum_\nu\varepsilon_\nu\,s_\nu(B),
\qquad \varepsilon_\nu\in\{0,\pm1\},
\end{equation}
the coefficients being ribbon signs: that the skew evaluation $s_{\lambda/\nu}(\zt)$ is $0$ or a
ribbon sign is the skew form of Theorem \ref{thm:LR}, due to Macdonald \cite[p.~91]{Mac95}. Now
$B=(z,z^{-1})$ has exactly two letters and they are
reciprocal, so $s_\nu(B)=\chi_{\nu_1-\nu_2}(z)$ is a \emph{single} $\mathfrak{sl}_2$ character
depending on one difference, and the short signed sum collapses. Step one requires the frozen half
to be a full orbit, since Theorem \ref{thm:LR} needs the rows to meet every residue class exactly;
a coset does not, which is (D3). Step two requires the free half to be exactly one reciprocal pair:
a free pair has $\det=xy$ and couples to both central charges, which is (D4), and enlarging the free
half --- by more pairs, or by more orbit variables, which enlarges $\nu$ --- makes $s_\nu(B)$ cease
to be a single character and the sum cease to collapse, which is (D1) and (D2). \emph{The four
failures are one condition seen from four sides.}

\begin{remark}
Equivalently: the reciprocal pair contributes exactly \emph{one} coupling, the cross term $d_3$,
because $\det=1$ prevents it from seeing the two central charges separately. This is why there are
three factors and not two, and Ayyer--Kumari's factorizations, which are products over quotient
components with no coupling at all, are the case where that contribution is absent.
\end{remark}

\begin{remark}[a representation-theoretic reading]\label{rem:twining}
The alphabet \eqref{eq:object} is closed under inversion, hence a torus element of $O(N)$, and its
determinant is $(-1)^{t+1}$. For even $t$ it therefore lies in the non-identity component, where a
character is a \emph{twining} character; the identity computing such characters by folding is due to
Jantzen \cite{Jantzen} and, in the form closest to this one, to Kumar--Lusztig--Prasad \cite{KLP},
the relevant folded algebra being the orbit Lie algebra. The passage from their statement to ours
is not a single step, and it is worth naming the two: restricting the $GL_N$-representation to
$O(N)$ is a branching problem, and once it is decomposed into irreducible constituents the twining
formula applies to each of them on the non-identity component; Clifford theory is what governs the
$SO(N)/O(N)$ passage inside that, not the branching. Our restriction is reducible --- for a general
$\lambda$ it is a $\ZZ$-combination of symplectic characters and not one of them, as the last remark
of \S\ref{sec:zeros} records --- so what one reads off is a sum of twining characters rather than a
single one. We record this because Ciucu--Krattenthaler
\cite{CK09} wrote that they could not offer a representation-theoretic interpretation of their
factorizations, and \cite{AB19} and \cite{AK25} record that it is still missing; the reading here is
of that kind. The mechanism is entirely theirs, and the proof above does not use it. Lemma
\ref{lem:AtoSp} makes the reading concrete for the alphabet \eqref{eq:psir}, where the folded algebra
is $C_r$ and the twining character is a symplectic character on the nose.
\end{remark}

\section{The zero locus survives every $r$}\label{sec:zeros}

This section is past the factorization, and reads differently from the seven before it. Those prove
statements about an identity; this one asks what is left when the identity is gone, and the answer is
a locus rather than a formula. It has two halves. The first fixes $t=2$ and proves both directions
for every $r$ and every $\lambda$, the converse resting on one external input; that is Theorem
\ref{conj:crit}, and the route to it is what takes the section its length. The second lets $t$ go
free again and asks the same question of the whole family: there we prove one implication for every
$t$ and every $r$ --- Theorem \ref{thm:suff}, by an involution that pairs off every term of the
expansion --- record what the other direction forces, and state as Conjecture \ref{conj:general} that
the two conditions are the same one. The change of register is the subject changing, not the
standard.

By (D1) the product of Theorem \ref{thm:main} does not survive a second reciprocal pair. Its
\emph{zero locus} does, and for every $r$ it has a closed form. Write
\begin{equation}\label{eq:psir}
\Psi_r(\lambda)\;:=\;s_\lambda(1,-1,z_1,\bar z_1,\dots,z_r,\bar z_r),\qquad N=2r+2,
\end{equation}
so that $\Psi_1$ is \eqref{eq:object} at $t=2$. Throughout this section $\lambda$ is padded to
exactly $N$ parts, $\beta=\beta(\lambda)$ is its beta set, and $E$ and $O$ are its even and odd
entries. Call $\lambda$ \emph{self-complementary of width $w$} if $\lambda_i+\lambda_{N+1-i}=w$ for
every $i$, in which case $w=\lambda_1+\lambda_N$. The two conditions the section is about are
\begin{equation}\label{eq:branches}
\boxed{\;\;
\underbrace{|E|\in\{0,N\}}_{\text{branch (a)}}
\qquad\qquad
\underbrace{\lambda_i+\lambda_{N+1-i}=w\ \ \text{for all }i,\ \ w\ \text{odd}}_{\text{branch (b)}}
\;\;}
\end{equation}
Branch (a) says the beta set has constant parity; branch (b) says $\lambda$ is self-complementary of
odd width, the width being $w=\lambda_1+\lambda_N$.

We state the four claims separately, because they are not proved to the same extent and a single
biconditional would hide that.

\begin{theorem}[the two branches]\label{thm:zeros}
For every $r\ge1$ and every $\lambda\in\PP_N$: if branch {\rm(a)} or branch {\rm(b)} holds, then
$\Psi_r(\lambda)=0$.
\end{theorem}

\begin{corollary}\label{cor:sizes}
If $\lambda$ is self-complementary of odd width $w$ then $|\lambda|=w(r+1)$. In particular branch
{\rm(b)} occurs only for $|\lambda|\equiv r+1 \pmod{2(r+1)}$, and a size outside that progression
needs no further test.
\end{corollary}

\begin{proof}
Summing $\lambda_i+\lambda_{N+1-i}=w$ over $i$ counts $|\lambda|$ twice, so $2|\lambda|=wN=w(2r+2)$.
\end{proof}

The progression is visible in Figure \ref{fig:zeros}, and it is worth recording that the analogous
statement for the whole locus is false: branch (a) is not confined to it, which is why the condition
has to be read off branch (b) alone.

\begin{figure}[!ht]
\centering
\includegraphics[width=\textwidth]{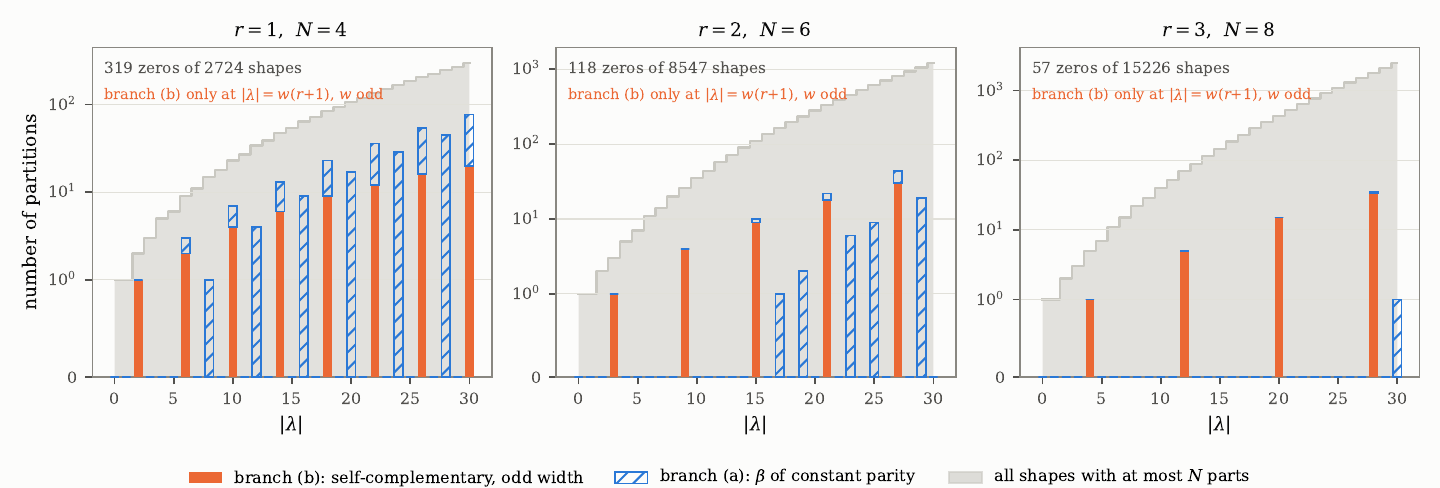}
\caption{The zero locus of Theorem \ref{conj:crit} --- the biconditional, of which Theorem
\ref{thm:zeros} is the sufficient half --- counted rather than plotted. For each size
$|\lambda|$ the bars give the number of partitions with at most $N$ parts that vanish, split by
branch; the shaded step is the total number of shapes of that size, on the same axis, so that the
rarity of the locus is visible and not asserted. The sizes carrying branch (b) are exactly the
multiples $w(r+1)$ with $w$ odd, by Corollary \ref{cor:sizes}; branch (a) occupies other sizes and
needs the beta set to miss one parity entirely, which is why it is absent at small $|\lambda|$ once
$N$ grows.}
\label{fig:zeros}
\end{figure}

\begin{theorem}[one pair]\label{thm:r1}
For $r=1$ the converse holds as well: $\Psi_1(\lambda)=0$ only if {\rm(a)} or {\rm(b)}.
\end{theorem}

\begin{theorem}[inside Littlewood's range]\label{thm:stable}
For every $r$ and every $\lambda$ with $\ell(\lambda)\le N/2$, the converse holds. Equivalently, on
that range $\Psi_r(\lambda)=0$ if and only if $\lambda=(k^{\,N/2})$ with $k$ odd.
\end{theorem}

\begin{remark}[rectangles are where the general arguments stop]
It is worth noticing that both places where a general argument in this paper fails are rectangles,
at two different heights and for two different reasons. At height $N/2$ an odd rectangle is
self-complementary of odd width, so it vanishes; by Theorem \ref{thm:stable} these are the only
shapes in Littlewood's range that do, and in the proof they are precisely the shapes for which the
associate witness cannot be built. At full height the object does not vanish but degenerates the
other way: $s_{(c^{\,N})}(x)=(x_1\cdots x_N)^{c}$ is a monomial, so on any alphabet $A$ the value is
$\det(A)^{c}$ and nothing of the shape survives --- which at $t=2$ is case {\rm(i)} of Proposition
\ref{prop:rect}, the one case there in which the parity of $c$ plays no part. Everywhere between the
two heights the generic argument applies unchanged. The rectangles are not a nuisance to be handled
case by case; they are the boundary of the shape, and each of the two extremes breaks a different
hypothesis --- the first the existence of a witness, the second the non-degeneracy of the interval
triple.
\end{remark}

\begin{theorem}[the criterion in general]\label{conj:crit}
For every $r$ and every $\lambda\in\PP_N$, $\Psi_r(\lambda)=0$ if and only if {\rm(a)} or
{\rm(b)}.
\end{theorem}

\noindent Proved at the end of \S\ref{sec:zeros}, modulo Theorem \ref{thm:PvW}.

Theorems \ref{thm:zeros}--\ref{thm:stable} are proved in \S\S\ref{sec:brancha}--\ref{sec:conv},
and Theorem \ref{conj:crit} --- the criterion for every $r$ --- at the end of this section, by a
route independent of all three. It is stated apart from them for two reasons. It is the only one of
our principal new results whose proof rests on a non-elementary external input, Theorem
\ref{thm:PvW}; and the two routes to it are not the same route. Littlewood's reduces it, in \S\ref{sec:unstable}, to one existence
statement, and that reduction is not completed: the range $\ell(\lambda)>N/2$ is exactly where the
restriction rule stops applying, and while the values there are settled by Lemma
\ref{lem:AtoSp}, the statement about the coefficients is not --- the certificate we proposed for
it fails, by Proposition \ref{conj:iso}. The extremal argument of \S\S\ref{sec:extremal} and
\ref{sec:rigidity} does not pass through that range at all: it works in the degree filtration of
the numerator and never meets a Littlewood coefficient.

Read on the group rather than on the alphabet, the same theorem describes a rigidity property, and we
state it separately because it is the form in which a representation theorist would want it. The
translation is not ours and is not new: that a representation of $O(N)$ is $\det$-stable exactly when
its character vanishes off the identity component is elementary, and that this is the condition
$m_\nu=m_{\nu^{*}}$ on Littlewood's associates \cite{Lit50} is classical. The corollary therefore
claims exactly what Theorem \ref{conj:crit} claims and nothing besides.

\begin{corollary}[determinant-twist rigidity]\label{cor:dettwist}
Let $N=2r+2$, let $\lambda\in\PP_N$ and let $V_\lambda$ be the irreducible polynomial $GL(N)$-module
of highest weight $\lambda$. Then
\[
V_\lambda\big|_{O(N)}\;\cong\;V_\lambda\big|_{O(N)}\otimes\det
\]
if and only if $\beta(\lambda)$ has constant parity, or $\lambda$ is self-complementary of odd
width. Equivalently: \eqref{eq:branches} classifies exactly the irreducible polynomial
$GL(N)$-modules whose restriction to $O(N)$ is stable under the character of the component group.
\end{corollary}

\begin{proof}
Over characteristic $0$ a finite-dimensional $O(N)$-module is semisimple and is determined by its
character, so for $W=V_\lambda|_{O(N)}$ one has $W\cong W\otimes\det$ if and only if
$\chi_W=\chi_{W\otimes\det}$. Those two characters agree on the identity component automatically, and
on the other one $\chi_{W\otimes\det}=-\chi_W$; so the condition is that $\chi_W$ vanish on
$O(N)^{-}$. Every \emph{semisimple} element of $O(N)^{-}$ is conjugate to one with spectrum
$(1,-1,z_1^{\pm1},\dots,z_r^{\pm1})$, which is the alphabet of \eqref{eq:psir}; and since the
semisimple elements are dense in $O(N)^{-}$ and $\chi_W$ is regular, vanishing on that torus is
vanishing on the whole component. The condition is therefore $\Psi_r(\lambda)\equiv0$, and Theorem
\ref{conj:crit} is the criterion. The dictionary between the two readings is Littlewood's associate
expansion, recalled in \S\ref{sec:conv}.
\end{proof}

\noindent The credit splits exactly as it does for Theorem \ref{conj:crit}, and it is worth repeating
here because the group language can make a known statement look new. That the two families
\eqref{eq:branches} \emph{are} $\det$-stable is not ours: at the endpoint of the family the first is
Karmakar's \cite[Theorem 4.1(A)]{Kar24}, and the second is complementation over the index family of
Ayyer and Behrend \cite{AB19}, as \eqref{eq:compl} records. What Theorem \ref{conj:crit} adds, and
all it adds, is that there are no others.

\begin{remark}[what the core sees, and what it does not]\label{rem:core}
Branch (a) is a condition on the core, and a simple one: at a fixed number $N$ of beta numbers the
residue profile determines $\core_t(\lambda)$ --- it is the Garvan--Kim--Stanton coordinate
\cite{GKS90} --- so every condition on the profile is one on the core. It is not, however, the
vanishing of the core but its opposite extreme: at $t=2$ branch (a) says
\[
\core_2(\lambda)=(N-1,N-2,\dots,1)\qquad\text{or}\qquad\core_2(\lambda)=(N,N-1,\dots,1),
\]
the two largest $2$-cores that $N$ beta numbers admit, whereas $\core_2(\lambda)=\varnothing$ is the
balanced profile. Branch (b) moves neither the profile nor the core, and is therefore invisible to
it. Consequently $\core_2$ does \emph{not} determine whether $\Psi_r$ vanishes, and the smallest
witness is the empty core itself: among the shapes with $\core_2(\lambda)=\varnothing$ at $r=1$,
\[
(1,1),\ (3,3),\ (2,2,1,1),\ \dots\quad\text{vanish, while}\quad
\varnothing,\ (2),\ (1^4),\ \dots\quad\text{do not,}
\]
and the same split occurs at $r=2$ and $r=3$, there also in the classes with core $(1)$, $(2,1)$ and
$(3,2,1)$: over the ranges of \S\ref{sec:verif} five core classes carry both behaviours, and the
profile determines the core with no clash over $260$ profiles at $t\le5$. No restatement of Theorem
\ref{thm:zeros} in terms of $\core_2(\lambda)$ can exist: the
criterion needs the shape and not only its core. Read against Theorem \ref{thm:LR}, where the empty
core is exactly what decides, the two criteria disagree in both directions --- $(1,1)$ has empty
$2$-core and $\Psi_1=0$, while $(1)$ has non-empty $2$-core and $\Psi_1\ne0$.
\end{remark}

The index family in (b) is not new: it is exactly the family Ayyer and Behrend single out
\cite[remark after Thm.~3]{AB19}, together with the parity obstruction they record there --- a
partition cannot be self-complementary in an $a\times b$ rectangle with $a$ and $b$ both odd, and
here $a=N$ is even, so $b=w$ odd is the admissible case. What their theorems say about those shapes
is that the character \emph{factorizes}; the alphabet there carries the fixed point $+1$ and has
determinant $+1$. Ours carries $+1$ and $-1$ and has determinant $-1$.

\keybox{On the \emph{same} family of shapes, the alphabet of determinant $+1$ makes the character
factorize and the alphabet of determinant $-1$ makes it vanish. One letter separates a factorization
from an annihilation, and the determinant of the alphabet is what decides; \eqref{eq:compl} below is
where the separation happens.}

\subsection{Branch (a): the squared alphabet}\label{sec:brancha}

This branch is Karmakar's, at the endpoint. \cite[Theorem 4.1(A)]{Kar24} evaluates a
$\mathrm{GL}(n,\CC)$ character at $C_{n-k,k}=(1,\dots,1,-1,\dots,-1)$ and gives
$\chi_\lambda(C_{n-k,k})=0$ when $\#\eta_0(\lambda)>n-k$, where the $\eta_i(\lambda)=\{a\in
\lambda+\rho: a\equiv i \bmod 2\}$ is exactly the partition of $\beta(\lambda)$ into $E$ and $O$ used
here; at $k=1$, after the normalization by the determinant character, that is branch (a). The
one-line argument below is included because it is the one the rest of the section reuses, not because
it is new.

\begin{proof}[Proof of Theorem \ref{thm:zeros}, branch {\rm(a)}]
Let $A=\{1,-1,z_1,\bar z_1,\dots\}$. If $\beta=2\gamma$ then
$\det(a^{\beta_j})_{a\in A}=\det\big((a^2)^{\gamma_j}\big)_{a\in A}$, and if $\beta=2\gamma+1$ the
same holds after extracting $\prod_{a\in A}a$. Either way the bialternant factors through the
squared alphabet $A^2=\{1,1,z_1^{2},\bar z_1^{2},\dots\}$, in which the letter $1$ occurs twice
because $1^2=(-1)^2$. Two rows coincide.
\end{proof}

\subsection{Branch (b): a determinant twist, and self-duality}

\begin{proof}[Proof of Theorem \ref{thm:zeros}, branch {\rm(b)}]
Let $\hat\lambda$ be the complement of $\lambda$ in the $w\times N$ box, reversed. In beta sets
$\beta(\hat\lambda)=c-\beta(\lambda)$ read backwards, with $c=w+N-1$, and $\hat\lambda$ is the
highest weight of $V_\lambda^{*}\otimes\det^{\,w}$ as a $\mathrm{GL}(N)$-module. If
$\lambda=\hat\lambda$ then $V_\lambda\cong V_\lambda^{*}\otimes\det^{\,w}$. Restrict to $O(N)$: every
representation of $O(N)$ is self-dual, so $V_\lambda|_{O(N)}\cong V_\lambda^{*}|_{O(N)}\cong
V_\lambda|_{O(N)}$, whence
\[
V_\lambda|_{O(N)}\;\cong\;V_\lambda|_{O(N)}\otimes\det^{\,w}.
\]
The alphabet \eqref{eq:psir} is a torus element of the $\det=-1$ component, so for $w$ odd the twist
is by $\det$ itself and the character vanishes there. For $w$ even $\det^{\,w}$ is trivial and the
argument gives nothing, which is why the parity of $w$ is not decoration.
\end{proof}

Equivalently, and this is the computation rather than the representation, the standard
complementation identity gives
\begin{equation}\label{eq:compl}
s_{\hat\lambda}(A)\;=\;\Big(\textstyle\prod_{a\in A}a\Big)^{c}\,(-1)^{\binom{N}{2}}(-1)^{r}\,
s_\lambda(A)\;=\;(-1)^{c}\,(-1)\,s_\lambda(A),
\end{equation}
using that $A$ is closed under inversion --- $A^{-1}$ is $A$ with the $r$ reciprocal pairs
transposed, contributing $(-1)^r$ --- and that $\prod_{a\in A}a=-1$. With $c$ even this reads
$s_{\hat\lambda}=-s_\lambda$, and a self-complementary shape is annihilated. Branch (b) of Theorem
\ref{thm:zeros} is therefore a two-line corollary of a textbook identity over a known index family,
and we claim it as nothing more. Figure \ref{fig:involution} draws the pairing the identity performs,
and next to it the same shape with one box moved, where the pairing breaks; the content of the
section is the converse.

\subsection{The converse}\label{sec:conv}

\begin{figure}[t]
\centering
\includegraphics[width=\textwidth]{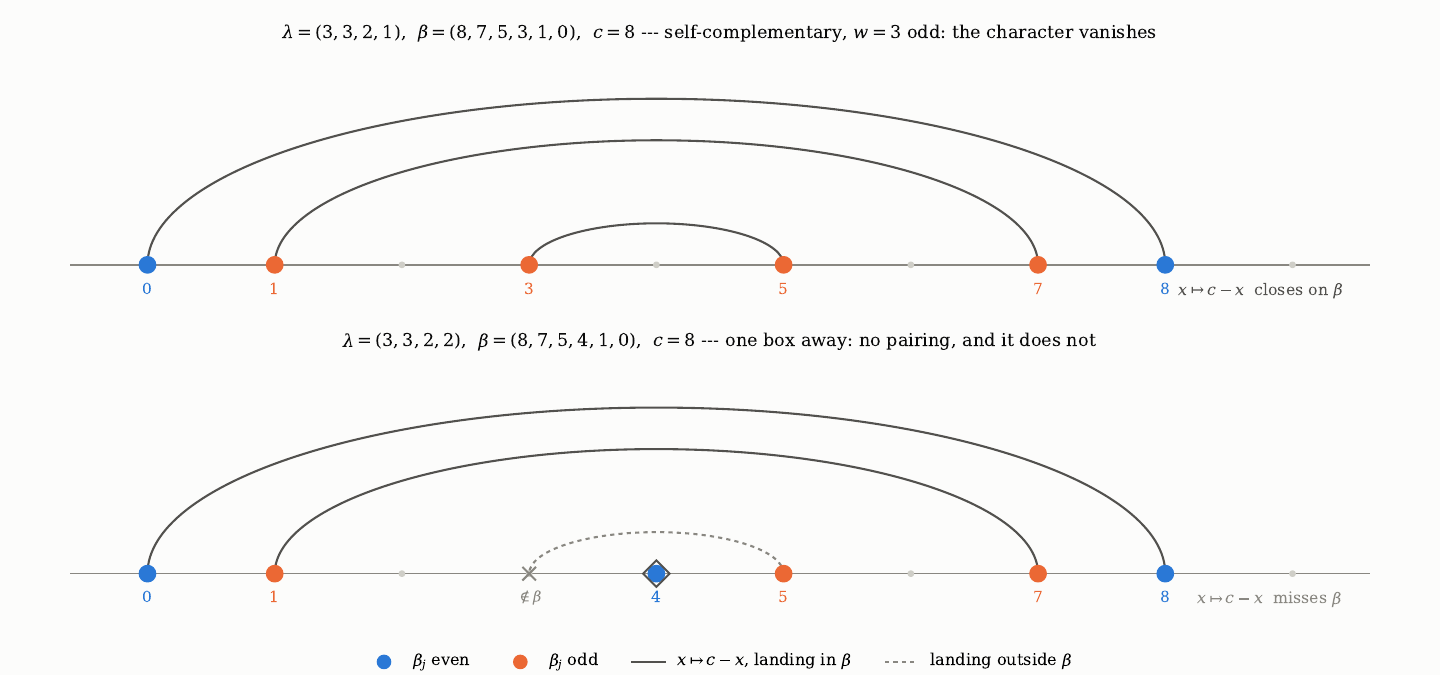}
\caption{Why self-complementarity kills the character. Expanding along the two frozen rows leaves
terms indexed by pairs $(e,o)$ with $e\in E$, $o\in O$, and the reflection $x\mapsto c-x$ acts on
them. Above, a self-complementary shape of odd width: every arc lands in $\beta$ and joins entries of
equal parity, because $c$ is even. Below, the same shape with one box moved: the arc $5\mapsto3$
leaves $\beta$, so there is no pairing and nothing cancels. The diamond marks the fixed point of the
reflection, which cannot occur inside an opposite-parity pair --- which is why the involution is
free.}
\label{fig:involution}
\end{figure}

\begin{proof}[Proof of Theorem \ref{thm:r1}]
Put $c_\ast=z+\bar z$ and let $S_n$ be defined by $S_0=0$, $S_1=1$,
$S_{n+1}=c_\ast S_n-S_{n-1}$, so $\deg_{c_\ast}S_n=n-1$ and $S_{-n}=-S_n$; the value $S_{-1}=-1$ is
used whenever $\beta_N=0$. Since $\Phi_A=(x^2-1)(x^2-c_\ast x+1)$, a polynomial $P$ supported on
$\beta$ is divisible by $\Phi_A$ exactly when $P(1)=P(-1)=0$ and $P\equiv0$ modulo $x^2-c_\ast x+1$.
Splitting $P(\pm1)$ by parity and reducing $x^n\equiv S_nx-S_{n-1}$, admissibility becomes the
singularity of the $4\times4$ matrix with rows $[\beta_j\ \mathrm{even}]$,
$[\beta_j\ \mathrm{odd}]$, $S_{\beta_j}$, $S_{\beta_j-1}$. Expanding along the two parity rows and
using d'Ocagne's identity \cite{Doc} for this recurrence,
\begin{equation}\label{eq:docagne}
S_bS_{b'-1}-S_{b'}S_{b-1}=-S_{b-b'},
\end{equation}
each $2\times2$ block collapses to a \emph{single} $S$, indexed by the difference of the two
surviving $\beta$'s. Distinct $S_n$ have distinct degrees in $c_\ast$ and are therefore linearly
independent, so the signed sum vanishes only by exact pairing of equal indices with opposite signs.
The case analysis is finite: $|E|\in\{0,4\}$ is branch (a); $|E|\in\{1,3\}$ leaves three terms whose
indices are the three pairwise differences of a $3$-set and hence always distinct, so no pairing
exists and the value is nonzero; and for $|E|=2$ there are four terms and three perfect matchings,
and one checks the six parity patterns directly. For $E=\{\beta_1,\beta_4\}$ and for
$E=\{\beta_2,\beta_3\}$ exactly one of the three matchings is available and it imposes the single
condition $\beta_1+\beta_4=\beta_2+\beta_3$, which is branch (b); for the four remaining patterns
every one of the three matchings forces two distinct entries of $\beta$ to coincide, so none is
available and the value is nonzero.
\end{proof}

Identity \eqref{eq:docagne} is classical. The collapse it produces is an accident of rank one: two
letters admit a single gap, so $s_{(m,n)}(z,\bar z)=S_{m-n+1}$ depends on that gap alone. There is no
rank-$r$ analogue. One would be a closed form for $s_\mu$ at a reciprocal alphabet for \emph{every}
$\mu$; \cite{CK09} supply that for rectangular shapes and \cite{AB19} for self-complementary ones,
and we know of none for arbitrary $\mu$.

For general $r$ we use associates, in Littlewood's sense \cite{Lit50}. On the $\det=-1$ coset $[\nu^{*}]=[\nu]\otimes\det$ gives
$\Psi_r=\sum_\nu\big(m_\nu-m_{\nu^{*}}\big)\chi_{[\nu]}$, so $\Psi_r=0$ if and only if
$m_\nu(\lambda)=m_{\nu^{*}}(\lambda)$ for every $\nu$; equivalently, $V_\lambda|_{O(N)}$ is
$\det$-invariant. That equivalence is the dictionary Corollary \ref{cor:dettwist} is read through,
and it is what turns the criterion into a statement about modules rather than about an alphabet. A
witness is therefore a single $\nu$ with $m_\nu\ne m_{\nu^{*}}$, and it can be
constructed.

Karmakar's \cite[Theorem 4.1]{Kar24} determines the value at the endpoint $z_i=1$, which is an element
of order two: case (A) is branch (a), case (B) is a nonzero dimension, and case (C) leaves an
alternating sum of products of dimensions with no vanishing criterion attached. Branch (b) lives in
that case (C), so the theorem below is not a consequence of it. Nor does case (B) help more as $r$
grows: at $k=1$ it asks that $N-1$ of the $N$ beta numbers share a parity, and the smallest shape
doing so is the staircase $(2r-1,2r-2,\dots,1)$, of size $r(2r-1)$, so on a fixed range of
$|\lambda|$ that nonvanishing case is visible only for small $r$ --- the result is sharpest exactly
where ours is easiest. What connects the endpoint to the
curve at all is rigidity, and rigidity is a statement about our object rather than about the order-two
element: over the ranges of \S\ref{sec:verif} it holds on $\ell(\lambda)\le N/2$ without exception
and fails outside, which is Remark \ref{rem:rankone}.

\begin{proof}[Proof of Theorem \ref{thm:stable}]
Branch (a) cannot occur: with $\ell(\lambda)\le N/2$ at least two trailing parts vanish, so $\beta$
contains two consecutive integers. In Littlewood's range
$m_\nu(\lambda)=\sum_{\beta'\ \mathrm{even}}c^{\lambda}_{\nu\beta'}$, and the term $\beta'=\emptyset$
gives $m_\mu\ge1$ for every $\mu\subseteq\lambda$. Take $\mu$ as follows. If
$\ell(\lambda)<N/2$, take $\mu=\lambda$: then $\ell(\lambda^{*})=N-\ell(\lambda)>\ell(\lambda)$, so
$\lambda^{*}\not\subseteq\lambda$ and $m_{\lambda^{*}}=0\ne m_\lambda$. If $\ell(\lambda)=N/2$ and
$\lambda_{N/2}$ is even, take $\mu=(\lambda_1,\dots,\lambda_{N/2-1})$; then
$\beta'=(\lambda_{N/2})$ is even, $\lambda/\mu$ is a horizontal strip and $m_\mu\ge1$ while
$\ell(\mu)<N/2$, so the previous computation applies to $\mu$. If $\ell(\lambda)=N/2$ and $\lambda$
is not a rectangle, let $j\le N/2-1$ be largest with $\lambda_j>\lambda_{j+1}$ and remove both the
last row and one box from row $j$; then $\beta'=(\lambda_{N/2}+1)$ is even and
$\lambda_j\ge\lambda_{N/2}+1$ keeps $\lambda/\mu$ a horizontal strip. An even rectangle
$\lambda=(k^{\,N/2})$ has already been covered by the second construction, so the only case left is
$k$ odd, where neither construction is available --- and that is branch {\rm(b)}.
\end{proof}

\subsection{Outside Littlewood's range}\label{sec:unstable}

For $\ell(\lambda)>N/2$ the restriction rule acquires non-standard labels, and the modification rules
that reduce them are King's \cite{King71}; Koike and Terada \cite{KT87} record that the
type-$D$ specialization sends a universal character to $0$ or to $\pm$ one other, and Fauser, Jarvis,
King and Wybourne observe that beyond the two classical cases \emph{``the general theory, having
infinitely many cases, needs a not yet available formalism''} \cite[\S4]{FJKW05}. Extensions of Littlewood's rule to all parameters do exist for the two classical
cases --- Newell's modification rules \cite{New51}, Sundaram \cite{Sun90} in the symplectic case,
Gavarini's Brauer-algebra proof \cite{Gav99}, and Enright and Willenbring \cite{EW04}, whose Theorem 4
gives the multiplicity as an alternating sum over the Weyl group which in favourable cases collapses
to a difference of two Littlewood coefficients \cite[\S7.11]{EW04}. That observation of
\cite{FJKW05} is from 2005, and it is worth recording that the gap it names has since been filled on
the side we need. For $GL_n\downarrow O_n$ the first combinatorial formula fully
extending Littlewood's restriction rule is Jang and Kwon's \cite{JK20}, and we take the attribution
from Frohmader, who states it in those words; it has besides a property worth separating out. They
write
\[
\bigl[V^{\lambda}_{GL_n}:V^{\mu}_{O_n}\bigr]
=\sum_{\delta\in\mathcal P(2)}\overline{c}^{\,\lambda}_{\delta\mu},
\]
their Theorem 1.1, where $\overline{c}^{\,\lambda}_{\delta\mu}$ \emph{counts} the
Littlewood--Richardson tableaux of shape $\lambda/\delta$ and content $\mu^{\pi}$ whose companion
tableau satisfies one flag condition, $\tau_j+n_j\le n+1$; their Corollary 4.12 restates the count
on the ordinary tableau. The phrase ``vanish in a stable range'' has a precise meaning here:
when $\ell(\lambda)\le n/2$ the two sequences become $m_i=2i-1$ and $n_j=2j$, the inequality holds
of itself, and Littlewood's rule is recovered --- their Corollary 4.13. What separates it from the
earlier extensions is that it is a count and not an alternating sum; they note that those were
obtained algebraically and are therefore not subtraction-free, and theirs is. Frohmader
\cite{Frohmader} then gives combinatorial branching rules for $GL_n\downarrow O_n$ and
$GL_{2n}\downarrow Sp_{2n}$ at once, generalizing Littlewood's restriction rules and, in its words,
\emph{``valid for all parameters, i.e.\ there are no stable range constraints''}: there the
multiplicity becomes a sum of generalized Littlewood--Richardson coefficients
$c^{\lambda}_{\mu\nu}(O_n)$ read off flag conditions on ordinary Littlewood--Richardson tableaux,
with no modification rule anywhere. We do not use
either rule below, and say why in a moment; but they are the natural instruments for the one
question this section leaves open, and we return to them in Problem \ref{prob:instrument}.

The route taken here is the dual
one: keep the coefficients universal, which they are in every range, and move the whole
range-dependence into the values, where one lemma disposes of it. On this alphabet no formalism for
the labels is needed at all, and the reason is that the values are not type-$D$ values.

Adding a letter $c$ to an alphabet acts on the ring of symmetric functions by $p_k\mapsto p_k+c^k$;
write $\iota_c$ for that homomorphism. From the generating function of the complete homogeneous
symmetric functions, $h_m(W,c)-c\,h_{m-1}(W,c)=h_m(W)$ for every $m$.

\begin{lemma}[the alphabet is symplectic]\label{lem:AtoSp}
$\iota_{-1}\iota_{1}\,o_\nu=sp_\nu$ for every partition $\nu$, as an identity between Koike--Terada
\emph{universal} characters in $\Lambda$. In particular, for the alphabet
\eqref{eq:psir},
\begin{equation}\label{eq:AtoSp}
o_\nu(A)\;=\;sp_\nu(z_1,\dots,z_r)
\end{equation}
with no hypothesis on $\ell(\nu)$ and with no sign. For $\ell(\nu)>r$ the right-hand side is
that universal character specialized at $r$ variables, not an irreducible $Sp(2r)$ character; it is
the modification rules below that fold it onto one.
\end{lemma}

\begin{proof}
Both steps are Ayyer and Kumari's, and both are the same manipulation: the recurrence turns the
column operations $C_j\to C_j-c\,C_{j-1}$, applied for $j=n,\dots,2$ in that order, into a change of
character type, and operations of that shape leave a determinant unchanged. At $c=1$ they carry the
defining determinant of $o_\nu$ evaluated at $(W,1)$ to that of the odd orthogonal universal
character $so_\nu$ over $W$, which is Koike and Terada's dictionary in the form
$so_\nu(X)=o_\nu(X,\bar X,1,0,0,\dots)$ \cite[(2.13)]{AK25}. At $c=-1$, where the recurrence is
\cite[Proposition 2.1]{AK25} and the operations read $C_j\to C_j+C_{j-1}$, they carry $so_\nu$ to
$sp_\nu$, which is \cite[Lemma 3.2]{AK25}. Composing gives the first statement, and evaluating it on
the reciprocal alphabet $(z_1,\dots,z_r)$ gives \eqref{eq:AtoSp}.
\end{proof}

\noindent Since $\nu'_1=\ell(\nu)$, the trichotomy on $\nu'_1$ against $N/2=r+1$ is a trichotomy on
the length of $\nu$ against the rank of $Sp(2r)$, and the three facts we need are then the classical
symplectic modification rules --- one of the two cases \cite{FJKW05} exempts --- read through
\eqref{eq:AtoSp}: $o_\nu(A)=0$ when $\ell(\nu)=r+1$, the vanishing of a symplectic character one row
past its rank; $o_\nu(A)=\pm o_{\nu^{*}}(A)$ when $\ell(\nu)>r+1$, King's rule folding the label back
onto a dominant one; and the same for non-standard $\nu$, which by Lemma \ref{lem:nonstd} are exactly
the tall ones. The expansion $s_\lambda=\sum_\nu c_\nu o_\nu$, with
$c_\nu=\sum_{\beta'\ \mathrm{even}}c^{\lambda}_{\nu\beta'}$, is an identity of symmetric functions and
holds in every range; specializing it at \eqref{eq:psir} therefore lands every term on
$\{o_\mu(A):\ell(\mu)\le r\}$, which by \eqref{eq:AtoSp} is the set of irreducible characters of
$Sp(2r)$ and so is linearly independent. Figure \ref{fig:reduction} sorts the labels into the four
classes this paragraph describes and computes $o_\nu(A)$ on each, so that the reduction can be seen
to land where it is claimed to. Hence

\begin{figure}[!ht]
\centering
\includegraphics[width=\textwidth]{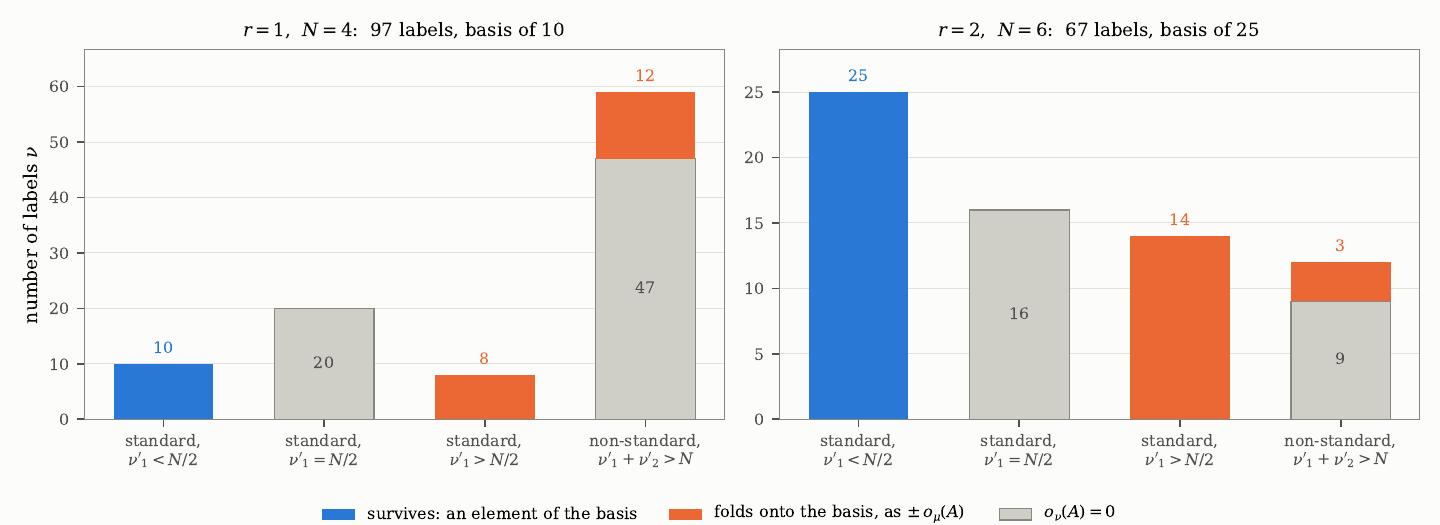}
\caption{The reduction of Lemma \ref{lem:AtoSp}, drawn. Each label $\nu$
occurring in $s_\lambda=\sum_\nu c_\nu o_\nu$ is placed in one of four classes and its value
$o_\nu(A)$ computed. Self-associate labels vanish, every one of them; labels standard with
$\nu'_1>N/2$ fold onto the basis as $\pm o_{\mu}(A)$; non-standard labels either vanish or fold onto
a single basis element with coefficient $\pm1$. By \eqref{eq:AtoSp} these three are the symplectic
modification rules in disguise, the classes being cut by $\ell(\nu)$ against the rank $r$; the figure
was computed before that was noticed, and is left as drawn. Everything lands on the independent
basis $\{o_\mu(A):\ell(\mu)\le r\}$, which is what turns the vanishing into the finite condition
\eqref{eq:Cmu}. The bar heights are label counts over the range each panel was computed on, which is
smaller than the range of the corresponding row of \S\ref{sec:verif}; the classification, not the
count, is what the picture is about.}
\label{fig:reduction}
\end{figure}
\begin{equation}\label{eq:Cmu}
\Psi_r(\lambda)=\sum_{\ell(\mu)\le r}C_\mu(\lambda)\,o_\mu(A),
\qquad
\Psi_r(\lambda)=0\iff C_\mu(\lambda)=0\ \text{for all }\mu,
\end{equation}
where $C_\mu=c_\mu\pm c_{\mu^{*}}+\sum\pm c_\nu$ over the non-standard $\nu$ reducing to $\mu$. The
reduction is carried out in two steps rather than one: most non-standard labels pair with an associate
inside the range, while $21$ of them --- $18$ at $r=1$ and $3$ at $r=2$ --- have no associate in range
and are reduced separately, each landing in the span of the standard values with integer coefficients.

We had computed the three facts before locating them, and record the computation because it is what
the figure below draws: \texttt{sec8\_derivation.sage} checks \eqref{eq:AtoSp} directly for all
$|\nu|\le8$, confirms that the labels of length $r+1$ vanish and that longer ones fold onto the basis
with a sign, and verifies the independence as a rank computation, at $r=1,2,3$. A control that had to
fire: $\iota_1 o_\nu$ and $\iota_{-1}o_\nu$ each agree with $sp_\nu$ only for $\nu=\varnothing$, so
both letters are needed and \eqref{eq:AtoSp} is a statement about the reciprocal pair, not about
either letter alone. \texttt{sec8\_formulas.sage} checks the proof rather than the statement --- the
recurrence at four values of $c$, the two defining determinants against an independent
implementation, and the column operations entry by entry. That last check earned its keep: an earlier
draft of the proof wrote $C_j\to C_j+C_{j-1}$ at $c=1$, the sign \cite[Lemma 3.2]{AK25} uses for its
own step at $c=-1$.

The non-standard labels, which are the whole obstruction here, cannot be small:

\begin{lemma}[non-standard labels are large]\label{lem:nonstd}
Every non-standard $\nu$ occurring in the expansion satisfies $|\nu|\ge N+1=2r+3$. Consequently, if
$|\lambda|\le2r+2$ then no non-standard label is contained in $\lambda$, the last sum in $C_\mu$ is
empty, and $C_\mu(\lambda)=c_\mu(\lambda)\pm c_{\mu^{*}}(\lambda)$ --- with no hypothesis on
$\ell(\lambda)$, hence also outside Littlewood's range.
\end{lemma}

\begin{proof}
Since $c^{\lambda}_{\nu\beta'}=0$ unless $\nu\subseteq\lambda$, only labels with
$\ell(\nu)\le\ell(\lambda)\le N$ occur. Non-standard means $\nu'_1+\nu'_2>N$. Were $\nu$ a single
column we would have $\nu'_2=0$ and hence $\nu'_1=\ell(\nu)>N$, which is excluded; so $\nu'_2\ge1$ and
$|\nu|\ge\nu'_1+\nu'_2>N=2r+2$. A non-standard $\nu$ can therefore contribute only when
$|\lambda|\ge|\nu|\ge2r+3$.
\end{proof}

\noindent Since $\ell(\lambda)>N/2$ forces $|\lambda|\ge r+2$, Lemma \ref{lem:nonstd} settles the
band $r+2\le|\lambda|\le2r+2$ of the unstable range outright: there Littlewood's rule applies
verbatim and the converse follows by the argument of Theorem \ref{thm:stable}. What this route still
has to cross is $|\lambda|>2r+2$.

Two exact facts bound the competitors. First $\ell(\mu^{*})=N-\ell(\mu)$, so $\ell(\mu)\le
N-1-\ell(\lambda)$ forces $\mu^{*}\not\subseteq\lambda$ and $c_{\mu^{*}}=0$. Second, a non-standard
$\nu$ has $\nu'_1+\nu'_2>N$ with $\nu'_2\le\nu'_1$, hence $2\ell(\nu)>N$ and $\ell(\nu)\ge r+2$; in
particular no non-standard label fits inside $\lambda$ in Littlewood's range. Call $\mu$
\emph{isolating} for $\lambda$ if exactly one of the labels feeding $C_\mu$ is contained in
$\lambda$; then $C_\mu=\pm c_\nu\ne0$ and $\Psi_r(\lambda)\ne0$. What the route asks for there is one
existence statement --- open as an existence statement, though not as a theorem, since Theorem
\ref{conj:crit} does not go through it.

\begin{proposition}[the isolating witness does not always exist]\label{conj:iso}
At $r=2$ the shape $\lambda=(5,5,2,1)$ has $\Psi_2(\lambda)\ne0$ and is of neither type {\rm(a)} nor
type {\rm(b)}, and no $\mu$ is isolating for it: every $\mu$ whose $C_\mu$ has a nonzero contributor
inside $\lambda$ has exactly two, namely $\mu$ and its associate $\mu^{*}$.
\end{proposition}

We had proposed the isolating witness as the route to Theorem \ref{conj:crit}, and it does not
reach. The reason is legible in the isolation lemma: it needs $\ell(\mu)\le N-1-\ell(\lambda)$, which
at $\ell(\lambda)=4$ and $N=6$ leaves only one-row $\mu$, and a shape as wide as $(5,5,2,1)$ contains
every associate. Over $r=2$ and $|\lambda|\le14$ there are eleven such shapes, and none at all below
$|\lambda|=13$, which is why the certificate looks sound on small ranges. Theorem \ref{conj:crit}
itself is untouched --- it asserts that the $C_\mu$ vanish, not that a certificate exists --- but
this route to it is closed. The one that reaches it, \S\ref{sec:rigidity}, does not go through
Littlewood's range at all, which is why the failure here costs the explanation and not the theorem.

The mechanism does reach a great deal: at $r=2$ and $|\lambda|\le14$, $227$ of the $238$ non-vanishing
unstable shapes carry a proved isolating witness. It also reaches further than fact (ii) allows on
its own --- the true minimum of $\ell(\nu)$ over non-standard $\nu$ with $o_\nu(A)\ne0$ is $5$ there,
against the $r+2=4$ the bound gives. What defeats it is width rather than height: the basis in
\eqref{eq:Cmu} grows with $r$ and the competing labels are constrained to be tall, so we had read
isolation as getting easier with $r$, and along the height axis it does; the residue lives where
$\lambda$ is wide enough to swallow the associates. That reading predicts where the residue is
\emph{not}: at $r=3$ the isolation lemma allows $\ell(\mu)\le7-\ell(\lambda)$, so a shape must be
wider than at $r=2$ before it can swallow every associate, and the same sizes should come out clean.
They do --- over $r=3$ and $|\lambda|\le13$ all $149$ non-vanishing unstable shapes carry a proved
witness and the residue is empty --- which is a control on the explanation rather than a step toward
Theorem \ref{conj:crit}. At $r=1$ the question does not arise, since that case is settled independently by
Theorem \ref{thm:r1}.

The tool one reaches for first does not fit. Saturation, in the honeycomb form of Knutson and Tao
\cite{KT99}, decides when a single Littlewood--Richardson coefficient is positive, and every
competitor feeding $C_\mu$ is positive as soon as its label sits inside $\lambda$. What an isolating
$\mu$ asks for is not positivity but \emph{isolation}: that exactly one of them does. The positivity
is what makes the question hard rather than what answers it, and we know of no tool aimed at the
second question.

\begin{figure}[tb]
\centering
\includegraphics[width=\textwidth]{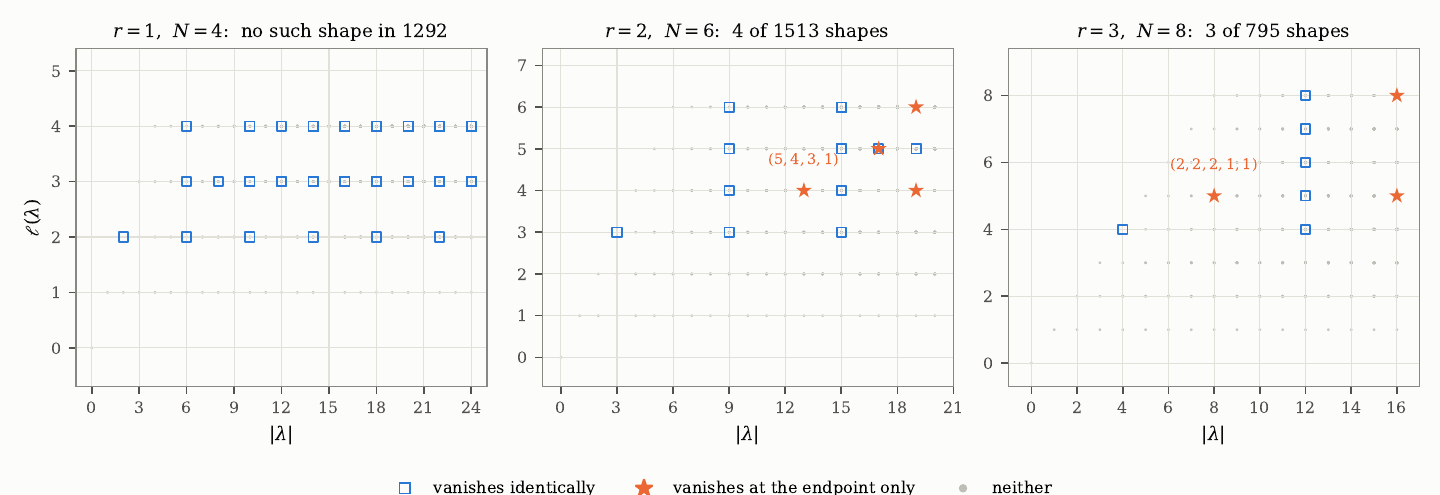}
\caption{Where one reciprocal pair stops being typical. Each partition with at most $N$ parts is a
point at $(|\lambda|,\ell(\lambda))$; the class drawn as a star is the one that vanishes at the
endpoint $z_i=1$ without vanishing identically, and it is the whole content of the picture; each
panel names its smallest member and Remark \ref{rem:rankone} lists the rest. The panels cover
$|\lambda|\le24$, $20$ and $16$ respectively. At $r=1$ the class is
empty over all $1292$ shapes in range, because the object there is $\pm$ a genuine character and its
value at the identity is a dimension. From $r=2$ on properly virtual examples occur and the class is
not empty --- not that every shape is one: Table \ref{tab:virtual} counts both. See Remark
\ref{rem:rankone}.}
\label{fig:phase}
\end{figure}

\begin{remark}[the rank-one accident, and where the two ranges part]\label{rem:rankone}
At $r=1$ the object is $\pm$ a genuine character, so its value at the identity is a dimension and
$\Psi_1(\lambda)\equiv0$ if and only if $\Psi_1(\lambda)|_{z=1}=0$. At $r\ge2$ the folded expansion
\emph{can} be properly virtual --- not that it always is, and Table \ref{tab:virtual} counts both
columns --- and one shape where it is suffices to break that equivalence:
$\lambda=(5,4,3,1)$ at $r=2$ has $\Psi_2(\lambda)\ne0$ and
$\Psi_2(\lambda)|_{z_1=z_2=1}=0$. Figure \ref{fig:virtual} shows the two regimes side by side,
on a pair of shapes differing in a single box. Over the ranges of Figure \ref{fig:phase} the class of such shapes
is empty at $r=1$ across $1292$ partitions; at $r=2$ it consists of the four shapes $(5,4,3,1)$,
$(5,5,4,2,1)$, $(10,5,3,1)$ and $(6,5,4,2,1,1)$, and at $r=3$ of the three shapes $(2,2,2,1,1)$,
$(5,5,4,1,1)$ and $(3,3,3,2,2,1,1,1)$; it is thin, but it is not empty, and one counterexample is all
the equivalence can survive. The surviving implication, exact vanishing $\Rightarrow$ vanishing
at the endpoint, is what makes the endpoint a valid sieve and no more. This is why results computed
at elements of order two --- Karmakar \cite{Kar24} along the whole family of them, Eisenk\"olbl
\cite{Eis07} at the balanced one $x_i=(-1)^i$ --- bear on a point and not on the curve, and why
Theorem \ref{thm:zeros} is not a corollary of them for $r\ge2$.
\end{remark}

It is worth putting a number on what the endpoint sieve loses, since the qualitative statement is
easy to misread as a small correction. Over the ranges of Figure \ref{fig:phase} the test
$\Psi_r(\lambda)|_{z_i=1}=0$ flags $143$ shapes at $r=1$, $22$ at $r=2$ and $9$ at $r=3$; of those,
$0$, $4$ and $3$ respectively do not vanish identically. So the sieve is exact at one pair, and the
fraction of its verdicts that are spurious grows with the number of pairs --- $0\%$, $18\%$, $33\%$
over these ranges. An endpoint computation is a filter that never misses a genuine zero, and past one
pair it is nothing more than that. The rank-one case at $t=2$, where the sieve happens to be exact
because the object is a genuine character, was treated separately in \cite{Mar26}; the alphabet
\eqref{eq:object} arose there, in a Wilson-line computation on an orbifold, which is why $-1$ appears
as a letter rather than as a choice.

The reason the two ranges of $\ell(\lambda)$ behave so differently is worth separating from the
mechanism that proves it. Littlewood's rule is a statement about which labels $\nu$ can occur, and
$\ell(\lambda)\le N/2$ is exactly the condition under which the non-standard ones cannot fit inside
$\lambda$ at all. Lemma \ref{lem:nonstd} says that the obstruction is governed by $|\lambda|$ as well
as by $\ell(\lambda)$, so the honest picture is not two ranges but a region: \emph{this route}
reaches the converse wherever either $\ell(\lambda)\le N/2$ or $|\lambda|\le2r+2$, and outside their
union it does not --- the isolating witness covers most of what is left but provably not all of it,
by Proposition \ref{conj:iso}. The converse is not open there. Theorem \ref{conj:crit} settles it
for every $r$ and every $\lambda$ by the extremal argument of
\S\S\ref{sec:extremal}--\ref{sec:rigidity}, which never enters this range at all. What the region
marks is where Littlewood's reduction stops explaining, not where the theorem stops holding.

\subsection{The criterion at \texorpdfstring{$t=2$}{t=2}: the extremal argument}\label{sec:extremal}

Throughout this subsection and the next two, $t\ge2$ and $r\ge1$ are arbitrary and
\[
  N:=t+2r,\qquad
  \Phi_{t,r}(\lambda;z):=s_\lambda\bigl(\zt,\,z_1,\zb_1,\dots,z_r,\zb_r\bigr),
\]
so that $\Phi_{t,1}$ is \eqref{eq:object} and $\Phi_{2,r}=\Psi_r$. The beta set, the residue counts
$n_i$ and the column indexing are those of \S\ref{sec:background}; $\lambda\in\PP_N$ and
$\beta_1>\dots>\beta_N\ge0$. Write
\[
  \mathcal{E}=\{i:n_i(\lambda)\ge2\},\qquad e=|\mathcal{E}|,\qquad
  \mathcal{S}=\{\beta_j:\beta_j\equiv i\ (t)\ \text{for some } i\in\mathcal{E}\}
\]
for the \emph{excess classes} and the \emph{excess values}. Since the $t-e$ remaining classes are
singletons, \eqref{eq:excess} becomes
\begin{equation}\label{eq:excessgen}
  \sum_{i\in\mathcal{E}}\bigl(n_i(\lambda)-1\bigr)=2r,\qquad |\mathcal{S}|=2r+e,\qquad 1\le e\le2r .
\end{equation}
We assume throughout the \emph{occupancy hypothesis} $n_i\ge1$ for every $i$; by Lemma \ref{lem:L1}
its failure gives $\Phi_{t,r}\equiv0$ outright, which is branch (a).

\begin{lemma}[the Laplace decomposition, general $r$]\label{lem:L3gen}
Let $\mathcal T$ be the set of $t$-subsets $S$ of columns carrying all $t$ residues, and for
$S\in\mathcal T$ let $A(S^{c})$ be the minor of $\det(x_i^{\beta_j})$ on the $2r$ rows carrying
$z_1,\zb_1,\dots,z_r,\zb_r$ and the columns $S^{c}$. Then
\begin{equation}\label{eq:laplacegen}
  \det\bigl(x_i^{\beta_j}\bigr)_{i,j=1}^{N}
  \;=\;(-1)^{\binom{t+1}{2}}\,V\sum_{S\in\mathcal T}w(S)\,A(S^{c}),
  \qquad
  w(S)=(-1)^{\,\sum_{j\in S}j+\inv(b_S)} .
\end{equation}
\end{lemma}

\begin{proof}
Laplace along the first $t$ rows gives
\[
\det(x_i^{\beta_j})=\sum_{|S|=t}(-1)^{\binom{t+1}{2}+\sum_{j\in S}j}M_S\,A(S^{c}).
\]
By Lemma
\ref{lem:L1}, $M_S=0$ unless the residues $b_S$ are pairwise distinct --- that is, unless
$S\in\mathcal T$ --- and then $M_S=(-1)^{\inv(b_S)}V$.
\end{proof}

At $r=1$ the complementary minor is the $2\times2$ determinant $f(\beta_{j_1}-\beta_{j_2})$ and
\eqref{eq:laplacegen} is Lemma \ref{lem:L3}; the content below is what replaces that $2\times2$
minor when $r>1$. Only the relative signs $w(S)$ matter, and since a member of $\mathcal T$ is
determined by which excess value it keeps in each excess class --- the singleton classes are kept by
every $S$ --- we index the terms by that choice, written $g$, and set $T_g=\mathcal{S}\setminus g$,
a set of $2r$ values listed decreasingly.

One question is worth settling here, because everything below moves $g$ about and the sign has to be
carried along: what happens to $w$ when the choice in a single class changes. At $r=1$ that is Lemma
\ref{lem:L4}, where a class has two elements and there is nothing between them. In general there is.

\begin{lemma}[the column move, any class size]\label{lem:step}
Let $g$ and $g'$ agree outside one class and choose there the values $u>v$ respectively. Then
\begin{equation}\label{eq:step}
\frac{w(g')}{w(g)}\;=\;(-1)^{\,1+B+M},
\end{equation}
where $B=\#\{\beta_j:v<\beta_j<u\}$ and $M$ is the number of \emph{frozen} values lying strictly
between $v$ and $u$ --- the $t$ values a transversal keeps, one in each residue class, so the $e$
values of $g$ together with the $t-e$ singleton classes. It may be counted in either $g$ or $g'$,
which agree strictly between $v$ and $u$.
\end{lemma}

\begin{proof}
The two frozen sets differ by replacing the column of $u$ with that of $v$. Since $\beta$ is strictly
decreasing, that column index increases by $B+1$, which is the change in $\sum_{j\in S}j$. In the
word $b_S$ the letter carried by that column travels past exactly the $M$ other frozen letters lying
strictly between, and each crossing changes $\inv$ by one; the letter itself is unchanged, and the
$B-M$ intervening columns that are \emph{not} frozen contribute no letter at all. So
$\inv(b_S)$ changes by $M$ modulo $2$, and \eqref{eq:step} follows from the definition of $w$ in
Lemma \ref{lem:L3gen}.
\end{proof}

\noindent Lemma \ref{lem:L4} is the case $B=M=0$ of this, read twice. The bookkeeping is the one
Albion \cite{Alb23,Alb25} and Ayyer--Behrend \cite{AB19} carry out whenever they move or reverse a
block of columns and collect the sign, and we follow it because it is the accounting this kind of
determinant asks for; we set it out in this form only because \S\ref{sec:rigidity} moves a single
column many times over and the count then has to be exact rather than absorbed into a constant. The
two terms must both be there: dropping either gives a rule that is right about half the time, and
$B-M$ is exactly the number of intervening columns that no transversal freezes --- which is why a
class of size two, with nothing between its elements, hides the distinction. The singleton classes
must be counted in $M$ as well, and that is not a formality: reading $M$ over the choices $g$ alone
already fails at $t=3$, $r=1$ on $\beta=(4,3,2,1,0)$, where moving the class of $1$ from $u=4$ to
$v=1$ has $B=2$ and one intervening frozen value that lies in a singleton class, so the two readings
differ by a sign.

\subsubsection*{The degree filtration}

For a monomial $\prod_j z_j^{c_j}$ we use the \emph{total degree} $\sum_j c_j$. Every monomial of
$A(S^{c})$ has the form $\prod_j z_j^{\,u-u'}$ for a matching of $T_g$ into ordered pairs assigned
to the variables, so writing $T=(u_1>\dots>u_{2r})$,
\begin{equation}\label{eq:deggen}
  \deg(T):=\sum_{a\le r}u_a-\sum_{a>r}u_a
\end{equation}
is the largest total degree occurring in $A$, attained exactly by the matchings joining the top half
$H=(u_1,\dots,u_r)$ to the bottom half $L=(u_{r+1},\dots,u_{2r})$. Writing
$a_X(z)=\det(z_j^{X_i})_{r\times r}$ and $P(T)=a_H(z)a_L(z^{-1})$, which is nonzero, being a product
of two $GL(r)$ alternants, the component of maximal total degree of $A$ is
\begin{equation}\label{eq:Amaxgen}
  [A]_{\deg(T)}\;=\;(-1)^{\binom r2}\,P(T).
\end{equation}
Indeed in maximal degree each variable receives one row of $H$ with a positive exponent and one row
of $L$ with a negative one, the two assignments being independent, so the sum factorises into the two
determinants; the sign is that of the shuffle sending the rows of $H$ to the columns
$z_1,\dots,z_r$ and those of $L$ to the columns $z_1^{-1},\dots,z_r^{-1}$, which is
$(-1)^{\binom r2}$. Under the
grading by $\sum_j|c_j|$ the maximal component splits into $2^{r}$ orientation sectors --- one for each subset of the
variables that is inverted --- with disjoint
monomials, permuted by $z_j\mapsto z_j^{-1}$ with a uniform sign, so they vanish together; all
statements below use the total degree.

\begin{lemma}[reflection]\label{lem:reflgen}
$P(c-T)=P(T)$ for every $c\in\ZZ$ and every $T$.
\end{lemma}

\begin{proof}
The top half of $c-T$ is $c-L$ listed decreasingly, so
$a_{c-L}(z)=(\prod_j z_j^{c})(-1)^{\binom r2}a_L(z^{-1})$, and symmetrically for the bottom half.
The two powers of $\prod_j z_j^{c}$ cancel and $(-1)^{r(r-1)}=+1$.
\end{proof}

Let $G$ be the set of $g$ maximising $\deg(T_g)$ and $D_1$ that maximum. Since the numerator
$\det(x_i^{\beta_j})$ and $\Phi_{t,r}$ differ by the nonzero factor $\det(x_i^{N-j})$, we have
$\Phi_{t,r}\equiv0$ if and only if the numerator vanishes, and \emph{every degree and stratum below
refers to the numerator}; no comparison between its homogeneous parts and those of the quotient is
made. By \eqref{eq:laplacegen},
\begin{equation}\label{eq:topgen}
  \bigl[\det(x_i^{\beta_j})\bigr]_{D_1}
  =(-1)^{\binom{t+1}{2}+\binom r2}\,V\sum_{g\in G}w(g)\,P(T_g),
\end{equation}
the second sign coming from \eqref{eq:Amaxgen}. Only the relative signs $w(g)$ matter below, so the
global constant is recorded once and not carried.

\begin{remark}[a warning about that global sign]
It is tempting to drop $(-1)^{\binom r2}$ on the ground that nothing below depends on it, and we
record it for the opposite reason: an identity that is stated is an identity that gets checked, and
this one is \emph{false} without it as soon as $r\ge2$. The sign is not decoration. It is the
signature of the shuffle in \eqref{eq:Amaxgen} --- $+1$ for $r=1$ and $r\equiv0,1 \pmod 4$, $-1$
otherwise --- and it is invisible at $r=1$, which is exactly the rank at which one first writes the
formula down.
\end{remark}

\subsubsection*{The extremal set has at most two elements}

For $i\in\mathcal{E}$ list the excess values of the class decreasingly,
$c_{i,1}>c_{i,2}>\dots>c_{i,n_i}$ with $n_i=n_i(\lambda)$, and define its \emph{increments}
\begin{equation}\label{eq:incrgen}
  \Delta_i(k)\;:=\;c_{i,k}+c_{i,k+1},\qquad 1\le k\le n_i-1 .
\end{equation}
By \eqref{eq:excessgen} there are exactly $2r$ of them in all; write
$\Delta_{(1)}\ge\dots\ge\Delta_{(2r)}$ for the list in decreasing order and put
$\tau:=\Delta_{(r)}$.

\begin{theorem}[the extremal set]\label{thm:Ggen}
Assume $n_i\ge1$ for every $i$. Then
\begin{enumerate}
\item[\rm(a)] $g\in G$ if and only if the multiset of increments it selects, in the sense of Step 2
below, consists of all increments $>\tau$ together with enough increments equal to $\tau$;
\item[\rm(b)] $|G|\le2$, and $|G|=2$ if and only if $\Delta_{(r)}=\Delta_{(r+1)}$, in which case $t$
is even, the two tied increments belong to classes $i$ and $i+t/2$ --- both necessarily in
$\mathcal{E}$ --- and the two maximisers differ exactly in those two classes;
\item[\rm(c)] if $t$ is odd then $|G|=1$.
\end{enumerate}
\end{theorem}

\begin{proof}
\emph{Step 1 (reformulation).} For any $2r$-set $T$ we have
\[
\deg(T)=\max_{A\subseteq T,\,|A|=r}\bigl(2\Sigma A-\Sigma T\bigr),
\]
the maximum being at the top
half. With $T=\mathcal{S}\setminus g$ this reads
$\deg(\mathcal{S}\setminus g)=\max_A\bigl[2\Sigma A+\Sigma g\bigr]-\Sigma\mathcal{S}$, so we
maximise $\Theta(A,g)=2\Sigma A+\Sigma g$ over disjoint pairs with $|A|=r$ and $g$ a choice of one
excess value in each excess class.

\emph{Step 2 (shape of an optimum).} Let $(A,g)$ be optimal and let $x>y$ lie in the same class $i$.
If $y\in A$ and $x=g_i$, exchanging them leaves $g$ admissible and changes $\Theta$ by $x-y>0$; if
$y=g_i$ and $x$ lies in neither $A$ nor $g$, the same exchange gives $x-y>0$. Both are impossible,
so within each class, read decreasingly, the elements of $A$ come first, then $g_i$, then the
unchosen ones. Hence an optimum is parametrised by one integer per class,
$k_i=|A\cap\text{class }i|\in\{0,\dots,n_i-1\}$ with $\sum_i k_i=r$, via
$A\cap\text{class }i=\{c_{i,1},\dots,c_{i,k_i}\}$ and $g_i=c_{i,k_i+1}$. The map $(k_i)\mapsto g$ is
injective and for fixed $g$ the optimal $A$ is the top half of $T_g$, so optimal pairs project
bijectively onto $G$.

\emph{Step 3 (separability).} With that parametrisation $\Theta=\sum_{i\in\mathcal{E}}f_i(k_i)$,
where $f_i(k)=2(c_{i,1}+\dots+c_{i,k})+c_{i,k+1}$: the objective splits over the classes and the
only coupling is $\sum_i k_i=r$.

\emph{Step 4 (the exchange).} $f_i(k)-f_i(k-1)=\Delta_i(k)$, which is strictly decreasing in $k$
because the $c_{i,k}$ are, so
\[
\Theta=\sum_{i\in\mathcal{E}}f_i(0)+\Sigma,\qquad
\Sigma=\sum_{i\in\mathcal{E}}\sum_{k\le k_i}\Delta_i(k),
\]
and maximising $\Theta$ is maximising $\Sigma$. A feasible $(k_i)$ selects, in each class, a
\emph{prefix} of that class's increment list, with $r$ increments in all; call such a selection
\emph{admissible}. Now take any $r$ increments of maximal total, admissible or not. That selection
\emph{is} admissible: if it contained $\Delta_i(k)$ without $\Delta_i(k-1)$, exchanging the two
would raise the total, since $\Delta_i(k-1)>\Delta_i(k)$ --- and the exchange is legitimate because
$\Delta_i(k-1)$ was not already selected. And no admissible selection can beat the total of the $r$
largest increments, being itself a selection of $r$ of them. Hence the maximisers of $\Theta$ are
exactly the admissible selections whose multiset of increments is a multiset of $r$ largest ones,
which is the assertion of {\rm(a)}. (This is the marginal-allocation greedy for a separable
concave objective under one cardinality constraint; see \cite{FG86,IK88}. We have given the argument
because it is four lines, so that the only external input to Theorem \ref{conj:crit} is
Theorem \ref{thm:PvW}.)

\emph{Step 5 (residues).} $c_{i,k}\equiv c_{i,k+1}\equiv i$, hence
\begin{equation}\label{eq:congrgen}
  \Delta_i(k)\;\equiv\;2i \pmod t .
\end{equation}
Within a class the increments are pairwise distinct by Step 4. Across classes,
$\Delta_i(k)=\Delta_{i'}(k')$ forces $2i\equiv2i'$, that is $i'=i$ or --- only if $t$ is even ---
$i'=i+t/2$. Therefore \emph{no value is shared by more than two increments}, and when it is, the two
come from classes $i$ and $i+t/2$.

\emph{Step 6.} The optima are the selections containing every increment $>\tau$ together with $s$ of
those equal to $\tau$, with $s\ge1$. At most two increments equal $\tau$, so either both fit
($|G|=1$) or exactly one does ($|G|=2$), the latter precisely when $\Delta_{(r)}=\Delta_{(r+1)}$.
For $t$ odd, $x\mapsto2x$ is injective on $\ZZ/t$, so no tie is possible and $|G|=1$.
\end{proof}

\begin{remark}
The content of Theorem \ref{thm:Ggen} is \eqref{eq:congrgen} together with the fact that
$x\mapsto2x$ on $\ZZ/t$ has kernel of order $\gcd(2,t)$: the multiplicity of an extremal value is
controlled by the $2$-torsion of the residue group, which is why the parity of $t$ governs
everything here and in \S\ref{sec:main}. Steps 1--4 use only that the $\beta_j$ are distinct
integers; neither $\beta_j\ge0$ nor $\lambda$ being a partition nor $N=t+2r$ enters beyond
\eqref{eq:excessgen}. Figure \ref{fig:increments} is the theorem carried out on one shape, with the
tie and the prefix form of Step 2 both visible.
\end{remark}

\begin{remark}[where this shape of argument comes from]
Steps 1--4 are an old story told in a new alphabet. Maximising a separable concave objective under one
cardinality constraint is the marginal-allocation problem of resource allocation, solved by the greedy
in \cite{FG86,IK88}; we gave the exchange in four lines only so that the reader need not leave the
paper. Step 6 is also familiar, from a different direction: in tropical linear algebra the initial
form of a determinant is the signed sum over the \emph{optimal} assignments, and the matrix is called
tropically singular exactly when two or more optima attain the optimum \cite{TropCramer}. Read that
way, $|G|=2$ is a tie at the outer, \emph{transversal} level: by \eqref{eq:topgen} the component of
the numerator in top degree is the signed sum of the two optimal transversal contributions, and the
whole question of \S\ref{sec:rigidity} is whether they cancel. We put it no more strongly than that.
Tropical singularity of the full determinant is a statement about assignments, and inside a single
$A(T_g)$ many assignments already attain the maximum --- precisely the ones producing the two
alternants of \eqref{eq:Amaxgen} --- which is why $[A]_{\deg(T)}$ is a nonzero product rather than a
cancellation. The tie that matters here is between the $T_g$, not between assignments. What is \emph{not} standard,
and is the reason the parity of $t$ governs everything, is \eqref{eq:congrgen}: it is a congruence,
not an inequality, and it is what bounds the number of optima by two.
\end{remark}

\begin{figure}[!ht]
\centering
\includegraphics[width=\textwidth]{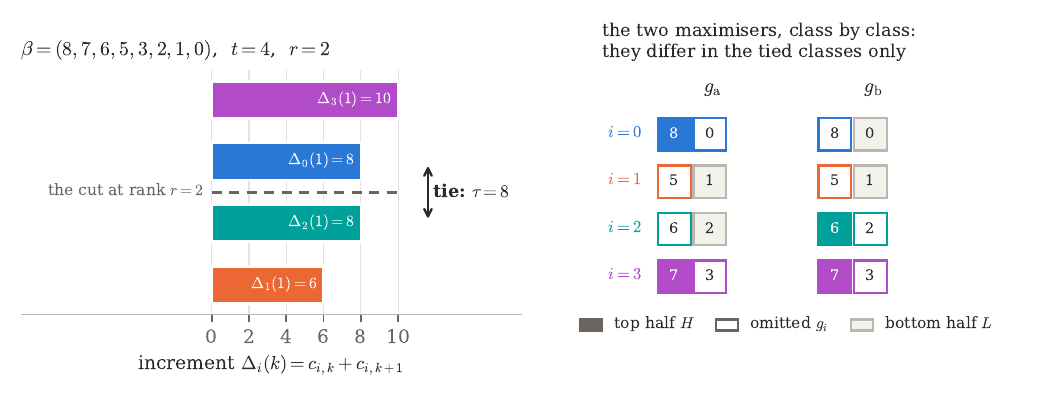}
\caption{Theorem \ref{thm:Ggen} computed on one shape rather than stated. For
$\beta=(8,7,6,5,3,2,1,0)$ at $t=4$, $r=2$ the four excess classes are $\{8,0\}$, $\{5,1\}$,
$\{6,2\}$, $\{7,3\}$ and the four increments are $\Delta_3(1)=10$, $\Delta_0(1)=8$, $\Delta_2(1)=8$,
$\Delta_1(1)=6$. A maximiser takes the $r=2$ largest; the second and third are equal, so
$\tau=\Delta_{(2)}=\Delta_{(3)}=8$ and there are exactly two of them. The tied pair belongs to the
classes $0$ and $2=0+t/2$, which is \eqref{eq:congrgen} and not a coincidence: no value is shared by
increments of classes other than $i$ and $i+t/2$. On the right the two maximisers are drawn class by
class in the prefix form of Step 2 --- reading a class decreasingly, first the elements of the top
half, then the omitted $g_i$, then those of the bottom half --- which makes visible that they differ
in the tied classes and nowhere else. Colour encodes the class, not the rank. Computed by
\texttt{anc/fig\_v2.py}.}
\label{fig:increments}
\end{figure}

\begin{corollary}[a unique maximiser]\label{cor:uniquegen}
If $|G|=1$ then $\Phi_{t,r}\not\equiv0$.
\end{corollary}

\begin{proof}
By \eqref{eq:topgen} the component of the numerator in degree $D_1$ is $\pm V\,P(T_g)\ne0$ for the
unique $g$. A Laurent polynomial with a nonzero homogeneous component is nonzero, so the numerator
is nonzero and hence so is $\Phi_{t,r}$.
\end{proof}

\begin{corollary}[the middle block]\label{cor:midgen}
Write $\mathcal{S}=\{s_1>\dots>s_n\}$, $n=2r+e$. If the block $\{s_{r+1},\dots,s_{r+e}\}$ contains
exactly one element of each excess class, then $\Phi_{t,r}\not\equiv0$.
\end{corollary}

\begin{proof}
$\deg(\mathcal{S}\setminus g)\le(s_1+\dots+s_r)-(s_{n-r+1}+\dots+s_n)$, since the top $r$ elements
of $\mathcal{S}\setminus g$ are dominated termwise by $s_1,\dots,s_r$ and its bottom $r$ dominate
$s_{n-r+1},\dots,s_n$. Equality forces $\mathcal{S}\setminus g$ to contain both extreme blocks, that
is $g\subseteq\{s_{r+1},\dots,s_{r+e}\}$, and $|g|=e$ forces equality of the two sets. So the bound
is attained by exactly one term and Corollary \ref{cor:uniquegen} applies.
\end{proof}

\begin{corollary}[odd $t$]\label{cor:oddgen}
Let $t$ be odd and $r\ge1$. Then $\Phi_{t,r}(\lambda;z)\equiv0$ if and only if some residue class
modulo $t$ is absent from $\beta(\lambda)$.
\end{corollary}

\begin{proof}
If some $n_i=0$ then $\mathcal T=\varnothing$ in Lemma \ref{lem:L3gen} and the numerator vanishes.
Conversely, if all $n_i\ge1$ then $|G|=1$ by Theorem \ref{thm:Ggen}(c) and Corollary
\ref{cor:uniquegen} applies.
\end{proof}

\begin{remark}
For $r=1$ Corollary \ref{cor:oddgen} is Proposition \ref{rem:lattice}, proved there by showing that
the concentric clause is vacuous for odd $t$ --- an argument internal to the two-class profile that a
single reciprocal pair forces. What is new is that the statement survives for every $r\ge1$, and by
a route that does not mention the profile at all: the congruence \eqref{eq:congrgen}. At $r=0$ it is
false, since Littlewood's rule gives $s_\lambda(\zt)=0$ whenever the $t$-core is nonempty, odd $t$
included; and the observation that $\prod\zt=+1$ for odd $t$ does not by itself deliver it, because
the alphabet is a subtorus of dimension $r$ with the remaining coordinates frozen.
\end{remark}

\subsection{The criterion at \texorpdfstring{$t=2$}{t=2}: rigidity and the reflection}\label{sec:rigidity}

Corollary \ref{cor:uniquegen} settles every shape with a unique maximiser. What follows is the other
half: when $|G|=2$ the two terms of \eqref{eq:topgen} can only cancel if the two transversals are
reflections of one another, and that rigidity comes from a single external input.

\begin{theorem}[Purbhoo--van Willigenburg {\cite[Thm.~2.5]{PvW}}]\label{thm:PvW}
For partitions with at most $n$ parts, writing $\lambda^{-}:=\lambda/(\lambda_n)^{n}$,
\[
  s_\lambda(x_1,\dots,x_n)\,s_\mu(x_1,\dots,x_n)=s_\nu(x_1,\dots,x_n)\,s_\rho(x_1,\dots,x_n)
\]
if and only if $\lambda_n+\mu_n=\nu_n+\rho_n$ and $\{\lambda^{-},\mu^{-}\}=\{\nu^{-},\rho^{-}\}$ as
multisets.
\end{theorem}

\begin{remark}
Theorem 3 of \cite{Rajan} gives the statement for $n$ factors, but under the hypothesis that all the
highest weights be nonzero. Theorem \ref{thm:PvW} needs no such hypothesis, and that is why it is the
one used here: the degenerate case, one factor trivial, is exactly what Lemma \ref{lem:dictgen}
produces whenever $\alpha^{*}$ or $\tilde\alpha$ is empty. Neither statement can be shortcut by unique
factorisation, since Schur polynomials are not irreducible: $s_{(3)}(x,y)=(x+y)(x^2+y^2)$.
\end{remark}

\subsubsection*{The dictionary}

Write $T=(u_1>\dots>u_{2r})$ with halves $H=(h_1>\dots>h_r)$ and $L=(l_1>\dots>l_r)$, let
$\delta=(r-1,r-2,\dots,0)$, so that $a_\delta(z)=\det(z_j^{\,r-i})$ is the Vandermonde determinant in
the notation of \S\ref{sec:extremal}, and put
\[
  \alpha=H-\delta,\qquad \tilde\alpha=\alpha-(\alpha_r)^{r},\qquad
  L^{*}=(l_1-l_r>\dots>l_1-l_1),\qquad \alpha^{*}=L^{*}-\delta .
\]

\begin{lemma}[the dictionary]\label{lem:dictgen}
$\alpha^{*}_r=0$ and
\begin{equation}\label{eq:dictgen}
  P(T)\;=\;(-1)^{\binom r2}\Bigl(\textstyle\prod_j z_j\Bigr)^{h_r-l_1}
  a_\delta(z)^2\,s_{\tilde\alpha}(z)\,s_{\alpha^{*}}(z).
\end{equation}
In particular $P$ is unchanged by $T\mapsto T+m$ and, by Lemma \ref{lem:reflgen}, by $T\mapsto c-T$.
\end{lemma}

\begin{proof}
$a_H(z)=s_\alpha(z)a_\delta(z)$ is the bialternant formula. For the second factor, multiplying column
$j$ of $\det(z_j^{-l_i})$ by $z_j^{l_1}$ gives $(\prod_j z_j^{l_1})a_L(z^{-1})=\det(z_j^{\,l_1-l_i})$,
whose row exponents increase with $i$; reversing the $r$ rows costs $(-1)^{\binom r2}$ and produces
$a_{L^{*}}(z)=s_{\alpha^{*}}(z)a_\delta(z)$. Finally $s_\alpha=(\prod_j z_j)^{\alpha_r}s_{\tilde\alpha}$
and $\alpha_r=h_r$.
\end{proof}

\begin{proposition}[the orbit of $P$]\label{prop:orbitgen}
Let $T,\widetilde T$ be $2r$-sets of distinct integers with $P(T)=\pm P(\widetilde T)$. Then
$\widetilde T=T+m$ for some $m\in\ZZ$, or $\widetilde T=c-T$ for some $c\in\ZZ$.
\end{proposition}

\begin{proof}
Two remarks first. By \eqref{eq:dictgen} both sides carry the same nonzero factor
$(-1)^{\binom r2}a_\delta(z)^{2}$; cancel it before comparing anything. What is left on each side is
a monomial times $s_{\tilde\alpha}s_{\alpha^{*}}$, a nonzero polynomial with nonnegative
coefficients, so the sign is $+$. The order matters: $a_\delta^{2}$ does \emph{not} have
nonnegative coefficients, and reading the sign off before cancelling it proves nothing. And the
monomial must be reabsorbed: with $d=h_r-l_1\ge1$, which holds because
$T$ is strictly decreasing, $(\prod_j z_j)^{d}s_{\tilde\alpha}=s_{\tilde\alpha+(d^{r})}$, and it is
this $d$ that plays the role of $\lambda_n+\mu_n$ in Theorem \ref{thm:PvW}, $\alpha^{*}$ being already
normalised. Cancelling $a_\delta^2$ and applying Theorem \ref{thm:PvW} with $n=r$ to
$(\lambda,\mu)=(\tilde\alpha+(d^{r}),\alpha^{*})$ and to the corresponding pair $(\nu,\rho)$ for
$\widetilde T$, the criterion says that the integer $d$ agrees and that
$\{\tilde\alpha,\alpha^{*}\}$ agrees as a multiset. Agreement in the same order says that $H,\widetilde
H$ and $L,\widetilde L$ are translates, and the equality of the two integers $d$ forces the same
translation on both halves, that is $\widetilde T=T+m$. Agreement after swapping says that $H$ is, up
to translation, the reverse of $\widetilde L$ and $L$ that of $\widetilde H$, that is $\widetilde
T=c-T$.
\end{proof}

\subsubsection*{The two maximisers reflect one another}

From here to the end of the subsection assume $|G|=2$ and write $G=\{g_{\mathrm a},g_{\mathrm b}\}$,
$T_{\mathrm a}=\mathcal{S}\setminus g_{\mathrm a}$, $T_{\mathrm b}=\mathcal{S}\setminus g_{\mathrm b}$,
with halves $T_{\mathrm a}=H_{\mathrm a}\sqcup L_{\mathrm a}$ and
$T_{\mathrm b}=H_{\mathrm b}\sqcup L_{\mathrm b}$, and
\[
  \mathcal{K}=T_{\mathrm a}\cap T_{\mathrm b}.
\]
By Theorem \ref{thm:Ggen}(b) the two transversals differ exactly in the tied classes $i$ and
$i'=i+t/2$, say
\begin{equation}\label{eq:swapgen}
  g_{\mathrm a}\ni p_1=c_{i,k+1},\ q_1=c_{i',k'},\qquad
  g_{\mathrm b}\ni p_2=c_{i,k},\ q_2=c_{i',k'+1},
\end{equation}
so that $p_1+p_2=q_1+q_2=\tau$ by \eqref{eq:incrgen}: the tie value is the sum of each swapped pair.

\begin{proposition}[translation forces reflection]\label{prop:translgen}
If $T_{\mathrm b}=T_{\mathrm a}+m$ with $m\ne0$, then $T_{\mathrm b}=\tau-T_{\mathrm a}$.
\end{proposition}

\begin{proof}
$|T_{\mathrm a}\cap(T_{\mathrm a}+m)|=|\mathcal{K}|=2r-2$. Put an edge $x\to x+m$ whenever both $x$ and
$x+m$ lie in $T_{\mathrm a}$; since $m\ne0$ each vertex has at most one outgoing and one incoming
edge, so the graph is a disjoint union of directed paths. It has $2r$ vertices and, by the displayed
cardinality, exactly $2r-2$ edges, hence exactly two components. So $T_{\mathrm a}$ is a union of
exactly two maximal arithmetic progressions of step $m$, say
$T_{\mathrm a}=\{p,p+m,\dots,p+(u-1)m\}\cup\{q,\dots,q+(v-1)m\}$ with $u+v=2r$. Then
$T_{\mathrm a}\setminus T_{\mathrm b}=\{p,q\}$ and $T_{\mathrm b}\setminus T_{\mathrm a}=\{p+um,q+vm\}$.
The tie says that $x\mapsto\tau-x$ carries the first pair onto the second, and there are two
matchings. If $\tau-p=p+um$ and $\tau-q=q+vm$ then
$\tau-T_{\mathrm a}=\{p+m,\dots,p+um\}\cup\{q+m,\dots,q+vm\}=T_{\mathrm b}$. If $\tau-p=q+vm$ and
$\tau-q=p+um$ then subtracting gives $um=vm$, hence $u=v$, and again
$\tau-T_{\mathrm a}=T_{\mathrm b}$.
\end{proof}

\begin{proposition}[the centre is the tie value]\label{prop:centregen}
If $T_{\mathrm b}=c-T_{\mathrm a}$ then $c=\tau$.
\end{proposition}

\begin{proof}
$x\mapsto c-x$ is an involution carrying $T_{\mathrm a}$ to $T_{\mathrm b}$, hence $T_{\mathrm b}$ to
$T_{\mathrm a}$, hence $\mathcal{K}$ to $\mathcal{K}$; therefore it carries
$T_{\mathrm a}\setminus\mathcal{K}$ onto $T_{\mathrm b}\setminus\mathcal{K}$. By \eqref{eq:swapgen}
these two sets are $T_{\mathrm a}\setminus\mathcal{K}=\{p_2,q_2\}$ and
$T_{\mathrm b}\setminus\mathcal{K}=\{p_1,q_1\}$, since $T_{\mathrm a}$ omits exactly $p_1,q_1$ and
$T_{\mathrm b}$ omits exactly $p_2,q_2$. Either $c-p_2=p_1$, and then $c-q_2=q_1$, whence
$c=p_1+p_2=\tau$; or $c-p_2=q_1$ and $c-q_2=p_1$, and adding gives $2c=p_1+p_2+q_1+q_2=2\tau$, again
$c=\tau$. In the second case moreover $c=p_2+q_1\equiv i+i'\equiv2i+t/2$ while $\tau\equiv2i$ by
\eqref{eq:congrgen}, so $c=\tau$ would force $t/2\equiv0\pmod t$, impossible for $t\ge2$: the crossed
matching does not occur.
\end{proof}

\begin{corollary}[the reflection]\label{cor:reflectgen}
If $\bigl[\det(x_i^{\beta_j})\bigr]_{D_1}=0$ then $|G|=2$ and $T_{\mathrm b}=\tau-T_{\mathrm a}$.
Consequently $\tau-\mathcal{K}=\mathcal{K}$ and, since $x\mapsto\tau-x$ also fixes $\{p_1,p_2\}$ and
$\{q_1,q_2\}$ setwise,
\begin{equation}\label{eq:Ssymgen}
  \tau-\bigl(\mathcal{K}\cup\{p_1,p_2,q_1,q_2\}\bigr)=\mathcal{K}\cup\{p_1,p_2,q_1,q_2\}.
\end{equation}
\end{corollary}

\begin{proof}
$|G|=1$ is excluded by Corollary \ref{cor:uniquegen}. With $|G|=2$, \eqref{eq:topgen} and
$w(g)=\pm1$ give $P(T_{\mathrm a})=\pm P(T_{\mathrm b})$; apply Proposition \ref{prop:orbitgen} and
then Propositions \ref{prop:translgen} and \ref{prop:centregen}. The identity
$\tau-\mathcal{K}=\mathcal{K}$ is the first line of the proof of Proposition \ref{prop:centregen}, and
$p_1+p_2=q_1+q_2=\tau$ is the tie.
\end{proof}

\begin{remark}[three constants that are easy to confuse]
$\tau=\Delta_{(r)}$ is the tie value, a datum of the greedy with nothing to do with symmetry; $c$ is a
centre of reflection, produced by rigidity with nothing to do with residues; $C=\min\mathcal{S}+\max\mathcal{S}$
is a statistic of the beta set, and it is the one condition \textup{(ii)} uses. That $c=\tau$
\textup{(}Proposition \ref{prop:centregen}\textup{)} and $C=\tau$ \textup{(}Corollary
\ref{cor:Ctaugen}\textup{)} are theorems, not notation, and both need the vanishing hypothesis:
writing $C$ where $\tau$ belongs proves less than it appears to. We made that slip; hence the three
symbols.
\end{remark}

\subsubsection*{The centre of the beta set}

The two maximisers agree outside the tied classes, so their common omitted part
\begin{equation}\label{eq:gcomgen}
  g_{\mathrm{com}}:=g_{\mathrm a}\cap g_{\mathrm b},\qquad
  \mathcal{S}=\mathcal{K}\ \sqcup\ g_{\mathrm{com}}\ \sqcup\ \{p_1,p_2,q_1,q_2\}
\end{equation}
is empty exactly when $e=2$. By \eqref{eq:Ssymgen} the set $\mathcal{S}\setminus g_{\mathrm{com}}$ is
stable under $x\mapsto\tau-x$, but $g_{\mathrm{com}}$ need not be, so $\mathcal{S}$ itself need not be.
That is more than is needed: only the two \emph{extremes} of $\mathcal{S}$ have to avoid
$g_{\mathrm{com}}$, and they do. Figure \ref{fig:reflection} draws both the reflection and the
shape that has $|G|=2$ without it, which is what the next proposition must exclude.

\begin{figure}[!ht]
\centering
\includegraphics[width=\textwidth]{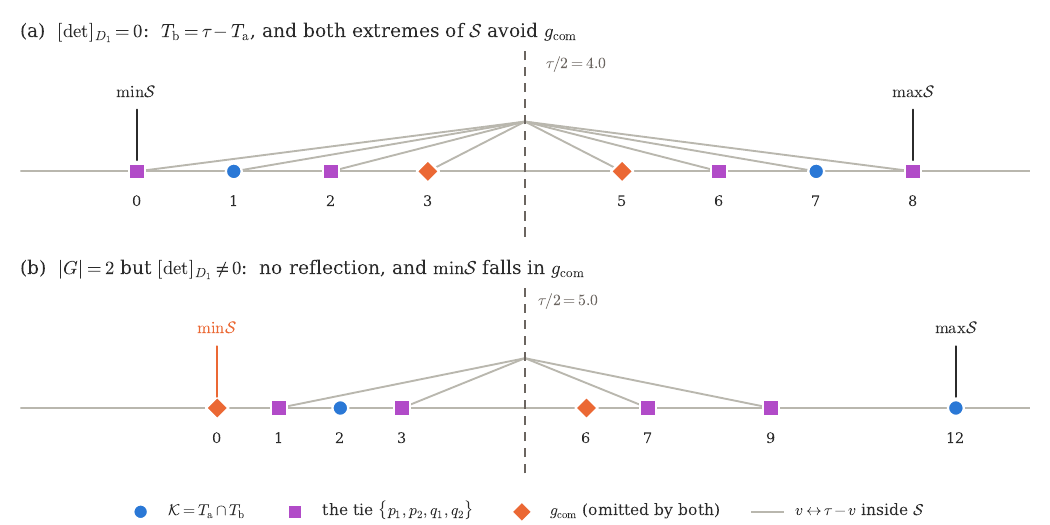}
\caption{Corollary \ref{cor:reflectgen} and Proposition \ref{prop:A1gen} on two shapes, both with
$t=4$, $r=2$ and $|G|=2$. Each panel draws $\mathcal{S}$ on a line, coloured by which part of
\eqref{eq:gcomgen} the value belongs to; an arc joins $v$ to $\tau-v$ whenever both lie in
$\mathcal{S}$, and the dashed line is the centre $\tau/2$. \emph{(a)} $\beta=(8,7,6,5,3,2,1,0)$:
$\tau=8$, $\mathcal{K}=\{1,7\}$, $g_{\mathrm{com}}=\{3,5\}$, tie $\{0,2,6,8\}$. Here
$T_{\mathrm b}=\tau-T_{\mathrm a}$, so $\mathcal{S}\setminus g_{\mathrm{com}}$ is symmetric about
$\tau/2$, and both extremes of $\mathcal{S}$ avoid $g_{\mathrm{com}}$, so $C=0+8=\tau$. Here
$g_{\mathrm{com}}=\{3,5\}$ turns out to be symmetric as well --- the arc joining its two members is
drawn like any other --- which is what Conjecture \ref{conj:gcom} asserts in general and what
\S\ref{sec:rigidity} does not prove: the reflection is established on $\mathcal{S}\setminus
g_{\mathrm{com}}$ and observed on the rest. \emph{(b)}
$\beta=(12,9,7,6,3,2,1,0)$: $\tau=10$ and $|G|=2$ again, but the two maximisers are not reflections
of one another; $\min\mathcal{S}=0$ falls in $g_{\mathrm{com}}$, and $C=0+12=12\neq\tau$. Panel (b)
is the reason Proposition \ref{prop:A1gen} cannot be proved from $|G|=2$ alone. Computed by
\texttt{anc/fig\_v2.py}.}
\label{fig:reflection}
\end{figure}

\begin{proposition}[the extremes avoid the common part]\label{prop:A1gen}
Suppose $\bigl[\det(x_i^{\beta_j})\bigr]_{D_1}=0$. Then
$\max\mathcal{S}\notin g_{\mathrm{com}}$ and $\min\mathcal{S}\notin g_{\mathrm{com}}$.
\end{proposition}

\begin{proof}
By Corollary \ref{cor:reflectgen}, $T_{\mathrm b}=\tau-T_{\mathrm a}$. Since $x\mapsto\tau-x$ reverses
order, it carries the bottom half of $T_{\mathrm a}$ to the top half of $T_{\mathrm b}$ and back:
\begin{equation}\label{eq:halvesgen}
  H_{\mathrm b}=\tau-L_{\mathrm a},\qquad L_{\mathrm b}=\tau-H_{\mathrm a}.
\end{equation}
Recall from Step 2 of the proof of Theorem \ref{thm:Ggen} that a maximiser is described by the
integers $k_i$: inside class $i$, read decreasingly, the first $k_i$ excess values lie in the top half,
the next one is the omitted $g_i$, and the remaining ones lie in the bottom half. Recall also from
Theorem \ref{thm:Ggen}(a) that a maximiser selects every increment $>\tau$, that the increments of a
class are selected in prefix order, and that the only increments equal to $\tau$ are the two tied ones,
which lie in the two tied classes.

\emph{The maximum.} Put $m=\max\mathcal{S}$ and let $i$ be its class, so $m=c_{i,1}$ and $n_i\ge2$;
write $c=c_{i,2}$. Then $m\in g_{\mathrm a}\cap g_{\mathrm b}$ if and only if $k_i=0$ in both
maximisers. Assume it. Then class $i$ is not tied: a maximiser taking the tied increment of a tied
class has $k_i\ge1$ there. Hence no increment of class $i$ is $\ge\tau$, and in particular
\[
  \Delta_i(1)=m+c<\tau .
\]
On the other hand $k_i=0$ places $c$ in the bottom half of both maximisers, so $c\in L_{\mathrm b}$,
and \eqref{eq:halvesgen} gives $\tau-c\in H_{\mathrm a}\subseteq\mathcal{S}$, whence
$\tau-c\le\max\mathcal{S}=m$, that is $\tau\le m+c$. The two displayed inequalities contradict each
other.

\emph{The minimum} is the mirror image. Put $m^{-}=\min\mathcal{S}$, let $i^{-}$ be its class,
$n=n_{i^{-}}\ge2$ and $c'=c_{i^{-},n-1}$, so that $m^{-}=c_{i^{-},n}$. Then
$m^{-}\in g_{\mathrm a}\cap g_{\mathrm b}$ if and only if $k_{i^{-}}=n-1$ in both, that is if and only
if both maximisers select \emph{every} increment of class $i^{-}$; again class $i^{-}$ cannot be tied,
since the maximiser that does not take the tied increment of a tied class omits one there. Hence every
increment of class $i^{-}$ exceeds $\tau$, in particular $\Delta_{i^{-}}(n-1)=m^{-}+c'>\tau$. And
$k_{i^{-}}=n-1$ places $c'$ in the top half of both, so $c'\in H_{\mathrm b}$, and
\eqref{eq:halvesgen} gives $\tau-c'\in L_{\mathrm a}\subseteq\mathcal{S}$, whence
$\tau-c'\ge\min\mathcal{S}=m^{-}$, that is $\tau\ge m^{-}+c'$. Contradiction.
\end{proof}

\begin{corollary}[the centre is the tie value]\label{cor:Ctaugen}
If $\bigl[\det(x_i^{\beta_j})\bigr]_{D_1}=0$ then
$C:=\min\mathcal{S}+\max\mathcal{S}=\tau$. Moreover $t$ is even and the two tied classes are exactly
the two solutions of $2\kappa\equiv C\pmod t$, both lying in $\mathcal{E}$.
\end{corollary}

\begin{proof}
By \eqref{eq:Ssymgen} the set $\mathcal{S}\setminus g_{\mathrm{com}}$ is stable under
$x\mapsto\tau-x$, so its largest and smallest elements add up to $\tau$; by Proposition
\ref{prop:A1gen} the two extremes of $\mathcal{S}$ lie in that set, hence they \emph{are} its largest
and smallest elements, and $C=\tau$. Theorem \ref{thm:Ggen}(b) gives $t$ even and the two tied classes
$i,i+t/2$, both in $\mathcal{E}$; by \eqref{eq:congrgen} $\tau\equiv2i\pmod t$, so
$2\kappa\equiv\tau=C$ has exactly the two solutions $\kappa=i$ and $\kappa=i+t/2$.
\end{proof}

\begin{corollary}[a necessary condition, for every even $t$ and every $r$]\label{cor:necgen}
Let $t\ge2$ and $r\ge1$, and suppose $\Phi_{t,r}(\lambda;z)\equiv0$ with every $n_i\ge1$. Then $t$
is even, $|G|=2$, $\min\mathcal{S}+\max\mathcal{S}=\tau$, and $\mathcal{S}\setminus g_{\mathrm{com}}$
is symmetric about $\tau/2$.
\end{corollary}

\begin{proof}
$\Phi_{t,r}$ and the numerator differ by a nonzero factor, so the numerator vanishes identically and
in particular $\bigl[\det(x_i^{\beta_j})\bigr]_{D_1}=0$. Apply Corollaries \ref{cor:reflectgen} and
\ref{cor:Ctaugen} and \eqref{eq:Ssymgen}.
\end{proof}

\noindent This is the part of \S\S\ref{sec:extremal}--\ref{sec:rigidity} that survives outside
$t=2$. What it does \emph{not} give is the step that closes $t=2$, where $e=2$ forces
$g_{\mathrm{com}}=\varnothing$ and the symmetry of $\mathcal{S}\setminus g_{\mathrm{com}}$ becomes
the symmetry of $\beta$ itself. For $t\ge4$ the common part $g_{\mathrm{com}}$ can be non-empty, and
the argument above says nothing about it: the reflection is established on
$\mathcal{S}\setminus g_{\mathrm{com}}$ only, so it constrains $\beta$ without determining it. That
$g_{\mathrm{com}}$ is symmetric too is what the measurements show and what Conjecture
\ref{conj:gcom} states; it is not proved here.

\subsubsection*{The condition is sufficient, for every \texorpdfstring{$t$}{t} and every
\texorpdfstring{$r$}{r}}

A necessary condition invites its converse, and half of it can be had outright. The reflection that
Corollary \ref{cor:necgen} produces from the top stratum turns out, when it holds on all of
$\mathcal{S}$, to be a symmetry of \emph{every} term of the Laplace sum and not only of the two
maximisers; and an involution with no fixed point cancels what it pairs. Write
\[
C:=\min\mathcal{S}+\max\mathcal{S},\qquad\text{and write } x\mapsto C-x
\text{ for the reflection it centres.}
\]

\begin{theorem}[the sufficient direction, all $t$ and $r$]\label{thm:suff}
Let $t\ge2$, $r\ge1$, $N=t+2r$, and let every residue class be occupied. Suppose
\begin{equation}\label{eq:suffcond}
C-\mathcal{S}=\mathcal{S}
\qquad\text{and}\qquad
\Delta_i(k)=C\ \text{for some } i\in\mathcal{E},\ k .
\end{equation}
Then $\Phi_{t,r}(\lambda;z)\equiv0$.
\end{theorem}

The two clauses do two different jobs, and it is worth saying which before the proof: the first makes
the reflection act on the transversals, and the second makes that action free. Figure
\ref{fig:pairing} is the whole mechanism on one shape, with a shape that satisfies the first clause
and not the second beside it.

\begin{figure}[!ht]
\centering
\includegraphics[width=\textwidth]{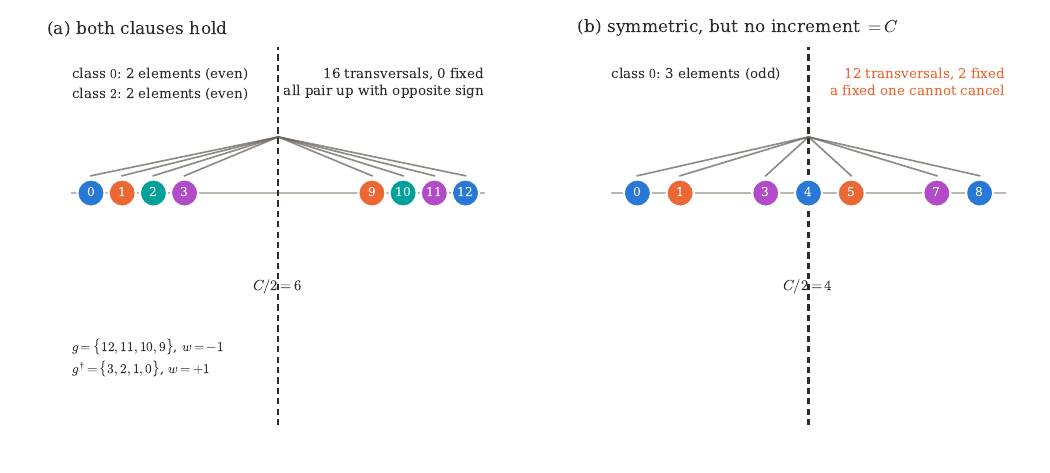}
\caption{Why the sum cancels, and what the second clause is for. Each panel draws $\mathcal{S}$ on a
line, coloured by residue class, with an arc joining $v$ to $C-v$ and the dashed axis at $C/2$.
\emph{(a)} $\beta=(12,11,10,9,3,2,1,0)$ at $t=4$, $r=2$: both classes fixed by the reflection have
\emph{even} size, so neither contains $C/2$; the $16$ transversals pair off with no fixed point, and
every pair carries opposite signs, so every term of \eqref{eq:laplacegen} is cancelled by another.
One such pair is printed with its two signs, computed and not asserted. \emph{(b)} a shape with
$C-\mathcal{S}=\mathcal{S}$ but no increment equal to $C$, found by search rather than chosen: there
the fixed class has \emph{odd} size, so it contains $C/2$, and two transversals are their own
reflections. A term paired with itself cannot cancel, and the sum survives. That is the whole of
what the second clause buys. Computed by \texttt{anc/fig\_involution2.py}.}
\label{fig:pairing}
\end{figure}

\begin{lemma}[the reflected transversal]\label{lem:reflact}
Assume $C-\mathcal{S}=\mathcal{S}$. Then $x\mapsto C-x$ carries the class $i$ onto the class
$i^{*}:=C-i \bmod t$, both in $\mathcal{E}$ or neither; so
\[
g^{\dagger}:=\bigl(C-(g\cap\mathcal{S})\bigr)\cup\bigl(g\setminus\mathcal{S}\bigr)
\]
is again a transversal, $g\mapsto g^{\dagger}$ is an involution on them, and $T_{g^{\dagger}}=C-T_g$.
\end{lemma}

\begin{proof}
Every element of the class $i$ is $\equiv i$, so $x\mapsto C-x$ sends it into the class $i^{*}$; since
$C-\mathcal{S}=\mathcal{S}$ and the map is injective, it carries the class $i$ \emph{onto} the
class $i^{*}$ whenever $i\in\mathcal{E}$, which forces $i^{*}\in\mathcal{E}$ and $n_i=n_{i^{*}}$. The
singleton classes are not touched: their single element is frozen in every transversal. Finally
$T_{g^{\dagger}}=\mathcal{S}\setminus g^{\dagger}=C-(\mathcal{S}\setminus g)=C-T_g$.
\end{proof}

\begin{lemma}[the reflected minor]\label{lem:Arefl}
$A(C-T)=A(T)$ for every $2r$-set $T$ and every $C$.
\end{lemma}

\begin{proof}
Extracting $x_i^{C}$ from each of the $2r$ rows multiplies the determinant by
$\prod_j z_j^{C}z_j^{-C}=1$ and leaves the matrix of $T$ with $x_i$ replaced by $x_i^{-1}$, which is
the matrix of $T$ with the two rows of each reciprocal pair exchanged: a sign $(-1)^{r}$. And
$C-T$ listed decreasingly reverses the column order of $T$: a sign $(-1)^{\binom{2r}{2}}=(-1)^{r}$,
since $2r-1$ is odd. The two cancel.
\end{proof}

\noindent The two moves in that proof are not ours, and it is worth saying whose they are, because
they are the right instrument and we chose them for that reason. Scaling each row by a power so as
to replace $x_i$ by $\bar x_i$, and then reversing a block of columns to collect
$(-1)^{n(n-1)/2}$, is the manoeuvre Ayyer and Behrend use on the bialternant of a
self-complementary shape \cite[\S3]{AB19}; they describe their proofs as relying only on the
determinant expressions ``and the application of standard determinant operations'', and this is one
of those operations. Albion performs the same reversal on the root-of-unity side
\cite{Alb23,Alb25}, and at the level of characters the same fact is the invariance of $sp_\lambda$
under $x_i\mapsto\bar x_i$, recorded in \cite{AB19} and put to work in \cite{AF20}.

Its usefulness here is worth spelling out, since it is what makes the whole section short. A
reflected minor could be attacked by expanding it, and then the singleton classes and the reflected
ones would have to be tracked term by term --- which is how we first tried, and it does not end.
The Ayyer--Behrend manoeuvre instead disposes of the reflection in \emph{two} sign computations that
do not look at the entries at all, and it is robust in exactly the way we need: it never uses that
the whole alphabet is being reflected, so it survives being applied to a \emph{part} of the beta
set, with the rest held fixed. That robustness is the reason Lemma \ref{lem:Arefl} is four lines
instead of a section, and it belongs to the tool rather than to us. It is also a small piece of
evidence for the thesis of \S\ref{sec:background}: the two lines this paper sits between share more
machinery than their statements suggest.

The next statement is the combinatorial heart of the argument, and it needs only the first clause of
\eqref{eq:suffcond}. Reflecting a beta set \emph{entirely} is classical --- it is the
complementation identity \eqref{eq:compl}, and the determinant steps behind it are the standard ones
just cited. What is asked here is a reflection of the excess part alone, with the singleton classes
held fixed and interleaved among the reflected values; the two involutions it produces, one of
$\beta$ and one of $\ZZ/t$, are what the sign turns out to depend on.

\begin{lemma}[the partial-reflection sign formula]\label{lem:signformula}
Assume $C-\mathcal{S}=\mathcal{S}$, and put
\[
\nu=\#\{x\in\mathcal{S}:2x=C\},\qquad \eta=\#\{i\in\mathcal{E}:2i\equiv C \bmod t\}.
\]
Then for every transversal $g$
\begin{equation}\label{eq:wrefl}
\frac{w(g^{\dagger})}{w(g)}\;=\;(-1)^{\,e-\frac{\nu+\eta}{2}} .
\end{equation}
In particular the ratio does not depend on $g$, nor on $r$, nor on $N$.
\end{lemma}

\begin{corollary}\label{cor:whenminus}
Under $C-\mathcal{S}=\mathcal{S}$, the ratio \eqref{eq:wrefl} equals $-1$ if and only if $\eta=2$.
\end{corollary}

\begin{proof}
An involution has as many fixed points as its set has elements, modulo two, so $\nu\equiv|\mathcal{S}|$
and $\eta\equiv e$; and $|\mathcal{S}|=2r+e$ by \eqref{eq:excessgen}, so $\nu\equiv e$ as well. If $e$
is odd then $\nu=\eta=1$ and the exponent is $e-1$, even. If $e$ is even then $\nu=0$ and $\eta$ is
$0$ or $2$, giving the exponent $e$ or $e-1$; the first is even and the second odd.
\end{proof}

\noindent That independence is the part that looks unlikely and is the point: the singleton classes
sit at arbitrary positions among the reflected values, and moving a transversal moves them past one
another --- but each crossing they contribute to the column sum is undone by the crossing they
contribute to the residue word, so they leave no trace in the ratio.

\begin{lemma}[the reflected sign]\label{lem:signrefl}
Assume \eqref{eq:suffcond}. Then $\eta=2$, the number $e$ is even, $\nu=0$, and
$w(g^{\dagger})=-w(g)$ for every transversal $g$.
\end{lemma}

\begin{proof}[Proof of Lemma \ref{lem:signformula}]
Let $\mathcal{R}$ be $x\mapsto C-x$ on $\mathcal{S}$ and the identity on
$\beta\setminus\mathcal{S}$, and let $\mathcal{R}_t$ be $i\mapsto i^{*}$ on $\mathcal{E}$ and the
identity elsewhere; both are involutions, of $\beta$ and of $\ZZ/t$ respectively, with $\nu$ and
$\eta$ fixed points, so $\sgn(\mathcal{R})=(-1)^{(|\mathcal{S}|-\nu)/2}$ and
$\sgn(\mathcal{R}_t)=(-1)^{(e-\eta)/2}$. Order the columns of a transversal $g$ by residue and the
remaining ones by decreasing $\beta$; the sign of the resulting permutation of $\{1,\dots,N\}$ is
$w(g)$ up to a factor depending on $t$ alone. Applying $\mathcal{R}$ to that one-line notation
contributes $\sgn(\mathcal{R})$; restoring the residue order of the first block contributes
$\sgn(\mathcal{R}_t)$; and $\mathcal{R}$ reverses $T_g$ entirely, so restoring its decreasing order
contributes $(-1)^{\binom{2r}{2}}$. Multiplying the three gives \eqref{eq:wrefl}, in which no term
refers to $g$.
\end{proof}

\begin{proof}[Proof of Lemma \ref{lem:signrefl}]
Counting fixed points of an involution gives
$|\mathcal{S}|\equiv\nu$ and $e\equiv\eta \pmod 2$, and $|\mathcal{S}|=2r+e$ by
\eqref{eq:excessgen}, so $e\equiv\nu$. Now the second clause of \eqref{eq:suffcond} says that two
\emph{adjacent} elements of some class $i_0$ sum to $C$; they are exchanged by the reflection, so
$i_0=i_0^{*}$, and writing the class decreasingly the exchange sends its $k$-th entry to its
$(n_{i_0}+1-k)$-th, whence $n_{i_0}=2k$ is even. The reflection $x\mapsto C-x$ has at most one fixed
point, namely $C/2$; since this class is stable under it and has even cardinality, its elements pair
off and it contains none, so $C/2\notin$ class $i_0$.

Before speaking of a second solution we must know there is one, and that is where the parity of $t$
is decided rather than assumed. Suppose $t$ were odd. Then $2$ is invertible modulo $t$, so
$i\mapsto i^{*}$ has $i_0$ as its \emph{only} fixed residue; the remaining classes of $\mathcal{E}$
pair off, so $e$ is odd, and $|\mathcal{S}|=2r+e$ is odd too. An involution on a set of odd size has
a fixed point, so $\nu=1$ and $C/2\in\mathcal{S}$; its residue $\kappa$ satisfies $2\kappa\equiv C$,
hence $\kappa=i_0$, putting $C/2$ in the class $i_0$ --- which we have just shown it misses. So $t$
is even, and $2i\equiv C$ has exactly the two solutions $i_0$ and $i_1=i_0+t/2$.

If $\nu=1$ then
$C/2\in\mathcal{S}$, and its residue is $i_0$ or $i_1$; not $i_0$, so $i_1\in\mathcal{E}$. If $\nu=0$
then $e$ is even, and $\mathcal{E}=\{i_0\}\sqcup\{\text{pairs}\}$ would make $e$ odd unless
$i_1\in\mathcal{E}$. Either way $\eta=2$, hence $e$ is even and $\nu=0$, and Corollary
\ref{cor:whenminus} gives the sign at once. Read off \eqref{eq:wrefl} directly, the exponent is
$e-\tfrac{0+2}2=e-1\equiv1 \pmod 2$.
\end{proof}

\begin{corollary}[the hypothesis, read on the residues]\label{cor:suffgeom}
Assume $C-\mathcal{S}=\mathcal{S}$. Then some $\Delta_i(k)=C$ if and only if $t$ is even and
\emph{both} solutions of $2i\equiv C \pmod t$ are excess classes. So \eqref{eq:suffcond} says: the
excess part of the beta set is symmetric about $C/2$, and the two residue classes that the
reflection fixes are both excess.
\end{corollary}

\begin{proof}
One direction is the first paragraph of the proof of Lemma \ref{lem:signrefl}. For the other, if
both fixed classes lie in $\mathcal{E}$ then $\eta=2$, so $e$ is even by $e\equiv\eta$, so
$|\mathcal{S}|=2r+e$ is even and $\nu=0$: no element of $\mathcal{S}$ is fixed. A fixed class is
therefore reflection-stable with no fixed point, hence of even size $2k$, and its $k$-th and
$(k+1)$-st entries are exchanged, so $\Delta_i(k)=C$.
\end{proof}

\noindent Read that way the hypothesis is the same shape as the conclusion of Corollary
\ref{cor:Ctaugen}, which is what makes the two halves of the section look at one another: there the
two tied classes \emph{turn out} to be the two solutions of $2\kappa\equiv C$; here we \emph{ask}
that they be excess.

\begin{proof}[Proof of Theorem \ref{thm:suff}]
By Lemma \ref{lem:reflact}, $g\mapsto g^{\dagger}$ is an involution on transversals with $T_{g^{\dagger}}=C-T_g$.
It has no fixed point: a fixed transversal would need $C-g_i=g_i$, that is $g_i=C/2$, in every
class with $i=i^{*}$, and the class $i_0$ of the proof above is such a class and misses $C/2$. By
Lemmas \ref{lem:Arefl} and \ref{lem:signrefl},
\[
w(g)A(T_g)+w(g^{\dagger})A(T_{g^{\dagger}})=w(g)A(T_g)-w(g)A(C-T_g)=0 ,
\]
so the sum of \eqref{eq:laplacegen} cancels in pairs and the numerator vanishes identically.
\end{proof}

\begin{remark}[what this is, and what it is not]\label{rem:suffscope}
Three things are worth separating. First, at $t=2$ the hypothesis \eqref{eq:suffcond} \emph{is}
branch {\rm(b)}: there $e=2$ forces $\mathcal{S}=\beta$, so $C-\mathcal{S}=\mathcal{S}$ is
self-complementarity, and the increment clause is the parity of the width --- the two statements
agree on every shape in the ranges of \S\ref{sec:verif}. So Theorem \ref{thm:suff} contains that
half of Theorem \ref{conj:crit}, and proves it again by a different mechanism. Second, for $t\ge4$
it is \emph{not} complementation: there $\mathcal{S}\subsetneq\beta$, $\lambda$ is not
self-complementary, and \eqref{eq:compl} says nothing. Third, the involution is global --- it pairs
every term of \eqref{eq:laplacegen}, not just the two maximisers, which is why it survives where the
extremal argument only produces a necessary condition.
\end{remark}

\begin{conjecture}[the criterion, in general]\label{conj:general}
Conversely, if $\Phi_{t,r}(\lambda;z)\equiv0$ and every residue class is occupied, then
\eqref{eq:suffcond} holds. Equivalently, for every $t$ and $r$,
\[
\Phi_{t,r}(\lambda;z)\equiv0
\iff
\text{some class is empty, or } C-\mathcal{S}=\mathcal{S}\text{ and some }\Delta_i(k)=C .
\]
\end{conjecture}

\noindent Under the vanishing hypothesis, most of what \eqref{eq:suffcond} asks is already known,
and saying which turns the conjecture into a single sentence about a single set. Corollary
\ref{cor:Ctaugen} gives $C=\tau$, so the tied increments equal $C$ and the second clause holds;
and \eqref{eq:Ssymgen} gives $C-(\mathcal{S}\setminus g_{\mathrm{com}})=\mathcal{S}\setminus
g_{\mathrm{com}}$. Since
$\mathcal{S}=(\mathcal{S}\setminus g_{\mathrm{com}})\sqcup g_{\mathrm{com}}$, all that is left is:

\begin{conjecture}[the common part is symmetric]\label{conj:gcom}
Under occupancy, if $\Phi_{t,r}(\lambda;z)\equiv0$ then $C-g_{\mathrm{com}}=g_{\mathrm{com}}$.
\end{conjecture}

\noindent Conjectures \ref{conj:general} and \ref{conj:gcom} are the same statement, and the second
is the useful form: it names the $e-2$ elements the argument has to reach, and no others. It also
says at once why $t=2$ closes and $t\ge4$ does not. At $t=2$ a tie is between the classes $i$ and
$i+1$, so both of them are excess and $e=2$ --- not because $e=2$ always holds there, since
\eqref{eq:excessgen} allows $e=1$ and it is common, but because $|G|=2$ forces it, and $g_{\mathrm{com}}$
is only defined when $|G|=2$. Hence $g_{\mathrm{com}}=\varnothing$ and there is nothing left to be
symmetric. The theorem there is not easier; there is simply no common part.

We state these as conjectures and not as theorems because only the direction proved above is
proved. The evidence is in \S\ref{sec:verif}: over $190\,443$ shapes in twenty-four configurations,
$t\le12$ and $r\le4$, the criterion has no false positive and no false negative; $43\,010$ of
them are redone by an evaluation sharing no code with this section, and a further sweep of
$12\,937$ shapes, with $131$ zeros, is redone in Sage, which shares neither code nor library with
either; on the ranges where a theorem
already decides it agrees with Theorem \ref{thm:main} at $r=1$ and with Theorem \ref{conj:crit} at
$t=2$ without a single disagreement, over $9913$ and $29\,508$ shapes; and a decoy that reads the
condition on $\beta$ instead of on $\mathcal{S}$ --- that is, self-complementarity of the shape
rather than of its excess part --- fails in every configuration of both sweeps. What is missing is one implication:
Corollary \ref{cor:necgen} gives the reflection on $\mathcal{S}\setminus g_{\mathrm{com}}$, and the
top stratum cannot give more --- there are shapes sharing $\tau$, $\mathcal{S}\setminus
g_{\mathrm{com}}$ and the tied classes, of which some vanish and some do not, so the missing
equation is not a refinement of \S\ref{sec:rigidity} but a statement about a lower stratum.

\begin{remark}[why not the obvious argument]
One tries first to exchange inside a \emph{single} maximiser: if $\max\mathcal{S}$ is omitted,
re-choose its class and hope the degree rises. That cannot work. A maximiser may omit
$\max\mathcal{S}$ --- precisely when $\max\mathcal{S}$ is one of the tied values $p_1,p_2$, so that
one maximiser omits it and the other keeps it --- so no argument reading one transversal at a time
concludes. Proposition \ref{prop:A1gen} is a statement about the \emph{pair}, which is why the proof
runs on \eqref{eq:halvesgen} and on nothing else. Table \ref{tab:extremal} records both the steps of
that proof and what happens to its conclusion when the hypothesis is dropped.
\end{remark}

\begin{table}[!ht]
\centering\small
\caption{The machinery of \S\S\ref{sec:extremal}--\ref{sec:rigidity}, measured. The population is
every shape with $|G|=2$ over $(t,r)\in\{4,6,8,10\}\times\{2,3\}$ within the widths swept. The last
row is the one that matters for Proposition \ref{prop:A1gen}: drop the hypothesis
$[\det]_{D_1}=0$ and the conclusion fails, so it is not carried for free.}
\label{tab:extremal}
\begin{tabular}{@{}llrl@{}}
\toprule
statement & population & result & script \\
\midrule
\rowcolor{hband} greedy parametrisation $=$ brute force & $4049$ shapes & \textcolor{hproved}{$0$ fail} & \texttt{a1\_proof} \\
$A$ is the top half of $T_g$ (Step 2) & same & \textcolor{hproved}{$0$ fail} & \texttt{a1\_proof} \\
\rowcolor{hband} $T_{\mathrm b}=\tau-T_{\mathrm a}$ & $2522$ shapes & \textcolor{hproved}{$0$ fail} & \texttt{a1\_proof} \\
$v\in L_{\mathrm b}\Rightarrow\tau-v\in H_{\mathrm a}$ & $3034$ values & \textcolor{hproved}{$0$ fail} & \texttt{a1\_proof} \\
\rowcolor{hband} the extremes avoid $g_{\mathrm{com}}$ & $2522$ shapes & \textcolor{hproved}{$0$ fail} & \texttt{a1\_proof} \\
the displayed identities of these subsections & $11$ identities & \textcolor{hproved}{$0$ fail} & \texttt{v2\_formulas} \\
\midrule
\rowcolor{hband} the same, with $[\det]_{D_1}=0$ dropped & $24718$ shapes & \textcolor{hverif}{$414$ fail} & \texttt{a1\_proof} \\
\bottomrule
\end{tabular}
\end{table}

\begin{remark}
Corollary \ref{cor:reflectgen} is doing real work in Proposition \ref{prop:A1gen}, and $|G|=2$ alone
does not suffice. For $t=4$, $r=2$ and $\beta=(200,199,198,197,194,193,191,4)$ the excess classes are
$\{200,4\}$, $\{197,193\}$, $\{198,194\}$, $\{199,191\}$ and the increments are $392,390,390,204$, so
$\tau=390$, the tie is between the classes $1$ and $3$, and $|G|=2$. But $\Delta_0(1)=204<\tau$, so
$k_0=0$ in both maximisers and $\max\mathcal{S}=200$ is omitted by both: the conclusion of Proposition
\ref{prop:A1gen} fails. Here $T_{\mathrm a}=(198,197,191,4)$ and $T_{\mathrm b}=(199,198,193,4)$ are
not reflections of one another --- $h_r-l_1$ is $6$ for the first and $5$ for the second --- and
accordingly the component of degree $D_1$ does not vanish.
\end{remark}

\begin{figure}[!ht]
\centering
\includegraphics[width=\textwidth]{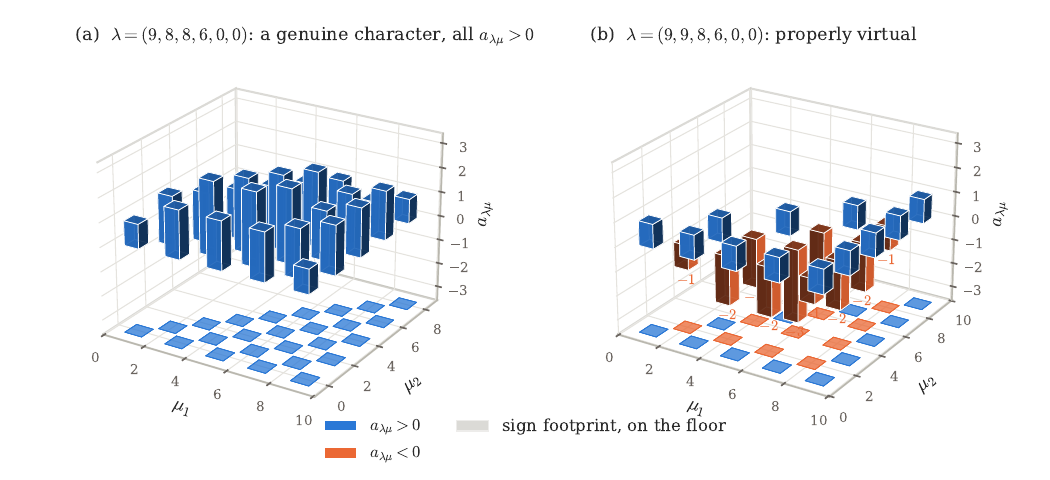}
\caption{Why one reciprocal pair is not typical, at $r=2$: the expansion
$\Psi_r(\lambda)=\sum_\mu a_{\lambda\mu}\,sp_\mu(z)$ over the dominant weights of $Sp(4)$. The two
shapes differ in a single box. For $\lambda=(9,8,8,6,0,0)$ every $a_{\lambda\mu}$ is positive: the
object is a genuine character, as it is for \emph{every} $\lambda$ when $r=1$, where Theorem
\ref{thm:main} writes it as a product of three $\mathfrak{sl}_2$ characters. For
$\lambda=(9,9,8,6,0,0)$ the expansion has $12$ positive and $10$ negative coefficients, the most
negative being $-3$: it is properly virtual, which is the phenomenon Remark \ref{rem:rankone}
records and the reason an endpoint computation stops being an exact test past one pair. The two
panels share the vertical scale, and the flat squares on the floor repeat the sign of each
coefficient, because a coefficient of $-1$ is invisible beside one of $+3$. Computed by
\texttt{anc/fig\_v2.py}.}
\label{fig:virtual}
\end{figure}

\subsubsection*{The criterion at \texorpdfstring{$t=2$}{t=2}}

\begin{proof}[Proof of Theorem \ref{conj:crit}]
The implication from {\rm(a)} or {\rm(b)} to $\Psi_r(\lambda)=0$ is Theorem \ref{thm:zeros}. For the
converse, assume $\Psi_r(\lambda)=\Phi_{2,r}(\lambda;z)\equiv0$ and that no residue class modulo $2$ is
absent from $\beta(\lambda)$, since otherwise branch {\rm(a)} holds and there is nothing to prove.
Then the numerator vanishes identically, so in particular its component of degree $D_1$ vanishes and
$|G|=2$ by Corollary \ref{cor:uniquegen}.

By Theorem \ref{thm:Ggen}(b) there is a tie at rank $r$ between classes $i$ and $i+t/2=i+1$; at $t=2$
these are the two classes, and both lie in $\mathcal{E}$, so $e=2$. (With $e=1$ all increments lie in a
single class and are pairwise distinct by Step 4, so no tie is possible.) Since $e=2$, neither class
is a singleton and $\mathcal{S}=\{\beta_1,\dots,\beta_N\}$ is the whole beta set.

By Corollary \ref{cor:Ctaugen}, $C=\min\mathcal{S}+\max\mathcal{S}=\beta_N+\beta_1=\tau$, and since
both classes change, $g_{\mathrm{com}}=\varnothing$ in \eqref{eq:gcomgen}, so \eqref{eq:Ssymgen} reads
$\tau-\mathcal{S}=\mathcal{S}$: the beta set is symmetric about $C$, that is
\[
  \beta_j+\beta_{N+1-j}=C\qquad\text{for every }j .
\]
Substituting $\beta_j=\lambda_j+N-j$ turns this into $\lambda_j+\lambda_{N+1-j}=C-N+1=:w$ for every
$j$, which is self-complementarity of width $w=\lambda_1+\lambda_N$. Finally, Corollary
\ref{cor:Ctaugen} says that the two solutions of $2\kappa\equiv C\pmod2$ are the two classes, which
holds exactly when $C$ is even; and $C=w+N-1$ with $N=2r+2$ even, so $C$ even is exactly $w$ odd.
Thus $\lambda$ is self-complementary of odd width: branch {\rm(b)}.
\end{proof}

\begin{remark}
At $t=2$ the congruence \eqref{eq:congrgen} is vacuous --- every increment is even --- and
$|G|\le2$ holds for the weaker reason that only two classes exist. What is \emph{not} vacuous at
$t=2$ is that $\tau$ is even, and it is used twice: it forbids the crossed matching in Proposition
\ref{prop:centregen}, and it \emph{is} the parity that becomes branch {\rm(b)}.
\end{remark}

\begin{remark}[\emph{genuine} does not mean \emph{irreducible}]
In Remark \ref{rem:rankone}, \emph{$\pm$ a genuine character} means a \emph{nonnegative} combination
of irreducibles, not a single one. It is not a single one: already
$s_{(2)}(1,-1,z,z^{-1})=z^{2}+2+z^{-2}=sp_{(2)}+sp_{(0)}$. At $t=2$, $r=1$ Theorem \ref{thm:main}
gives a product of three $\mathfrak{sl}_2$ characters, and a product of characters is a character ---
that is the whole claim. From $r=2$ negative coefficients appear and no representation exists, which
is what \emph{properly virtual} means; Table \ref{tab:virtual} counts both and Figure
\ref{fig:virtual} draws one of each.
\end{remark}

\begin{table}[!ht]
\centering\small
\caption{Genuine against properly virtual, over all $\lambda$ with $\lambda_1$ up to the bound shown.
The expansion is $\Psi_r(\lambda)=\sum_\mu a_{\lambda\mu}\,sp_\mu(z)$; \emph{one sign} means every
$a_{\lambda\mu}$ has the same sign, so that $\pm\Psi_r$ is a character, and \emph{mixed} means it is
properly virtual. At $r=1$ the mixed column is empty, as Theorem \ref{thm:main} forces; it is not
empty for any $r\ge2$ we can reach. Computed by \texttt{anc/folding\_t2.py}, whose archived run
prints each row as \emph{cero} / \emph{no negativa} / \emph{no positiva} / \emph{MEZCLADA}; the
\emph{one sign} column here is the sum of the two single-sign classes, which that run now also
prints. The figure alongside is \texttt{anc/fig\_v2.py}, which selects its two examples from the
same population but does not produce this table.}
\label{tab:virtual}
\begin{tabular}{@{}lrrrr@{}}
\toprule
 & \multicolumn{1}{c}{shapes} & \multicolumn{1}{c}{$\Psi_r=0$} & \multicolumn{1}{c}{one sign}
 & \multicolumn{1}{c}{mixed} \\
\midrule
\rowcolor{hband} $r=1$, $\lambda_1\le14$ & $3060$ & $352$ & $2708$ & \textcolor{hproved}{$\mathbf{0}$} \\
$r=2$, $\lambda_1\le8$  & $3003$ & $78$ & $2531$ & \textcolor{hverif}{$\mathbf{394}$} \\
$r=3$, $\lambda_1\le5$  & $1287$ & $27$ & $1008$ & \textcolor{hverif}{$\mathbf{252}$} \\
$r=4$, $\lambda_1\le2$  & $66$   & $2$  & $58$   & \textcolor{hverif}{$\mathbf{6}$} \\
\bottomrule
\end{tabular}
\end{table}

\begin{remark}[where the alphabet sits]
For $r=1$ the alphabet $(1,-1,z,\bar z)$ of \eqref{eq:psir} is literally the eigenvalue list of a
matrix in the non-identity component $O(4)^{-}$, tabulated as such by Hanany and Kalveks
\cite[Table 7]{HK16}, who also record the rank drop
$[\mathrm{vec}]_{O(2n)^{-}}\cong[\mathrm{vec}]_{O(2)^{-}}\oplus[\mathrm{fund}]_{C_{n-1}}$ for the
\emph{vector} representation. So $\Psi_r$ is a $GL(2r+2)$ character evaluated on that component, and
the folding $D_{r+1}\to C_r$ is the reason a symplectic group appears at all; we make no use of this
beyond the placement, since for a general $\lambda$ the restriction is a $\ZZ$-combination of
symplectic characters and not a single one.
\end{remark}

\section{Verification}\label{sec:verif}

Every statement above was checked numerically, by three implementations along different code paths:
one in exact Laurent arithmetic, one a from-scratch bialternant evaluation written from the
statements alone, and one in Sage, which shares no line of code and no library with the other two.
The ranges and counts are collected here once.

\begingroup\footnotesize\setlength{\tabcolsep}{4.0pt}
\rowcolors{2}{hband}{white}
\begin{longtable}{@{}>{\raggedright\arraybackslash}p{0.360\linewidth}%
                  >{\raggedright\arraybackslash}p{0.228\linewidth}%
                  >{\raggedleft\arraybackslash}p{0.278\linewidth}@{\hspace{5pt}}l@{}}
\toprule
statement & range & result & status\\
\midrule
\endfirsthead
\multicolumn{4}{@{}l}{\footnotesize\itshape Verification, continued}\\[1pt]
\toprule
statement & range & result & status\\
\midrule
\endhead
\midrule
\multicolumn{4}{r@{}}{\footnotesize\itshape continues on the next page}\\
\endfoot
\bottomrule
\endlastfoot
Thm \ref{thm:main}, sign included & $t\le7$, $|\lambda|\le14,\dots,10$ &
$959$ exact, $476$ zeros, $0$ fail & \stproved\\
\quad the same, an implementation written from the statement & $t\le6$, $|\lambda|\le12,\dots,9$ &
$749$ shapes, $0$ fail; sign decoy fails $35/133$ & \stproved\\
Lemma \ref{lem:L1} & $t\le6$ & $454$ nonzero minors & \stproved\\
Lemma \ref{lem:L4} & $t\le7$ & $724/724$ & \stproved\\
Proposition \ref{prop:quotient} & $t\le7$, $|\lambda|\le19$ & $2970/2970$ & \stproved\\
Prop. \ref{rem:lattice} (concentric locus) & $t\le9$, $|\lambda|\le16$ & $1331/1331$;
$0$ at odd $t$, $43$ at even & \stproved\\
Prop. \ref{rem:sign} (short form of the sign) & $t\le8$, $|\lambda|\le14$ & $1496/1496$;
proof steps $0$ fail & \stproved\\
\quad the shift law \eqref{eq:shift} and its sign & $t\le7$, $|\lambda|\le12$ &
$826$ shapes, $0$ fail & \stproved\\
Prop. \ref{prop:minimal} (the multiset separates) & $d_i\le60$; $t\le5$, $|\lambda|\le16$ &
$75640$ invariants, $0$ collisions & \stproved\\
\quad the same, both directions, from the statement & $t\le5$, $|\lambda|\le14,\dots,11$ &
$676$ shapes, $0$ collisions; decoy $104$ & \stproved\\
Prop. \ref{prop:fibrelattice} (the lattice) & $t\le5$, $|\lambda|\le14,12,10,10$ &
$84$, $35$, $20$, $14$ fibres, both ways & \stproved\\
Cor. \ref{cor:fibregf} (generating function) & same range &
$153/153$; wrong denominator $101$ & \stproved\\
Remark \ref{rem:signsees} (the sign sees the $m_i$) & $t\le4$ &
$107/107$, $43/62$, $26/36$ & \stverif\\
Lemma \ref{lem:single} & $t\le7$, $|\lambda|\le16$ & $2113/2113$ & \stproved\\
Thm \ref{thm:extra} & $t\le12$, $\lambda_1\le33$ & $119+21$, no others & \stproved\\
\quad the sign is $+1$, \eqref{eq:extrainv} & $t\le14$ even; cores $t\le10$ &
$28$ extra $+$ $219$ cores, all $+1$ & \stproved\\
Proposition \ref{prop:rect} & $t=2$, $r\le4$, $c\le60$ & $240/240$ & \stproved\\
Cor. \ref{cor:enum} vs direct enumeration & $t\le5$ & $14/14$; boxes $16/16$ &
\stproved\\
Prop. \ref{rem:involution} (runs alternate) & $t=2$, $|\lambda|\le14$, $\le4$ rows &
$4823$ minors; $217$ pairs, $0$ fail & \stproved\\
(D3), (D4) of \S\ref{sec:sharp} & $t\le6$, $|\lambda|\le10$ &
$600$ / $383$ / $200$ of $600$ & \stverif\\
\midrule
Thm \ref{conj:crit}, the criterion, every $r$ & --- & proof, \S\ref{sec:rigidity} & \stproved\\
\quad the same, read on $O(N)$, Cor.~\ref{cor:dettwist} & --- &
restatement, \S\ref{sec:conv} & \stproved\\
\quad Thm \ref{thm:PvW}, the external input & --- & \cite{PvW} & \stext\\
\quad Theorem \ref{thm:Ggen}, $|G|\le2$ & $t\le10$, $r\le3$ & greedy vs.\ brute force, $0/4049$ & \stproved\\
\quad Prop.~\ref{prop:A1gen}, the extremes & $t\le10$, $r\le3$ & $2522$ shapes, $0$ fail & \stproved\\
\quad \quad the hypothesis is needed & same & $414$ shapes fail without it & \stproved\\
\quad the displayed identities, \S\S\ref{sec:extremal}--\ref{sec:rigidity} & $t=2$, $r\le3$ & $11$ identities, $0$ fail & \stverif\\
Lemma \ref{lem:step}, the column move, any class & $t\le6$, $r\le3$ &
$285\,600/285\,600$; one-count decoy $149\,652$ & \stproved\\
Thm \ref{thm:suff}, the sufficient direction & $t\le8$, $r\le4$ &
$438$ shapes, $0$ fail; \eqref{eq:wrefl} $438/438$ & \stproved\\
Conj. \ref{conj:general}, both directions & $t\le12$, $r\le4$ &
$190\,443$ shapes, $24$ configurations, $0$ fail & \stverif\\
\quad the same, independent evaluation & $t\le8$, $r\le4$ &
$43\,010$ shapes, $409$ zeros, $0$ fail & \stverif\\
\quad the same again, in Sage & $t\le8$, $r\le3$ &
$12\,937$ shapes, $131$ zeros, $0$ fail & \stverif\\
\quad against Theorems \ref{thm:main} and \ref{conj:crit} & $r=1$; $t=2$ &
$9913$ and $29\,508$ shapes, $0$ disagreements & \stproved\\
\quad the decoy reading $\beta$ instead of $\mathcal{S}$ & same &
fails in $24$ of $24$ and $14$ of $14$ & \stproved\\
\midrule
Thm \ref{conj:crit}, both directions, direct check & $r\le3$, $|\lambda|\le28,26,22$ &
$9961$ shapes, $318$ zeros, $0$ fail & \stproved\\
\quad branch (a), Thm.~\ref{thm:zeros} & $r\le3$, $|\lambda|\le15$ &
$16/16$ & \stproved\\
\quad branch (b), Thm.~\ref{thm:zeros} & $r\le3$ & $24/24$ & \stproved\\
\quad even-width control for (b) & $r\le3$ & $56/56$ nonzero & \stproved\\
\quad converse at $r=1$, Thm.~\ref{thm:r1} & $\beta_1\le25$ &
$14950$ beta sets, $0$ fail & \stproved\\
\quad converse, $\ell(\lambda)\le N/2$, Thm.~\ref{thm:stable} & $r\le3$ &
$372/372$ witnesses & \stproved\\
\quad the isolating witness, Prop.~\ref{conj:iso} & $r=2$, $|\lambda|\le14$ &
$227$ witnesses, $11$ residue & \stproved\\
\quad the same at $r=3$ & $r=3$, $|\lambda|\le13$ &
$149$ witnesses, $0$ residue & \stproved\\
\midrule
Identity \eqref{eq:compl} & $r\le2$ & $474/474$ & \stproved\\
Identity \eqref{eq:docagne}, and $\deg_{c_\ast}S_n=n-1$ & $b\le31$ & $496/496$; $39/39$ &
\stproved\\
Lemma \ref{lem:AtoSp}, the \S\ref{sec:unstable} reduction & $r\le3$, $|\nu|\le8$ &
identity $0$ fail; folds $0$ fail & \stproved\\
\quad the reduction carried out & $r\le2$ & $236$ labels, $0$ unresolved & \stproved\\
\eqref{eq:Cmu} against the exact object & $r\le2$, both ranges & $184$ shapes, $0$ fail &
\stverif\\
\eqref{eq:Jlambda}, the term set of \S\ref{sec:lift} & $t=2,3$; $\lambda_1\le7$ &
$329/329$ and $791/791$ & \stproved\\
Prop. \ref{prop:wrongset}: $\|\Phi\|_1$ against the Laplace count & same &
$\lambda=(7,4)$ gives $30$ exactly; over $8$ in $128$ and $216$ & \stproved\\
Lemma \ref{lem:liftid}, the lift & same &
$329/329$, $791/791$; two decoys break $289$ and $623$ & \stproved\\
Prop. \ref{prop:liftcount}, equality when $\Phi\ne0$ & same &
proof, \S\ref{sec:lift}; control $309/329$ and $791/791$ & \stproved\\
Cor. \ref{cor:local}, the local form & same &
proof, \S\ref{sec:lift}; control $1696/1696$ and $4074/4074$ & \stproved\\
Rem. \ref{rem:skeleton}, sign and torsion on the forced points & same &
$398/398$ and $1176/1176$; three ansatzes fail & \stverif\\
Prob. \ref{prob:instrument}: the locus inside Littlewood's range &
$(t,r)=(2,2),(2,3),(4,2)$ & $1$ of $164$, $494$, $494$, all three
Thm \ref{thm:stable}(b) & \stverif\\
\end{longtable}
\endgroup

\noindent The two groups of rows for \S\ref{sec:zeros} are the only place in the paper where a
row's status is not obvious from its statement. The first is the machinery and what it reaches:
Theorem \ref{conj:crit} closes $t=2$ for every $r$, and the one row in \stext\ is the external
input it rests on; below it the extremal steps, then Theorem \ref{thm:suff}, which reaches every
$t$ and every $r$ and needs no external input at all; and last Conjecture \ref{conj:general}, in
\stverif\ because it is the one statement here that is measured and not proved. Its two sub-rows
are what makes the measurement worth quoting: the criterion is compared against the two theorems
that already decide their own ranges, and a decoy that reads the condition on $\beta$ instead of on
$\mathcal{S}$ is run alongside and fails everywhere. The second group is the same locus checked
directly against the definitions, branch by branch, which is a different code path and not a
restatement. The rows below both are the identities and the reduction the section rests on. In the
column, \stproved\ means a proof is given here or cited, and \stverif\ means the statement is
checked exhaustively over the stated range and not proved.

\noindent The (D3), (D4) row is the one that matters most in the first group: the identity holds in
all $600$ cases for \eqref{eq:object} and fails in $217$ and $400$ of them for the two deformed
alphabets, so the test is capable of failing. In the second group the even-width control plays the
same role for Theorem \ref{thm:zeros}: $56$ self-complementary shapes of even width, none of which
vanishes, so the parity hypothesis is doing work rather than decorating. A separate audit rebuilds
all of the second group from the definitions in a single pass, independently of the scripts that
produced them: $3414$ checks, $0$ failures. A second one takes the \emph{displayed formulas} rather
than the counts --- Theorem \ref{thm:LR}, the interval reading \eqref{eq:intervals}, Lemmas
\ref{lem:L2} and \ref{lem:L5}, the splitting \eqref{eq:threestep}, the complementation
\eqref{eq:compl}, d'Ocagne \eqref{eq:docagne} and the scalar caveat \eqref{eq:scalar} --- and
evaluates both sides at numeric points from the definitions, using neither the closed forms of this
paper nor the scripts behind the table: $3297$ evaluations over nine displayed formulas, $0$ failures. The scripts are ancillary to this paper, and so is the
saved output of each one, so every count in the table above can be located in an archived run rather
than taken on trust.

\section{Attribution}\label{sec:attr}

The paper cedes ground in a dozen separate places --- inside a proof, after a lemma, in a remark ---
and a reader who wants to know what is borrowed should not have to find those places one by one.
This section is that ledger, in two lists rather than one, because the two are read for different
reasons. Nothing here is new; every row points at a statement above, and where the body already says
whose an argument is, this table repeats it rather than replacing it.

\subsection*{What we use}
\begingroup\rowcolors{2}{hband}{white}\footnotesize\setlength{\tabcolsep}{4.0pt}
\begin{longtable}{@{}>{\raggedright\arraybackslash}p{0.60\linewidth}%
                   >{\raggedright\arraybackslash}p{0.34\linewidth}@{}}
\toprule
statement & due to\\
\midrule
\endfirsthead
\multicolumn{2}{@{}l}{\footnotesize\itshape What we use, continued}\\[1pt]
\toprule
statement & due to\\
\midrule
\endhead
$t$-quotient factorisation at a root-of-unity orbit, and the $\beta$-set bookkeeping the whole line
rests on & Littlewood \cite{Lit50}; Ayyer--Kumari \cite{AK22} \stext\\
factorisation at a purely reciprocal alphabet, for rectangular shapes & Ciucu--Krattenthaler
\cite{CK09} \stext\\
the same for self-complementary shapes, with the fixed point $+1$, and the parity obstruction in an
$a\times b$ rectangle & Ayyer--Behrend \cite{AB19} \stext\\
$o_\nu(A)=sp_\nu(z)$ --- \emph{the whole of} Lemma \ref{lem:AtoSp} & Ayyer--Kumari
\cite{AK25}: the composition of their (2.13) with their Lemma 3.2. We found it by computing and
gated it afterwards \stext\\
universal characters, and the shape of the type-$D$ specialisation & Koike--Terada \cite{KT87}
\stext\\
the symplectic modification rule used in \S\ref{sec:unstable} & King \cite{King71} \stext\\
twining: the character of a coset of a disconnected group is a character of the folded algebra &
Jantzen; Kumar--Lusztig--Prasad \cite{Jantzen,KLP} \stext\\
Littlewood's restriction rule, and the associate involution $[\nu^*]=[\nu]\otimes\det$ &
Littlewood \cite{Lit50}; classical \stext\\
the extension of the restriction rule covering all parameters --- so the relation space among the
$s_\nu$ at a reciprocal alphabet is \emph{not} undescribed & King \cite{King71};
Koike--Terada \cite{KT87}; Enright--Willenbring \cite{EW04} \stext\\
$S_bS_{b'-1}-S_{b'}S_{b-1}=-S_{b-b'}$ & d'Ocagne \cite{Doc}, for Chebyshev of the second kind;
classical \stext\\
that the initial form of a determinant is the signed sum over the optima of an assignment problem,
and that the matrix is singular exactly when two optima tie & tropical linear algebra, standard
\cite{TropCramer} \stext\\
the numerator lift of Definition \ref{def:lift}, the reason for preferring it to an
involution on $J_\lambda$, the correction of the fixed-point count from an equality to an
inequality, and the observation that a transversal carries two monomials & AI-assisted
exploration; see \S\ref{sec:disclosure} \stext\\
character values at an element of order two & Karmakar \cite{Kar24} \stext\\
that a sign-reversing involution reduces a signed sum to its fixed points, hence Lemma
\ref{lem:fixbound} & standard \stext\\
the last link of $\Phi_t=0\Rightarrow$ (ii), and its $n$-factor form & Purbhoo--van Willigenburg;
Rajan \cite{PvW,Rajan} \stext\\
\bottomrule
\end{longtable}
\endgroup

\subsection*{What is ours}

One thing has to be said before the table rather than inside it. The alphabet is what this paper
contributes before any single statement does: the root-of-unity extensions surveyed above add more root-of-unity
structure, and the reciprocal-pair line needs self-dual shapes; adjoining a \emph{free} reciprocal
pair to a full orbit is the step that, to our knowledge, does not appear in the factorisation
literature we survey here, and everything below is what that step makes answerable.

\begingroup\rowcolors{2}{hband}{white}\footnotesize\setlength{\tabcolsep}{4.0pt}
\begin{longtable}{@{}>{\raggedright\arraybackslash}p{0.44\linewidth}%
                   >{\raggedright\arraybackslash}p{0.50\linewidth}@{}}
\toprule
statement & status\\
\midrule
\endfirsthead
\multicolumn{2}{@{}l}{\footnotesize\itshape What is ours, continued}\\[1pt]
\toprule
statement & status\\
\midrule
\endhead
Theorem \ref{thm:main} and its proof, Lemmas \ref{lem:L1}--\ref{lem:L5} & \stproved; the Laplace
expansion along the frozen rows is elementary, and what is ours is the excess-two residue count, the
collapse and the sign\\
Lemma \ref{lem:L2} & \stproved; the determinant evaluated there is of a kind Ciucu and Krattenthaler
treat, and the evaluation is not the content --- the frozen block it is applied to is\\
Corollary \ref{cor:zero} & \stproved; it answers the question in the form Prasad posed it, which is
why it is stated separately\\
Propositions \ref{rem:lattice}, \ref{prop:fibrelattice} and Corollary \ref{cor:fibregf} --- the
two-class stratum is a lattice, with an explicit generating function & \stproved; ours\\
Proposition \ref{prop:minimal}, Proposition \ref{rem:sign} --- the multiset and the sign are exactly
what is seen & \stproved; ours\\
Proposition \ref{rem:involution}: the runs alternate & \stproved; \emph{part (i) is Macdonald's \cite[p.~91]{Mac95}, not
ours}, and the paper says so where the proposition stands\\
Theorem \ref{thm:extra}, Lemma \ref{lem:single} --- independence on the reciprocal locus &
\stproved; it extends a criterion of Ayyer--Kumari, and bounds where their theorem applies\\
Theorem \ref{thm:zeros}: the two branches & \stproved; branch (a) is one line --- constant-parity
$\beta$ makes the bialternant factor through $A^2$ --- and \emph{the self-complementary half is a
short corollary} of the standard complementation identity over the index family Ayyer and Behrend
single out. What this row does \emph{not} carry is the converse; that nothing else vanishes is proved in the
rows below, and finally for every $r$ in Theorem \ref{conj:crit}\\
Theorem \ref{thm:r1}, Theorem \ref{thm:stable} --- the converse at $r=1$, and inside Littlewood's
range for every $r$ & \stproved; ours, and we have not found this half stated in the sources surveyed\\
Corollary \ref{cor:dettwist}: determinant-twist rigidity, a biconditional for every $r$; and
that from $r\ge2$ the order-two endpoint stops deciding the curve &
\stproved; ours. That the endpoint stops deciding the curve is the finding, and
$\lambda=(5,4,3,1)$ is the first witness\\
Lemma \ref{lem:AtoSp} & \stext; \emph{not ours} --- see the first table; it is stated here because the
argument needs it, and its proof is theirs\\
Lemma \ref{lem:Arefl}: the reflected minor & \stproved; the reflection used is the one Ayyer and
Behrend already work with\\
Lemma \ref{lem:L3gen}, Lemma \ref{lem:step}, Theorem \ref{thm:Ggen} --- the Laplace decomposition and
the extremal set for general $r$ & \stproved; ours. The bound $|G|\le2$ is the tropical statement
above specialised to our matrix, and what we claim is only the \emph{reason} for it --- that a tie
forces $t$ even and puts the two tied increments in classes $i$ and $i+t/2$\\
that the top-degree factorisation is $[A(T)]_{\mathrm{top}}=a_H(z)a_L(1/z)$ & \stproved, and
\emph{elementary}: one line of Laplace plus a degree count. We claim nothing for it\\
that consecutive pairing is the unique minimum-weight matching on a line & \stext; folklore; a one-line
exchange argument, and we found no citable lemma\\
Theorem \ref{thm:suff}, Corollary \ref{cor:necgen} --- the sufficient direction for all $t$ and $r$,
and a necessary condition at even $t$ & \stproved; ours\\
Lemma \ref{lem:signformula}, Lemma \ref{lem:signrefl} --- the partial-reflection signs &
\stproved; ours\\
Conjectures \ref{conj:general}, \ref{conj:gcom}, \ref{conj:rank2} & \stverif; open, and measured; the ranges
are in \S\ref{sec:verif}\\
Proposition \ref{conj:rank}: the rank is the wrong invariant & \stproved; ours, and it is a negative
--- it removes a candidate invariant rather than supplying one\\
Proposition \ref{prop:wrongset}, Lemma \ref{lem:liftid}, Proposition \ref{prop:liftcount}
and Corollary \ref{cor:local} --- where such an involution cannot live, and where it can
& \stproved; \emph{the construction they are stated on is not ours} ---
see the first table --- and ours is the measurement\\
Remark \ref{rem:skeleton} --- what the size-one weight spaces force & \stverif; ours\\
the measurement inside Problem \ref{prob:instrument}: the vanishing locus meets
Littlewood's range in one shape, and that shape is a theorem already & \stverif;
ours, and it is a negative --- it says which kind of instrument the question needs,
not which one answers it\\
the remaining statements --- Corollaries \ref{cor:chars}, \ref{cor:enum}, \ref{cor:sizes},
\ref{cor:uniquegen}, \ref{cor:midgen}, \ref{cor:oddgen}, \ref{cor:reflectgen}, \ref{cor:Ctaugen},
\ref{cor:whenminus}, \ref{cor:suffgeom}; Propositions \ref{prop:quotient}, \ref{prop:rect},
\ref{conj:iso}, \ref{prop:orbitgen}, \ref{prop:translgen}, \ref{prop:centregen}, \ref{prop:A1gen};
Lemmas \ref{lem:L3}, \ref{lem:L4}, \ref{lem:nonstd}, \ref{lem:reflgen}, \ref{lem:dictgen},
\ref{lem:reflact}; and Theorem \ref{conj:crit} & \stproved; ours. They are the internal machinery of
the proofs above, and are listed by name so that this table covers every statement in the paper
rather than the memorable ones. The classical facts they invoke --- Littlewood's restriction rule and
its range, the modification rules --- are the ones ceded in the first table\\
\bottomrule
\end{longtable}
\endgroup

\noindent
Three negatives are worth stating as negatives, because each cost a search. Kumari's
$(z_1,z_2,k)$-asymmetry does \emph{not} classify our vanishing locus. We know of no rank-$r$ analogue
of d'Ocagne, which would be a closed form for $s_\mu$ at a reciprocal alphabet for every $\mu$, and
the literature supplies one only for special shapes. And a vanishing criterion on
$O(N)\setminus SO(N)$ --- which is what \S\ref{sec:zeros} is about --- we could not find under either
of two searches with different vocabularies.

\section{Open problems}\label{sec:open}

Figure \ref{fig:map} places what follows against what the paper proves and what it borrows;
the problems below are the boxes it leaves open.

\begin{figure}[!ht]
\centering
\includegraphics[width=\textwidth]{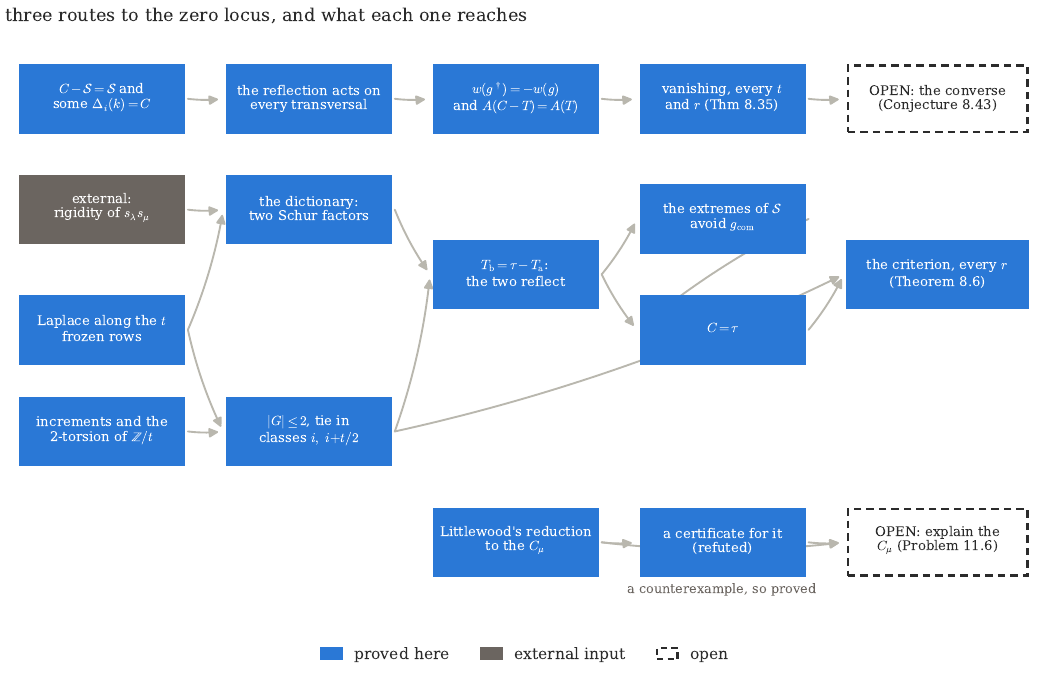}
\caption{Three routes to the zero locus, and what each one reaches. Boxes are statements and
arrows are dependencies; the colour is the status column of \S\ref{sec:verif} --- filled for what
is proved here, grey for the single external input, dashed for what is open. The \emph{middle}
path is the extremal argument of \S\S\ref{sec:extremal} and \ref{sec:rigidity}, which closes the
criterion at $t=2$ for every $r$. The \emph{lower} one is Littlewood's reduction, whose
certificate is refuted by Proposition \ref{conj:iso} --- a counterexample, hence itself a proof
--- so that it no longer bears on the criterion but on the separate question of explaining the
$C_\mu$, which is Problem \ref{prob:instrument}. The \emph{upper} one starts from a hypothesis
rather than from the expansion, and is the only one that reaches every $t$ and every $r$: there
the reflection pairs off all the terms at once, not just the two maximisers. The three meet at the
one implication that is missing, which is what makes Conjecture \ref{conj:general} a conjecture.
The boxes name their statements, and the four numbers they do carry are read from the
\texttt{.aux} when the figure is drawn, so that they cannot drift when the numbering does;
\texttt{check\_figrefs.py} checks that each one still resolves, and to a statement of the right
kind. Computed by \texttt{anc/fig\_proof.py}.}
\label{fig:map}
\end{figure}

\begin{problem}[the other classical types]\label{prob:types}
Everything above is type $A$. What do $sp_\lambda(\zt,z,\zb)$, $so_\lambda(\zt,z,\zb)$ and
$o_\lambda(\zt,z,\zb)$ do? The root-of-unity factorizations were carried from type $A$ to $B$, $C$,
$D$ by a uniform argument in \cite{AK22} and to the universal characters in \cite{Alb23}; the frozen
rows of the bialternant are the same in every type, so Lemma \ref{lem:L1} and the excess-two count
should survive verbatim, while the collapse, Lemmas \ref{lem:L4} and \ref{lem:L5}, would need
re-deriving once per type. The alphabet \eqref{eq:object} is already a torus element of $O(t+2)$, so
the other types are not the same function at a new point but new objects, and choosing the right ones
is a matter of taste inside a programme we did not build.

\emph{The classical symplectic analogue is not open, and we should have said so here rather than
only at Theorem 2.2.} Remark \ref{rem:notacase} above compares our alphabet with \cite[Theorem 2.2]{Kum24}, which is
the $GL$ statement; but the same paper's \cite[Theorem 2.8]{Kum24} evaluates
$sp_\lambda(X,\omega X,\dots,\omega^{t-1}X,\,y)$, and in a symplectic character adjoining $y$ is
adjoining the reciprocal pair $\{y,y^{-1}\}$. That answers the classical symplectic analogue of this problem,
with a vanishing criterion --- her $(z_1,z_2,k)$-asymmetry of the $t$-core --- and we have checked
that criterion against the vanishing of the universal symplectic character on
$\mu_t\cup\{z,z^{-1}\}$: it agrees on $285$ of $285$ shapes at $t=3,5,7$, in all four cells
\stverif. The distinction that survives is not of range but of object: her $sp_\lambda(W)$ is the
classical character, whose eigenvalue alphabet is $W$ together with its inverses, so at $X=(1)$ the
orbit is \emph{doubled}, and the universal character over the orbit taken once is a different
function --- the values differ on $155$, $96$ and $51$ of $219$ shapes, while the zero locus
coincides on $219$ of $219$ \stverif. So what this problem still asks is the orthogonal cases, and
the doubled-versus-undoubled gap in the symplectic one. \emph{The rest of this problem is a
measurement, and reported as one only.} Expanding the Koike--Terada universal characters in the Schur basis and evaluating
term by term with Theorem \ref{thm:main}, over $|\lambda|\le12$ for $t\le3$ and $|\lambda|\le10$ for $t=4,5$, the fraction of
$\ell(\lambda)\ge2$ shapes on which the value vanishes is markedly larger for $o_\lambda$ on the
non-identity component than on the identity one --- $83.8\%$ at $t=2$ against $41.8\%$ at $t=3$, and
$58.8\%$ at $t=4$ against $42.1\%$ at $t=5$ --- with $sp_\lambda$ in between and the Schur case
lowest. The effect is real and it tracks the component, which is what one would expect of an
associate cancellation; but it is quantitative and we found no law covering the whole table. A first
attempt at one, that
$o_\lambda$ vanishes for all $\ell(\lambda)\ge2$ when $t$ is even, is false: it was read off a
sample and the systematic sweep refutes it.

The $t=2$ column of that comparison is the one case which does have a law, and it is Lemma
\ref{lem:AtoSp}: there the alphabet is \eqref{eq:psir} at $r=1$, so $o_\nu$ equals $sp_\nu(z)$, a
\emph{rank-one} symplectic character. That vanishes identically for $\ell(\nu)=2$ --- all $36$ such
shapes in range, with no exception --- and folds for $\ell(\nu)>2$, which is why the $t=2$ figure is
the outlier. What remains without a law is odd $t$, where the alphabet is not of the form
\eqref{eq:psir} and the lemma does not apply.

The $sp_\lambda$ column at $t=2$, on the other hand, has an answer, and it explains the asymmetry of
the table rather than adding a row to it. Adjoining the same two letters to a universal
\emph{symplectic} character does not collapse it. Over $|\nu|\le8$,
\begin{equation}\label{eq:spbranch}
sp_\nu(W,1,-1)\;=\;\sum_{\gamma\subseteq\nu}c_{\nu/\gamma}\;sp_\gamma(W),
\end{equation}
where the coefficient depends on $\nu$ and $\gamma$ only through the skew shape $\nu/\gamma$ ---
$247$ coefficients over $74$ distinct skew shapes, with no clash --- and takes only the values $0$
and $\pm1$, vanishing whenever $|\nu/\gamma|$ is odd. That is the shape a branching rule for
universal symplectic functions produces, and such rules are available \cite{JJLW24}. Read against it,
Lemma \ref{lem:AtoSp} says that for $o_\nu$ the corresponding coefficients are concentrated at
$\gamma=\nu$: the two letters collapse the orthogonal character to a single symplectic one, and leave
the symplectic character spread over the whole interval below $\nu$. The asymmetry of the measurement
is that, and not an artefact of the ranges. What we have not done is identify $c_{\nu/\gamma}$, which
\eqref{eq:spbranch} makes a question about skew universal symplectic characters at $(1,-1)$ and
nothing more.

The third column resists the same route, and for a structural reason rather than an accidental one.
What answers the other two is that adjoining a letter is the ring map $p_k\mapsto p_k+c^k$, so
$\iota_{1}$ and $\iota_{-1}$ act on $\Lambda$ and compose. The identity that would bring $so$ into
the picture, $so_\nu(X)=o_\nu(X,\bar X,1,0,\dots)$ of \cite[(2.13)]{AK25}, relates the two on a
\emph{doubled} alphabet instead, and a doubling is not an adjunction: the $\iota$'s do not reach it.
So the $so$ column is not a gap of effort but of mechanism, and we leave it stated as such.
\end{problem}

\begin{problem}[a Verschiebung that sees a reciprocal pair]\label{prob:versch}
The operator $\varphi_t$ has nothing to say about a letter outside a $\zeta$-orbit, which is the
structural reason Theorem \ref{thm:main} is not a corollary of anything in that line, and the reason
our proof descends to the alternant. Is there an operator on universal characters, or a
factorization of $\varphi_t$ through a larger algebra, acting on alphabets of the form (union of
$\zeta$-orbits) $\cup$ (reciprocal pairs)? By \S\ref{sec:sharp} such an operator would decide (D1)
and (D2) at once. The discussion there suggests a quantitative form.

What makes the question well posed rather than vague is that each half already has a formalism, and
they are not the same one. The $\zeta$-orbits are the province of $\varphi_t$, in Albion's
formulation \cite{Alb23}. Alphabets closed under inversion are the province of the universal
symplectic and orthogonal functions, which have vertex operator realizations from which their
branching rules follow \cite{JJLW24}. Our alphabet is one of each, and Lemma \ref{lem:AtoSp} is the
one place where the two meet: there the $\zeta$-orbit is $\{1,-1\}$, small enough that adjoining it
is a pair of letter maps rather than an operator.
\emph{The companion paper \cite{PaperII} answers this in the weak sense, and says so.} It exhibits a
map taking the object with all pairs free to the object with the orbit frozen --- one construction
for even $t$ and a different one for odd $t$, a single construction producing both being left open
there. Its two disclaimers are what leave the question above standing: the map is not a Verschiebung
and is not claimed to factor one, and it is assembled after the fact from a branching followed by an
evaluation rather than from an operation on symmetric functions. What is asked here is the latter.
\end{problem}

\begin{remark}[what of Problem \ref{prob:versch} is already settled, and a correction]\label{rem:composes}
An earlier version of this problem asserted that we knew of nothing acting on both halves at once.
That was wrong, and the counterexample is in the reference the problem itself names for the twisted
half: Albion \cite[\S6.1]{Alb23} defines the Verschiebung with one free letter adjoined,
$\varphi^q_t h_{at+b}:=q^b\sum_k q^{kt}h_{a-k}$, and his Proposition 6.2 evaluates it on $s_\lambda$.
The single-letter case was in print. What was missing is the \emph{pair}, and the two formalisms do
compose there:
\begin{equation}\label{eq:composes}
s_\lambda\bigl(\mu_t\cup\{z,z^{-1}\}\bigr)=\sum_{\mu}\operatorname{sgn}_t(\mu)\,
s_{\lambda/\mu}(z,z^{-1}),
\end{equation}
the sum over $\mu\subseteq\lambda$ with empty $t$-core and every component of the $t$-quotient of at
most one row. This is not a new theorem: it is Albion's Theorem 3.1 evaluated at $1$, where
$s_\nu(1)=1$ for $\ell(\nu)\le1$ and $0$ otherwise. It is recorded because it answers the last
sentence of the problem --- the two halves compose, and the composition is explicit --- and because
a question should not be left standing on a premise its own reference refutes. Checked against the
direct value, $714/714$ at $t=2$ and $2001/2001$ at $t=3$; the decoy without the empty-core condition
fails $709$ and $2001$ times, the decoy with $(z,z)$ in place of $(z,z^{-1})$ fails $664$ and $1496$,
and the decoy without $\operatorname{sgn}_t(\mu)$ fails $665$ and $1884$ \stverif. What
\eqref{eq:composes} does \emph{not} give is an operator: it is a branching identity, not a
factorisation of $\varphi_t$ through a larger algebra, and the problem above asks for the latter.
\end{remark}

\begin{proposition}[the rank is the wrong invariant]\label{conj:rank}
Let $A=\zt\cup B$ with $B$ closed under inversion and disjoint from $\zt$, and let $\rho(B)$ be the
rank of the subgroup of $\CC^{\times}$ generated by $B$. Then $\rho(B)=1$ does not imply that
$s_\lambda(A)$ is a product of characters. At $t=2$ and $B=\{z,z^{-1},z^{2},z^{-2}\}$, which has
$\rho(B)=1$, the value at $\lambda=(2)$ has a root off the unit circle, and so has no such
factorization; the same happens for $B=\{z,z^{-1},z^{3},z^{-3}\}$.
\end{proposition}

\begin{proof}
A product of characters $\chi_k(z)=(z^{k+1}-z^{-k-1})/(z-z^{-1})$ has all of its roots on the unit
circle, so it is enough to exhibit one root off it. For $B=\{z,z^{-1},z^{2},z^{-2}\}$ and
$\lambda=(2)$ the value is
$z^{-4}+z^{-3}+z^{-2}+z^{-1}+3+z+z^{2}+z^{3}+z^{4}$. The Laurent polynomial is invariant under
$z\mapsto z^{-1}$, so writing $x=z+z^{-1}$ and using $z^{k}+z^{-k}\in\ZZ[x]$ it becomes
\[
x^{4}+x^{3}-3x^{2}-2x+3\;=\;(x-1)\bigl(x^{3}+2x^{2}-x-3\bigr).
\]
The cubic has discriminant
$18\cdot2\cdot(-1)(-3)-4\cdot2^{3}(-3)+2^{2}(-1)^{2}-4(-1)^{3}-27\cdot(-3)^{2}=-31<0$, so two of its
roots are not real. If $|z|=1$ then $x=z+z^{-1}=2\cos\theta$ is real; hence for a non-real root
$x_0$, neither root of $z^{2}-x_0z+1=0$ lies on the unit circle, and both are roots of the value.
So the factorization is impossible, and the statement is exact rather than numerical.

Over $|\lambda|\le8$ the same phenomenon occurs for $17$ of the $62$ nonvanishing shapes, and for
$B=\{z,z^{-1},z^{3},z^{-3}\}$ for $18$; on the same range $B=\{z,z^{-1}\}$ has every root on the
circle without exception, as Theorem \ref{thm:main} requires.
\end{proof}

What the rank does not see is how many free directions $B$ has. A single reciprocal pair is one
direction; $\{z,z^{-1},z^{2},z^{-2}\}$ generates a group of rank one and is two, and the two see each
other. So the invariant to conjecture on is not $\rho(B)$ but the number of independent letters,
and with that reading the statement becomes the one \S\ref{sec:sharp} already makes by other means:

\begin{conjecture}[one free direction]\label{conj:rank2}
Let $B$ be closed under inversion and disjoint from $\zt$, as in Proposition \ref{conj:rank}. Then
$s_\lambda(\zt\cup B)$ is zero or a product of characters, for every $\lambda$, if and only if $B$ is
a single reciprocal pair. (The alternative is not decoration: by Corollary \ref{cor:zero} the value
does vanish for some $\lambda$ already at one pair, and $0$ is not a product of characters.)
\end{conjecture}

The evidence for it is Theorem \ref{thm:main} in one direction and (D1)--(D4) together with
Proposition \ref{conj:rank} in the other: enlarging $B$ destroys the product whether the enlargement
raises the rank or not. We have no proof and no proposed mechanism, and this is precisely the
statement a $\varphi_t$-style argument ought to produce.

That hypothesis is what the conjecture is stated under, and it is also what a general form would have
to drop. Inversion-closure of $B$ is sufficient in the evidence above, but it is not necessary for a
product. Write $w=-iy$. Then the second block of
the alphabet $\{1,-1,y,-y^{-1}\}$ is not closed under inversion, yet
\begin{equation}\label{eq:scalar}
s_\lambda\big(1,-1,y,-y^{-1}\big)\;=\;i^{\,|\lambda|}\,s_\lambda\big(-i,\,i,\,w,\,w^{-1}\big),
\end{equation}
and the alphabet on the right is. The hypothesis a projectively invariant form would carry is therefore not inversion-closure of $B$
but that $A$ be a scalar multiple of an inversion-closed set. That hypothesis is weaker than it
looks, and (D3) is what measures how much weaker. A coset $c\zt$ is itself inversion-closed exactly
when $c^{2}\in\zt$, since $(c\zt)^{-1}=c^{-1}\zt$; and (D3) satisfies it for every $t$, its
$\zeta'=e^{\pi i/t}$ having $(\zeta')^{2}=\zeta$. Its alphabet is the odd powers of a $2t$-th root of
unity, and inversion preserves the parity of the exponent. Yet (D3) breaks the identity, in $217$ of
the $600$ shapes swept above. So inversion-closure does not survive as the whole hypothesis even
projectively: what (D3) lacks is the other half of \S\ref{sec:sharp}, that the frozen letters be a
\emph{full} orbit and not a coset of one. A projective form of Conjecture \ref{conj:rank2} has to
carry both, and it is (D3) that shows they are independent.

\begin{problem}[the instrument for the second pair]\label{prob:instrument}
Section \ref{sec:zeros} reduces the vanishing at $r\ge2$ to a statement about the coefficients
$C_\mu$ of \eqref{eq:Cmu}, which are signed sums of Littlewood coefficients. That statement is no
longer what is missing: Theorem \ref{conj:crit} settles the vanishing, by the extremal argument of
\S\ref{sec:rigidity}, which never meets a $C_\mu$. What the reduction asks for now is not a proof
but an \emph{explanation} --- why those signed sums cancel, in Littlewood's own terms --- and there
is still no certificate, the one we proposed failing by Proposition \ref{conj:iso}. A problem that
was a gap is now a question about a second proof, which is a better problem and a smaller one. Lemma \ref{lem:AtoSp} settles the other half of the reduction --- what
the values $o_\nu(A)$ are --- and settles it symplectically: the two readings of the alphabet agree
because they are the same reading, the orbit Lie algebra of $D_{r+1}$ being $C_r$ rather than the
fixed subalgebra $B_r$. What it does not touch is the coefficients. Kumari and Stokke \cite{KS26}
have given Murnaghan--Nakayama rules for symplectic, orthogonal and orthosymplectic Schur functions;
by Lemma \ref{lem:AtoSp} the relevant one here is the symplectic \cite[Theorem 3.4]{KS26}, not the
even orthogonal \cite[Theorem 3.10]{KS26}. Does the recursion it provides compute $C_\mu$? Their rule
carries a third sum beyond the adding and removing of border strips, and it is the one they single out
as complicated; \cite[Corollary 3.5]{KS26} shows it drops out once $\mu_n+1\ge r$, so it is a
shallow-shape phenomenon. An earlier form of this problem asked whether that sum is what accounts for
the failure of Littlewood's rule outside its range; Lemma \ref{lem:AtoSp} answers that it is not,
since the failure is carried by the symplectic modification rules, which act on the labels and not
inside the recursion. The coefficients are what is left, and whether that recursion computes them is
what we have not attempted.

There is a second instrument, and it is the one we would try first. Frohmader's branching rules
\cite{Frohmader} compute $\mathrm{mult}(\pi^{\nu}_{O_n},\pi^{\lambda}_{GL_n})$ as
$\sum_{\mu}c^{\lambda}_{\mu\nu}(O_n)$ over partitions $\mu$ with even rows, with no stable-range
constraint and no modification rule. Jang--Kwon \cite{JK20} compute that same multiplicity, and
their form is the sharper one for what we need: it is a manifestly nonnegative count of flagged
Littlewood--Richardson tableaux. The obstruction here is precisely that our $C_\mu$ of
\eqref{eq:Cmu} carry signs, and what \S\ref{sec:unstable} leaves unfinished is an equality
$m_\nu(\lambda)=m_{\nu^{*}}(\lambda)$ between two multiplicities of exactly this kind, in the range
$\ell(\lambda)>N/2$ where Littlewood's rule stops applying. A subtraction-free expression for each
side is what a proof there would want to start from. Our $C_\mu$ of \eqref{eq:Cmu} are signed sums of exactly the
multiplicities those rules compute, assembled by the associate involution $\nu\mapsto\nu^{*}$. So the
question becomes concrete: is $C_\mu(\lambda)=0$, on the two branches, an identity between
generalized Littlewood--Richardson coefficients that the flag conditions make visible? That is a
question about tableaux, which is the form in which we would like to see it answered.
\emph{Since this problem was written, the companion paper \cite{PaperII} has supplied an instrument
of the kind it asks for.} There the coefficients are computed as symplectic branching multiplicities
against a filter given in closed form, both factors computable without expanding a monomial; and on
the odd side the branching matrix leaves the numerator altogether, which turns that half into a
signed transversal count. Two limits are that paper's own, and they are why this problem stays open:
it does not explain the cancellation and does not claim to, and what it delivers in place of an
explanation is a single sharp conjecture. We give the pointer and not the content.
One measurement bears on which instrument is the right one to reach for, and it is a negative. By
\eqref{eq:Cmu} the vanishing at $r\ge2$ is the equality $m_\nu(\lambda)=m_{\nu^{*}}(\lambda)$ for
every $\nu$, so one may ask how much of that locus lies inside the range where Littlewood's rule
computes those multiplicities at all. Almost none of it: over $(t,r)=(2,2)$, $(2,3)$ and $(4,2)$ ---
$164$, $494$ and $494$ shapes with $\ell(\lambda)\le N/2$ --- exactly one shape in each case
satisfies the equality, and in all three it is the odd-width rectangle with $N/2$ rows \stverif.
That shape is not a residue left over by the sweep: it is branch (b) of Theorem \ref{thm:stable},
which proves it. So inside Littlewood's range the question has no content beyond a theorem already
in hand, and the coefficients whose cancellation \eqref{eq:Cmu} asks about live where the rule does
not apply. That is why an instrument carrying no stable-range constraint is not a convenience here
but a requirement, and it is what makes the two rules above the ones to reach for.
\end{problem}

\begin{problem}[the extra locus, in general]\label{prob:extralocus}
Does Theorem \ref{thm:extra} have an analogue for $sp$, $so$ and $o$? More generally, for which
specializations of the free variables does the criterion \eqref{eq:AK53} acquire extra solutions, and
can the extra locus always be read off the kernel of the specialization acting on the span of the
universal characters, as Remark \ref{rem:why} suggests? If each type acquires a family indexed by an
order-two element of the residue lattice, the principle is real.

The question needs the free part pinned down, and for a computed reason: the extra family is a
phenomenon of the \emph{minimal} alphabet. Adjoin a second reciprocal pair and it is gone --- over
$t=3,4,5,6$ and $|\lambda|\le14$ the locus is then exactly the $t$-cores, for type $s$ no less than
for $sp$ and $o$, while at one pair type $s$ has the two extras of Theorem \ref{thm:extra} at $t=4$
and one at $t=6$. (We test the second pair at $z_2=z_1^{2}$, a specialization, which can only create
coincidences and not remove them, so a count of zero there is a count of zero.) An analogue for
another type must therefore be sought at \emph{that type's} minimal free part, one letter of each
kind; comparing types across free parts of different sizes measures the free part instead, which is
what our first attempt did.

The kernel is the right thing to look at, but its \emph{presence} does not settle the locus; its
size does. Two free parts of the same size can both collapse the independence and behave quite
differently. On the reciprocal locus $s_\mu(z,\zb)=\chi_{\mu_1-\mu_2}(z)$ retains one grading, and
the extra locus is the single family of Theorem \ref{thm:extra}. Replace the pair by the
$\zeta_2$-orbit $(z,-z)$, where $s_\mu(z,-z)=z^{|\mu|}s_\mu(1,-1)$ retains only $|\mu|$ and a sign,
and over $\ell(\lambda)\le2$ the criterion acquires $28$ further solutions at $t=2$ and $186$ at
$t=8$, in no family we can name. The control is the free pair $(z,w)$, which has no kernel and
returns the $t$-cores and nothing else at every $t\le8$ --- that is \eqref{eq:AK53} itself, so the
control is their theorem rather than a measurement of ours. So the
reciprocal pair is the collapsing specialization of \emph{smallest} kernel, and that, rather than
the mere existence of a kernel, is what makes its extra locus one family.
\end{problem}

\begin{problem}[the fibres, refined by the sign]\label{prob:fibres}
Corollary \ref{cor:fibregf} describes the fibre of the interval triple, and only that. The invariant
is $I_t=(d_1,d_2,d_3,\varepsilon_\lambda)$, and the sign cuts each branch further. Its cut is not
bookkeeping, and it is not idle: by Remark \ref{rem:signsees} the sign sees the $t-2$ directions the
triple does not, from $t=3$ on --- constant in $(m_i)$ at fixed $v$ in all $107$ slots at $t=2$, but
in only $43$ of $62$ at $t=3$ and $26$ of $36$ at $t=4$. What is the restriction of
$\varepsilon_\lambda$ to a branch? The obstruction to reading it off \eqref{eq:fibrelattice} is
visible in \eqref{eq:sign}: $\inv(b_S)$ is a statistic of the whole beta-word, and the lattice
coordinates split that word into two distinguished classes and $t-2$ free letters, so a rule would
have to reassemble what the bijection separates.

Two smaller gaps go with it. The same description for the size-three profile should be the same
bookkeeping with one class of three in place of two of two; we have not written it. The degenerate
profile is a different question, since there the fibre is the whole vanishing locus of $\Phi_t$,
which Corollary \ref{cor:zero} and Proposition \ref{rem:lattice} describe by other means.
\end{problem}

\begin{problem}[a sieving phenomenon that keeps a parameter]\label{prob:sieving}
At $t=2$ Corollary \ref{cor:enum} is a signed count of $\mathrm{SSYT}(\lambda,4)$, and a signed count
of order two is what Stembridge's $q=-1$ phenomenon \cite{Ste94} computes: on a set carrying an
involution, the value at $-1$ counts the fixed points. That is the case $\#C=2$ of the cyclic sieving
phenomenon of Reiner, Stanton and White \cite{RSW04}, whose Theorem 4.3 evaluates the principal
specialization $s_\lambda(1,q,\dots,q^{k-1})$ at a primitive $d$-th root of unity in terms of the
$d$-core and the $d$-quotient of $\lambda$ --- the two invariants that carry Theorem \ref{thm:main}
as well --- and which Lee and Oh \cite{LO22} have since extended to skew shapes. Every evaluation in
that line is at a \emph{number}, as a sieving statement requires: the alphabet there is the geometric
progression $1,q,\dots,q^{k-1}$ specialized at a root of unity, not a plain orbit, and nothing is
left free. Ours is not of that kind --- the reciprocal pair keeps $z$, and \eqref{eq:main} is a
product formula in it. Is there a cyclic action on $\mathrm{SSYT}(\lambda,t+2)$, and a $q$-analogue of
$\Phi_t$, whose value at a $t$-th root of unity counts fixed points refined by $m_{t+1}-m_{t+2}$?
The case $\#C=2$ is where to look first, and there the involution such a statement would need already
exists: \S\ref{sec:enum} assembles it.

Asked that way the question has a name. Sieving statements that keep an extra statistic are studied
in their own right --- Ahlbach and Swanson \cite{AS18} refine the cyclic sieving on words by fixing
the cyclic descent type alongside the major index --- and the two-variable form, in which a second
grading carries a second action, is the \emph{bicyclic} sieving of Barcelo, Reiner and Stanton
\cite{BRS08}. On semistandard tableaux in particular, Oh and Park \cite{OP19} build a cyclic action
out of the crystal structure whose sieving polynomial is $q^{-\kappa(\lambda)}
s_\lambda(1,q,\dots,q^{n-1})$ --- the frozen half of our alphabet, evaluated exactly as a sieving
statement requires --- and obtain a bicyclic statement for hooks; Alexandersson, Kantarc\i{}
O\u{g}uz and Linusson \cite{AKL21} do the same under promotion for several families of shapes. None
of these keeps a variable \emph{free}, which is what \eqref{eq:object} does and what makes the
question here a different one, but together they say what an answer should look like. The sharp
form is therefore: does each fixed $z$-weight fibre of $\mathrm{SSYT}(\lambda,t+2)$ carry a cyclic
sieving of its own?

The existence half of the question is decidable without exhibiting the action. Alexandersson and
Amini \cite{AA19} give a necessary and sufficient criterion for a cyclic action realizing a given
sieving polynomial to exist: for every $d\mid t$, the value at $\omega^d$ must be
$\sum_{j\mid d}j\,c_j$ with every $c_j$ a nonnegative integer, and the $c_j$ --- the numbers of
orbits of each size --- are then computable by triangularity. Applied to the candidate that the
alphabet itself suggests, $F(q,z)=s_\lambda(1,q,\dots,q^{t-1},z,z^{-1})$, whose $z$-grading is
already the refinement asked for, the criterion is met by almost every graded piece in the ranges we
can reach; but so it is for a deliberately wrong refinement, grading by $m_{t+1}+m_{t+2}$ instead,
which passes at the same rate. At those sizes the criterion does not discriminate, and we report the
attempt as evidence in neither direction.
\emph{One negative from the companion paper \cite{PaperII} belongs here, because it is easy to
mistake for an answer to this question.} What is tested and fails there is cyclic sieving in the
\emph{free reciprocal} parameter. The question above is the other one --- a cyclic action in the
frozen direction, refined by the free weight --- and nothing in that paper bears on it. We record the
distinction so that the two are not confused, which is the reason that paper records it too.
\end{problem}

\begin{problem}[the toggle on the lift]\label{prob:toggle}
The involution of \S\ref{sec:enum} is assembled at $t=2$ and nowhere else, and \S\ref{sec:lift} says
why: on $J_\lambda$ the requirement is arithmetically impossible for a third of the shapes in range,
and the set on which it is not is the lift $\widehat J_\lambda$ of Definition \ref{def:lift}. What
is missing is the involution itself. Stated so that the objects match: a weight- and sign-preserving
injection $\eta_\lambda:\mathcal L(\lambda)\hookrightarrow\widehat J_\lambda$, together with a
weight-preserving, sign-reversing involution $\iota$ on $\widehat J_\lambda$ whose fixed set is
exactly $\operatorname{im}\eta_\lambda$. For $\Phi_t\ne0$, Corollary \ref{cor:local} fixes the target
independently in each weight space, so it may be built one space at a time, and the vanishing case
has to be treated apart; by Remark \ref{rem:skeleton} its values are already forced on $398$ and $1176$ of
those spaces, and any solution has to agree with them there. The shape to try is a toggle on a
lattice-path model in which the choices of $\Omega_t$ are steps, which is the setting of
\cite{Fulmek}: there Cauchy--Binet and Laplace, the two sides to be matched, live in one picture.
What we do not have is a candidate rule. What we do have is that nothing numerical stands in its way,
which was not true before the lift.
\end{problem}

The distinction we keep throughout is between a \emph{conjecture}, which is a statement we believe
and have tested, and a \emph{problem}, which is not a claim at all.

Taken together the problems say where we think the object stops being ours. The evaluation is
complete and the alphabet is minimal, so what is left is not more of the same: it is whether the two
formalisms that own the two halves of \eqref{eq:object} compose, whether a sieving statement can keep
a variable, and what the coefficients $C_\mu$ do once the certificate we proposed for them is gone.
Of the three, the last is the one the paper needs and the first is the one we would rather see
answered.

\section{Closing: one alphabet, read three ways}\label{sec:closing}

The introduction called \eqref{eq:object} a torus element of $O(N)$ with a single free direction, and
said that the three results here are three readings of that sentence. They are easy to read as three
separate results instead, so we close by saying what separates them, and what is left after them.

\emph{The value.} Theorem \ref{thm:main} is what the free direction is worth: three factors over a
fixed denominator, for every $t$ and every $\lambda\in\PP_N$, with no hypothesis on the shape. The
two halves of the alphabet produce it together and neither produces it alone. The residue profile
supplies the orbit-side skeleton --- it settles the degenerate branch and names the two excess
classes --- and the reciprocal pair is what couples those classes into the interval triple, the order
of the beta word supplying the sign. Remove the pair and the statement degenerates to Littlewood's;
keep it and the profile still does not reach the triple, since already at $t=2$ the single two-class
profile carries ten different values of $d_3$ over $|\lambda|\le16$.

\emph{The invariant.} The same theorem read backwards says that the evaluation cannot see more of
$\lambda$ than an interval triple and a sign: two partitions of any sizes agreeing there take the
same value. On the two-class stratum Corollary \ref{cor:fibregf} writes down the generating function
of what the triple forgets; the size-three stratum and the degenerate one are Problem
\ref{prob:fibres}, and by Remark \ref{rem:signsees} the sign sees directions the triple does not.
This is the part we expect to outlast the formula, because it is a statement about the alphabet and
not about the proof.

\emph{The zero locus.} Past one pair the three-factor formula of Theorem \ref{thm:main} is gone ---
(D1) already gives a counterexample at $r=2$, and that no such formula exists for any $r\ge2$ is
Conjecture \ref{conj:rank2} rather than a theorem --- and the zero set is not. What the inversion half supplies there is a second mechanism, and
at $t=2$ a reading as well: by Theorem \ref{conj:crit} the character vanishes exactly when the beta set has
constant parity, which is a statement about residues, or the shape is self-complementary of odd
width, which is complementation; and
Corollary \ref{cor:dettwist} reads that criterion on the group as rigidity under the twist by
$\det$. The determinant $(-1)^{t+1}$ decides which component of $O(N)$ the alphabet lies in, and with
it which reading is available; it does not by itself decide which $\lambda$ vanish. That the two
mechanisms are genuinely two is what Corollary \ref{cor:oddgen} shows: every odd $t$ and every $r$
falls out of the extremal argument alone, with no external input and no appeal to complementation.

\emph{What remains.} The two reasons the introduction named do not stand equally at the end. The
orbit half is exhausted on this alphabet: Theorem \ref{thm:main} evaluates it in closed form for
every $t$, and
Theorem \ref{thm:extra} shows that within the two-row independence problem exactly one further
family behaves that way. The inversion half is not. It has been enough to say a good deal about
\emph{where} the character vanishes --- Theorem \ref{conj:crit} at $t=2$ for every number of pairs,
Theorem \ref{thm:suff} one implication for every $t$ and every $r$ --- but it has not told us what
the character \emph{is} there for more than one pair, and the other implication is Conjecture
\ref{conj:general}. The value never stops being a character value: it is $s_\lambda$ evaluated at a
torus element throughout. What ceases to hold past one pair is that the folded symplectic expansion
is always a genuine character; from two pairs on, properly virtual examples occur, and that is a
statement about the expansion and not about the object. What the paper hands over is smaller than
either half and more portable: a closed form carrying its sign, which makes this family a test case
for any implementation of twisted or universal characters; a vanishing test that costs two
subtractions on the beta set; and an enumeration in which the parameter is still free. A reader
looking for the next alphabet should look for the one that makes the second reason as complete as the
first has become here; on the evidence of \S\ref{sec:sharp} it will not be found by adding letters
to this one.

\subsection*{Acknowledgements}
This paper exists because Arvind Ayyer asked whether the rank-one case generalizes to arbitrary $t$
with the free variable symbolic. The question was better than the note that prompted it. We thank him
for it, and for the correspondence that followed. Remark \ref{rem:core} exists because Nishu Kumari
asked whether the condition $n_i(\lambda)=0$ could be phrased in terms of the $t$-core; she also
confirmed the reading of \cite[Theorem 5.3]{AK25} in \S\ref{sec:independence}, and it was a pointer
of hers that sent us to \cite[\S3]{AK25}, where Lemma \ref{lem:AtoSp} was waiting; we are grateful
for it. The involution of \S\ref{sec:enum} is assembled from three lemmas of her thesis
\cite{KumariThesis}, and is hers in that sense too.

A third debt is to a tool rather than to a conversation. The proof of Theorem \ref{thm:suff} turns
on being able to reflect a beta set inside a determinant without expanding anything, and the way to
do that --- scale each row by a power to send $x_i$ to $\bar x_i$, then reverse a block of columns
--- is the manoeuvre of Ayyer and Behrend \cite[\S3]{AB19}, used on the root-of-unity side by Albion
\cite{Alb23,Alb25}. We adopted it because it is plainly the right instrument for an inversion-closed
alphabet, and we kept it because it turned out to be more robust than the setting it was written
for: it never uses that the reflection is applied to the whole alphabet, and so it goes through for
a part of one. Theorem \ref{thm:suff}, and with it the reflection-based part of
\S\ref{sec:zeros}, rests on that manoeuvre; Corollary \ref{cor:oddgen} does not, and is
proved from the extremal argument alone.

\subsection*{Code, data and disclosure}\label{sec:disclosure}\label{sec:disclosure}
Every computation quoted in this paper is performed by a named script distributed as an ancillary
file, and every number quoted is reproduced by an archived run of one; \S\ref{sec:verif} tabulates
which script answers for which claim. The same scripts and the same archived output are also kept at
\url{https://github.com/karlesmarin/schur-orbit-and-reciprocal-pair}, which is a convenience: the
ancillary files carry everything the paper appeals to. Generative AI was used throughout as a
research assistant. The mathematics, and the responsibility for it, are the author's.


\end{document}